\documentclass[11pt]{article}

\usepackage[utf8]{inputenc} %
\usepackage[T1]{fontenc}    %
\usepackage{hyperref}       %
\usepackage{url}            %
\usepackage{booktabs}       %
\usepackage{amsfonts}       %
\usepackage{nicefrac}       %
\usepackage{microtype}      %
\usepackage{amsmath}
\usepackage{bm}
\usepackage{amssymb}
\usepackage{mathtools}
\usepackage{graphicx}
\usepackage{physics}
\usepackage{float}
\usepackage{yhmath}
\usepackage{algorithm,algorithmic}

\usepackage{tikz}

\usepackage{amsthm}

\usepackage{empheq}

\let\savedegree\div
\let\div\relax
\let\savering\ring
\let\ring\relax
\usepackage{mathabx} %
\let\div\savedegree
\let\ring\savering

\newcommand{\mR}{\mathbb{R}}

\newcommand{\E}[1]{\mathbb{E}\left[#1\right]}

\newcommand{\bracket}[1]{\left[ #1 \right]}
\newcommand{\parentheses}[1]{\left( #1 \right)}
\newcommand{\curlyp}[1]{\left\{ #1 \right\}}

\newcommand{\EE}{\mathbb{E}}
\newcommand{\RR}{\mathbb{R}}
\newcommand{\cL}{\mathcal{L}}

\newcommand{\mc}[1]{\mathcal{#1}}
\newcommand{\mbb}[1]{\mathbb{#1}}

\newtheorem{remark}{Remark}
\newtheorem{theorem}{Theorem}
\newtheorem{lemma}{Lemma}
\newcommand\myeq[2]{\stackrel{\mathclap{\normalfont\mbox{#1}}}{#2}}
\newcommand{\bmu}{\boldsymbol{\mu}}

\DeclareMathOperator*{\argmax}{arg\,max}
\DeclareMathOperator*{\argmin}{arg\,min}
\DeclareMathOperator*{\softmin}{soft-min}

\title{Unified Reinforcement Q-Learning for Mean Field Game and Control Problems}

\author{
Andrea Angiuli\thanks{Department of Statistics and Applied Probability, University of California Santa Barbara. {\tt angiuli@pstat.ucsb.edu}} 
\and 
Jean-Pierre Fouque\thanks{
Department of Statistics and Applied Probability,
South Hall 5504,
University of California
Santa Barbara, CA 93106
{\tt fouque@pstat.ucsb.edu}. Work  supported by NSF grant DMS-1814091.
}
\and Mathieu Lauri{\`e}re\thanks{Department of Operations Research and Financial Engineering. Princeton University. {\tt lauriere@princeton.edu}. Work supported by ARO grant AWD1005491 and NSF award AWD1005433.}}

\begin{document}
\maketitle

\begin{abstract}
We present a Reinforcement Learning (RL) algorithm to solve infinite horizon asymptotic Mean Field Game (MFG) and Mean Field Control (MFC) problems. Our approach can be described as a unified two-timescale Mean Field Q-learning: The \emph{same} algorithm can learn either the MFG or the MFC solution by simply tuning the ratio of two learning parameters. The algorithm is in discrete time and space where the agent not only provides an action to the environment but also a distribution of the state in order to take into account the mean field feature of the problem. Importantly, we assume that the agent can not observe the population's distribution and needs to estimate it in a model-free manner. The asymptotic MFG and MFC problems are also presented in continuous time and space, and compared with classical (non-asymptotic or stationary) MFG and MFC problems. They lead to explicit solutions in the linear-quadratic (LQ) case that are used as benchmarks for the results of our algorithm.
\end{abstract}

\tableofcontents

\section{Introduction}
Reinforcement learning (RL) is a branch of machine learning (ML) which studies the interactions of an agent within an environment in order to maximize a reward signal. RL algorithms solve Markov Decision Processes (MDP) based on trials and errors. At each discrete time $n$, the agent observes the state of the environment $X_n$ and chooses an action $A_n$. Due to the agent's action, the environment evolves to a state $X_{n+1}$ and assigns a reward $r_{n+1}$. The goal of the agent is to find the optimal strategy $\pi$ which assigns to each state of the environment the optimal action in order to maximize the aggregate discounted rewards. A complete overview on the evolution of this field is given in \cite{sutton2018reinforcement}.
The Q-learning method was introduced by \cite{watkins1989learning} to solve a discrete time MDP with finite state and action spaces. It is based on the evaluation of the optimal action-value table, $Q(x,a)$, which represents the expected aggregate discounted rewards when starting in state $x$ and choosing the first action $a$, i.e.
\begin{align}\label{intro: Q}
    Q^*(x,a)=\max_\pi\EE\left[\sum_{n=0}^\infty \gamma^{n} r_{n+1} \,\Big\vert\,  X_0=x,A_0=a\right],
\end{align}
where $r_{n+1}=r(X_n,\pi(X_n))$ is the instantaneous reward, $\gamma \in (0,1)$ is a discounting factor, and $X_{n+1} = b(X_n, \pi(X_n))$. The maximum is taken over strategies (or policies) $\pi$, which are functions of the state taking values in some action space. 
Since the state's dynamics $b$ (and sometimes the reward function $r$) are unknown to the agent, the algorithm is characterized by the trade-off between exploration of the environment and exploitation of the current available information. This is typically accomplished by the implementation of an $\epsilon$-greedy policy. The greedy action which maximizes the immediate reward is chosen with probability $1-\epsilon$ and a random action otherwise, i.e.
\begin{align}\label{intro: pi}
    \pi^\epsilon(x)=\left\{
    \begin{array}{lll}
       a\in Unif(A),  &\mbox{with probability}&\epsilon,  \\
        a^*=\argmax_{a\in A}Q(x,a), &\mbox{with probability} &1-\epsilon.
    \end{array}\right.
\end{align}
Note that this is the randomized policy which will be used in the algorithm presented in Section \ref{sec:algo}, but as the optimal strategies will turn out to be deterministic (as $\epsilon $ goes to zero over learning episodes), in the following, we present the problems and the $Q$-learning approach only using deterministic policies called controls and denoted by $\alpha$ instead of $\pi$ (see \cite{motte2019mean} for additional details on randomized policies).

On the other hand, and to summarize, mean field games are the result of the application of mean field techniques from physics into game theory. The mean field interaction is introduced to describe the behavior of a large number $N$ of indistinguishable players with symmetric interactions. The complexity of the system would be intractable if we were to describe all the pairwise interactions. A solution to this problem is given by describing the interactions of each player $i$ with the empirical distribution of the other players. As the number of players increases, the impact of each of them on the empirical distribution decreases. By the principle of propagation of chaos (law of large numbers) each player becomes asymptotically independent from the others and its interaction is with its own distribution making the statistical structure of the system simpler. Two types of mean field problems can be distinguished between a mean field game and a mean field control depending on the goal the agents try to achieve. The aim of a mean field game is to find an equivalent of a Nash equilibrium in a non-cooperative $N$-player game when the number of players becomes large. On the other hand, a mean field control problem analyzes the social optimum in a cooperative game within a  large population. Since the seminal works~\cite{MR2295621}, and \cite{MR2346927,MR2352434}, the research in mean field game theory attracted a huge interest. We refer to the extensive works~\cite{carmona2018probabilisticI-II}, and \cite{MR3134900} for further details. Connections between machine learning and mean field theory have been proposed in the recent literature. Some model-based methods have first been introduced in~\cite{fouque2019deep,carmona2019convergence-I,carmona2019convergence-II} by combining neural network approximation tools and stochastic gradient descent. Furthermore, model-free methods and links with reinforcement learning have also attracted a surge of interest. \cite{yang2018mean} analyzes the benefits that a mean field (local) interaction
brings in a multi-agent reinforcement learning (MARL) algorithm when the number of player is finite. \cite{yang2018deep} uses inverse reinforcement learning to learn the dynamics of a mean field game on a graph. \cite{guo2019learning} defines a simulator based Q-learning algorithm to solve a mean field game with finite state and action spaces. \cite{SubramanianMahajan-2018-RLstatioMFG} designs a gradient based algorithm to solve cooperative games (MFC) and a two-timescale approach to solve non-cooperative games (MFG) with finite state and action spaces, analogously to~\cite{mguni2018decentralised}. Convergence of actor-critic method for linear-quadratic MFG~\cite{fu2019actor} and convergence regularized Q-learning for MFG with finite state and action spaces~\cite{anahtarci2020q} have also been proved.  To learn MFC optima, model-free policy gradient methods have been proved to converge for LQ problems in~\cite{CarmonaLauriereTan-2019-LQMFRL}, whereas Q-learning for a ``lifted'' MDP on the space of distributions has been introduced in~\cite{CarmonaLauriereTan-2019-MQFRL}. To learn MFG equilibria, the fictitious play scheme has been introduced in~\cite{MR3608094}, assuming the best response can be computed exactly. \cite{elie2020convergence} analyses the propagation of error when the best response is computed approximately in a model-free setting, while~\cite{perrin2020continuousfp} extends the analysis of the fictitious play scheme in continuous time of learning. Similarly to our approach, \cite{xie2020provable} studies a single-loop fictitious play algorithm in which the state and the policy are updated at each iteration. Fictitious play combined with deep neural networks has also been used to compute Nash equilibria in multi-agent games \cite{Han-Hu-2020}.

In this paper, we propose a mean field Q-learning algorithm which is able to solve the mean field game or mean field control problem depending on the tuning of the parameters and the rate of update of the distribution. Differently from the approach developed by \cite{guo2019learning}, the algorithm does not require a simulator of the population simplifying its application to real world problems. It exploits the mean field limit transposing the interaction of the player with the population to the interaction of the player with herself.

In Section \ref{sec: problem} we formulate in discrete time and space the type of infinite horizon Asymptotic  MFG and MFC problems that our algorithm will address. Comparison with classical (non-asymptotic) and stationary problems are also made. In Section \ref{sec: 2scale},  we recast them as a two-timescale problem of Borkar's type \cite{borkar1997stochastic,MR2442439} which provides convergence results. The algorithm itself is presented in Section \ref{sec:algo}. In Section \ref{sec:results}, we show numerical results with comparison to the benchmark case of discrete time and space approximations for continuous time and space  linear-quadratic problems for which we have explicit formulas derived in Appendix \ref{appendix : calculations}.

\section{Mean Field Game and Mean Field Control Problems} \label{sec: problem}

We start by presenting three formulations of MFG and MFC problems: non-asymptotic, asymptotic, and stationary. All these problems are on an infinite horizon and for the sake {of consistency with the RL literature, we present them in a discrete time and space framework. We will however resort to continuous time and space models In Section \ref{sec:results} in order to obtain simple benchmarks.}  
Note  that, as customary in the MFG literature, without loss of generality, we minimize a cost instead of maximizing a reward.

{%

Let $\mathcal{X}$ and $\mathcal{A}$ be finite sets corresponding to states and actions. We denote by $\Delta^{|\mathcal{X}|}$ the simplex in dimension $|\mathcal{X}|$, which we identify with the space of probability measures on $\mathcal{X}$. Let $p: \mathcal{X} \times \mathcal{A} \times \Delta^{|\mathcal{X}|} \to \Delta^{|\mathcal{X}|}$ be a transition kernel. We will sometimes view it as a function:
$$
    p: \mathcal{X} \times \mathcal{X} \times \mathcal{A} \times \Delta^{|\mathcal{X}|}\to [0,1], \qquad
    (x, x', a, \mu) \mapsto p(x'|x,a,\mu), 
$$
which will be interpreted as the probability, at any given time step, to jump to state $x'$ when starting from state $x$ and using action $a$ and when the population distribution is $\mu$.

Let $f: \mathcal{X} \times \mathcal{A} \times \Delta^{|\mathcal{X}|} \to \RR$ be a running cost function. We interpret $f(x,a,\mu)$ as the one-step cost, at any given time step, incurred to a representative agent who is at state $x$ and uses action $a$ while the population distribution is $\mu$. 
For a random variable $X$, we denote its law by $\cL(X)$. We will focus on feedback controls, i.e., functions of the state of the agent and possibly of time. 
}

\subsection{Non-asymptotic formulations}

In the usual formulation for time-dependent MFG and MFC, the interactions between the players are through the  distribution of states at the current time. More precisely, in a MFG, one typically looks for $(\hat \alpha, \hat \bmu)$ where $\hat\alpha: \mathbb{N} \times \mathcal{X} \to \mathcal{A}$ and $\hat \bmu = (\hat\mu_n)_{n \ge 0} \in (\Delta^{|\mathcal{X}|})^{\mathbb{N}}$ is a flow of probability distributions on $\mathcal{X}$, such that the following two conditions hold: 
\begin{enumerate}
    \item Optimality of the best response map: $\hat\alpha$ is the minimizer of 
    {
$$
    \alpha \mapsto J^{MFG}(\alpha; \hat \bmu)
    =
    \EE\left[ \sum_{n=0}^{+\infty} \gamma^n f(X^{\alpha,\hat \bmu}_n, \alpha_n(X^{\alpha,\hat \bmu}_n), \hat \mu_n) \right],
$$
}
where $\alpha_n(\cdot) \coloneqq \alpha(n,\cdot)$ and the process $X^{\alpha,\hat \bmu}$ follows the dynamics given by:  %
{
$$
   X^{\alpha,\hat \bmu}_{n+1} \sim p \left( \cdot|  X^{\alpha,\hat \bmu}_n, \alpha_n(X^{\alpha,\hat \bmu}_n), \hat \mu_n \right) 
$$ 
}
with initial distribution $X^{\alpha,\hat \bmu}_0 \sim \mu_0$;
    \item Fixed point condition: {$\hat \mu_n =  \cL(X^{\hat\alpha,\hat \bmu}_n)$ for every $n \ge 0$.} 
\end{enumerate}

In a MFC problem, the goal is  to find $\alpha^*$ such that the following condition holds:  $\alpha^*$ is the minimizer of
{
$$
    \alpha \mapsto J^{MFC}(\alpha)
    =
    \EE\left[ \sum_{n=0}^{+\infty} \gamma^n f(X^{\alpha}_n, \alpha_n(X^{\alpha}_n), \cL(X^{\alpha}_n)) \right],
$$
}
where the process $X^{\alpha}$ %
{ follows the dynamics:
$$
   X^{\alpha}_{n+1} \sim p\left( \cdot | X^{\alpha}_n, \alpha_n(X^{\alpha}_n), \cL(X^{\alpha}_n) \right)
$$
}
with initial distribution $X^{\alpha}_0 \sim \mu_0$.  {Note that $p$ is the same transition probability function as for the MFG above but we plug the law $\cL(X^{\alpha}_n)$ of $X^{\alpha}_n$ instead of a given distribution $\hat \mu_n$. In other words, the MFC problem is of McKean-Vlasov (MKV) type. }

We will sometimes use the notation $\bmu^* = \bmu^{\alpha^*}$ for the optimal distribution in the MFC. Note that the objective function in the MFC setting can be written in terms of the objective function in the MFG as:
$$
    J^{MFC}(\alpha) = J^{MFG}(\alpha; \bmu^\alpha),
$$
where {$\mu^\alpha_n = \cL(X^\alpha_n)$ for all $n \ge 0$}.  However, in general,
$$
    J^{MFC}(\alpha^*)  = J^{MFG}(\alpha^*; \bmu^*) \neq J^{MFG}(\hat\alpha; \hat\bmu).
$$

In these two problems, the equilibrium control $\hat\alpha$ or the optimal control  $\alpha^*$ usually depend on time due to the dependence of $p$ and $f$ on the mean field flow, which evolves with time.

Although these are the usual formulations of MFG and MFC problems, in order to draw connections with reinforcement learning more directly, we turn our attention to formulations in which the control is independent of time. That is naturally the case in some applications, and, roughly speaking, it is also in the spirit of an individual player trying to optimally join a crowd of players already in the long-time asymptotic equilibrium. This will be made more precise in the following section.

\subsection{Asymptotic formulations}\label{sec:asymptotic}

We consider the following MFG problem: Find $(\hat \alpha, \hat \mu)$ where $\hat\alpha: \mathcal{X} \to \mathcal{A}$ and $\hat\mu \in \Delta^{|\mathcal{X}|}$, such that the following two conditions hold: 
\begin{enumerate}
    \item $\hat\alpha$ is the minimizer of 
{
$$
    \alpha \mapsto J^{AMFG}(\alpha; \hat \mu)
    =
    \EE\left[ \sum_{n=0}^{+\infty} \gamma^n f(X^{\alpha,\hat \mu}_n, \alpha(X^{\alpha,\hat \mu}_n), \hat \mu) \right],
$$
}
where the process $X^{\alpha,\hat \mu}$ follows the %
{ transitions:
$$
    X^{\alpha,\hat \mu}_{n+1} \sim p\left( \cdot | X^{\alpha,\hat \mu}_n, \alpha(X^{\alpha,\hat \mu}_n), \hat \mu \right) 
$$
}
with initial distribution $X^{\alpha,\hat \mu}_0 \sim \mu_0$;
    \item {$\hat \mu = \lim_{n \to +\infty} \cL(X^{\hat\alpha,\hat \mu}_n)$.}
\end{enumerate}

We stress that in this problem the control is a function of the state only and does not depend on time, as $b$ and $f$ depend only on the limiting distribution but not on time. Intuitively, this problem corresponds to the situation in which an infinitesimal player wants to join a crowd of players who are already in the asymptotic regime  (as time goes to infinity). 
This stationary distribution is a Nash equilibrium if the new player joining the crowd has no interest in deviating from this asymptotic behavior.

We also consider the following MFC problem: Find $\alpha^*$ such that the following condition holds:  $\alpha^*$ is the minimizer of 
{
$$
    \alpha \mapsto J^{AMFC}(\alpha)
    =
    \EE\left[ \sum_{n=0}^{+\infty} \gamma^n f(X^{\alpha}_n, \alpha(X^{\alpha}_n), \mu^\alpha) \right],
$$
}
where the process $X^{\alpha}$ %
{follows the transitions
$$
    X^{\alpha}_{n+1} \sim p\left( \cdot | X^{\alpha}_n, \alpha(X^{\alpha}_n),  \mu^\alpha \right) 
$$
}
with initial distribution $X^{\alpha}_0 \sim \mu_0$, and with the notation {$\mu^\alpha = \lim_{n \to +\infty} \cL(X^{\alpha}_n)$}. 

We will sometimes use the shorthand notation $\mu^* = \mu^{\alpha^*}$ for the optimal distribution in the MFC setting. Here too, the control is independent of time, and $p$ and $f$ depend only on the limiting distribution. Intuitively, this problem can be viewed as the one posed to a central planner who wants to find the optimal stationary distribution such that the cost for the society is minimal when a new agent joins the crowd.

Note that in this formulation again, the objective function in the MFC setting can be written in terms of the objective function in the MFG as:
$$
    J^{AMFC}(\alpha) = J^{AMFG}(\alpha; \mu^\alpha),
$$
with the notation  {$\mu^\alpha = \lim_{n \to +\infty} \cL(X^{\alpha}_n)$.}

\begin{remark}
    Although the AMFG and AMFC problems in this section are defined using an initial distribution $\mu_0$ for the state process, one expects that under suitable conditions, {\it ergodicity} in particular, the optimal controls $\hat\alpha$ and $\alpha^*$ are independent of this initial distribution.
\end{remark}

\subsection{Stationary formulations}

Another formulation with controls independent of time  consists in looking at the situation in which the new agent joining the crowd starts with a position drawn according to the ergodic distribution of the equilibrium control or the optimal control. This type of problems has been considered e.g. in 
\cite{guo2019learning}, \cite{SubramanianMahajan-2018-RLstatioMFG}, and can be described as follows.

The stationary MFG problem is to find $(\hat \alpha, \hat \mu)$ where $\hat\alpha: \mathcal{X} \to \mathcal{A}$ and $\hat\mu \in \Delta^{|\mathcal{X}|}$, such that the following two conditions hold: 
\begin{enumerate}
    \item  $\hat\alpha$ is the minimizer of 
{
$$
    \alpha \mapsto J^{SMFG}(\alpha; \hat \mu)
    =
     \EE\left[ \sum_{n=0}^{+\infty} \gamma^n f(X^{\alpha,\hat \mu}_n, \alpha(X^{\alpha,\hat \mu}_n), \hat \mu) \right],
$$
}
where the process $X^{\alpha,\hat \mu}$ follows the SDE
{
$$
    X^{\alpha,\hat \mu}_{n+1} 
    \sim p\left( \cdot| X^{\alpha,\hat \mu}_n, \alpha(X^{\alpha,\hat \mu}_n), \hat \mu \right),
$$ 
}
and starts with distribution $X^{\alpha,\hat \mu}_0 \sim \hat\mu$;
    \item The process $X^{\hat\alpha,\hat \mu}$ admits $\hat \mu$ as invariant distribution {(so $\hat \mu = \cL(X^{\hat\alpha,\hat \mu}_n)$ for all $n \ge 0$)}.
\end{enumerate}

The key difference with the Asymptotic MFG formulation is that here the process starts with the invariant distribution $\hat\mu$. The control is a function of the state only and does not depend of time, and $p$ and $f$ depend only on this stationary distribution.

The stationary MFC problem is defined as follows: Find $\alpha^*$ such that the following condition holds:  $\alpha^*$ is the minimizer of 
{
$$
    \alpha \mapsto J^{SMFC}(\alpha)
    =
    \EE\left[ \sum_{n=0}^{+\infty} \gamma^n f(X^{\alpha}_n, \alpha(X^{\alpha}_n), \mu^\alpha) \right],
$$
}
where the process $X^{\alpha}$ %
{follows the MKV dynamics
$$
    X^{\alpha}_{n+1} \sim p\left( \cdot | X^{\alpha}_n, \alpha(X^{\alpha}_n), \mu^\alpha \right),
$$
}
with initial distribution $X^{\alpha}_0 \sim \mu^\alpha$, and such that $\mu^\alpha$ is the invariant distribution of $X^{\alpha}$ (assuming it exists). 

To conclude, let us mention that there is yet another formulation, in which the solution is stationary but depends on the initial distribution, see~\cite[Chapter 7]{MR3134900}.

\subsection{Connecting the three formulations}

Denoting by $\hat\alpha^{MFG}, \hat\alpha^{AMFG}$, and $\hat\alpha^{SMFG}$, the MFG equilibrium strategies respectively in the non-asymptotic, asymptotic, and stationary formulations, we expect
{
\begin{equation}\label{controls-MFG}
\left\{
\begin{split}
    \hat\alpha^{MFG}_n(x) &\to \hat\alpha^{AMFG}(x), \qquad \forall x, \qquad \hbox {as} \qquad n \to +\infty,
    \\
    \hat\alpha^{AMFG}(x) &= \hat\alpha^{SMFG}(x), \qquad \forall x.
\end{split}
\right.
\end{equation}
}

Similarly denoting by $\alpha^{*MFC}, \alpha^{*AMFC}$, and $\alpha^{*SMFC}$, the MFC optimal controls respectively in the non-asymptotic, asymptotic, and stationary formulations, we expect
{
\begin{equation}\label{controls-MFC}
\left\{
\begin{split}
    \alpha^{*MFC}_n(x) &\to \alpha^{*AMFC} (x), \qquad \forall x, \qquad \hbox {as} \qquad n \to +\infty,
    \\
    \alpha^{*AMFC} (x) &= \alpha^{*SMFC}(x), \qquad \forall x.
\end{split}
\right.
\end{equation}
}
In fact, we have  the following result.

\begin{theorem}\label{th:controls-A-S}
Consider the set of admissible controls to be defined as the set of controls $\alpha$ such that the process $(X_n^{\alpha})_{n \geq 0}$ is an irreducible and aperiodic Markov process on the finite space {\cal X}. If a solution for the asymptotic MFG (resp. MFC) exists, then it is equal to the solution of the corresponding stationary MFG (resp. MFC) and vice versa.
\end{theorem}   

\begin{proof}
 Let us consider the pair $(\hat\alpha^{AMFG}, \hat\mu^{AMFG})$ solution of  an asymptotic MFG. The optimal control $\hat\alpha^{AMFG}$ is an optimizer over the set of admissible controls such that the process $(X_n^{\alpha})_{n \geq 0}$ is an irreducible Markov process and admits a limiting distribution which is then the unique invariant distribution using the control $\hat\alpha^{AMFG}$. 
 Note that the control $\hat\alpha^{AMFG}$ doesn't depend on the initial distribution $\mu_0$ and consequently $\hat\mu^{AMFG}$ doesn't either. Therefore, $(\hat\alpha^{AMFG}, \hat\mu^{AMFG})$ is the solution of the AMFG starting from $\hat\mu^{AMFG}$, which is the corresponding stationary MFG problem. Thus, we deduce the desired relation $\hat\alpha^{AMFG}=\hat\alpha^{SMFG}$.
A similar argument for MFC problems applies and shows that 
$\alpha^{*AMFC}=\alpha^{*SMFC}$.
\end{proof}

\begin{remark}\label{remark:formulation_choice}
In terms of practical applications, the asymptotic formulation (AMFG and AMFC) seems to be the most appropriate, and if one is interested in the optimal controls, Theorem \ref{th:controls-A-S} shows that solving the asymptotic games also gives the solutions to the corresponding stationary games. Additionally, \eqref{controls-MFG} and \eqref{controls-MFC} indicate that it also gives the long time solutions to the corresponding time-dependent games. Developing Q-learning algorithms for solving time-dependent finite horizon games is addressed in our forthcoming paper \cite{AngiuliFouqueLauriere-Handbook2021}.
\end{remark}

In Appendix \ref{appendix : calculations}, we provide explicit solutions for MFG, AMFG, SMFG, MFC, AMFC, and SMFC, in the case of continuous time, continuous space Linear-Quadratic stochastic differential games. We verify that  \eqref{controls-MFG} and \eqref{controls-MFC}, and therefore, Theorem \ref{th:controls-A-S}, are satisfied in that case as well. In Section \ref{sec:results}, discrete approximations of these games will also serve as benchmarks for our algorithm described in Section \ref{sec:algo}.

\section{A unified view of learning for MFG and MFC} \label{sec: 2scale}

In this section we draw a connection between MFG, MFC, Q-learning and Borkar's two timescale approach~\cite{borkar1997stochastic,MR2442439}. \newline The definitions of MFG and MFC reveal that the two formulations are very similar and both involve an optimization and a distribution.  This leads to the idea of designing an iterative procedure which would update the value function and the distribution. However, in the MFG, the distribution is frozen during the optimization and then a fixed point condition is enforced, whereas in the MFC problem the distribution is directly linked to the control, which implies that it should change instantaneously when the control function is modified. Hence, to compute the solutions using an iterative algorithm, the updates should be done differently for each problem: intuitively, in a MFG, the value function should be updated in an inner loop and the distribution in an outer loop, whereas it should be the converse for MFC. More generally, we can update both functions in turn but at different rates. Then, to compute the MFG solution, the distribution should be updated at a lower rate than the value function. For MFC, it should be the converse. In the rest of this subsection, we formalize these ideas.

\subsection{Action-value function in the classical Q-learning setup}\label{sec:Q_classical_MDP}

One of the most popular methods in RL is the so-called Q-learning~\cite{watkins1989learning}. Instead of looking at the value function $V$ as in a PDE approach for optimal control, this method is based on the action-value function, also called $Q$-function, which takes as inputs not only a state $x$ but also an action $a$. Intuitively, in a standard (non mean-field) MDP, this function quantifies the optimal cost-to-go of an agent starting at $x$, using action $a$ for the first step and then acting optimally afterwards. In other words, the value of $(x,a)$ is the the cost of using $a$ when in state $x$, plus the minimal cost possible after that, i.e. the cost induced by using the optimal control; 
see e.g.~\cite[Chapter 3]{sutton2018reinforcement} for more details. 
The definition of the optimal $Q$-function, denoted by $Q^*$, is similar to~\eqref{intro: Q}, up to a change of sign since we consider a cost $f$ and a minimization problem instead of a reward $r$ and a maximisation problem, namely,
 \begin{align*}%
    Q^*(x,a)=\min_\alpha\EE\left[\sum_{n=0}^\infty \gamma^n f(X_{n},\alpha(X_{n})) \,\Big\vert\, X_{0}=x,A_{0}=a\right].
\end{align*}
Using dynamic programming, it can be shown that $Q^*$ is the solution of the Bellman equation:  
$$
    Q^*(x,a) = f(x, a) +\gamma \sum_{x' \in \mathcal{X}} p(x' | x, a) \min_{a'} Q^*(x',a'), \qquad (x,a) \in \mathcal{X} \times \mathcal{A}. 
$$
The corresponding value function $V^*$ is given by:
$$
    V^*(x) = \min_a Q^*(x,a), \qquad x \in \mathcal{X}.
$$
One of the main advantages of computing the optimal action-value function instead of the value function is that from the former, one can directly recover the optimal control, given by $\argmin_{a \in \mathcal{A}} Q^*(x,a)$. This is particularly important in order to design model-free methods, as we will see in the next section.

\subsection{Action-value function for Asymptotic MFG}\label{subsec: Q-A-MFG}

In the context of Asymptotic MFG introduced in Section \ref{sec:asymptotic}, we can view the problem faced by an infinitesimal agent among the crowd as an MDP \emph{parameterized} by the population distribution. Hence, given a population distribution $\mu$, standard RL techniques can be applied to compute the $Q$-function of an infinitesimal agent against this given $\mu$.

Then, the optimal $Q$-function is defined, for a given $\mu$, by
 \begin{align}%
 \label{eq:optimal-Q-fct}
    Q^*_\mu(x,a)=\min_\alpha\EE\left[\sum_{n=0}^\infty \gamma^n f(X_{n},\alpha(X_{n}), \mu) \,\Big\vert\, X_{0}=x,A_{0}=a\right],
\end{align}
where the cost function $f(x,a,\mu)$ depends on the fixed $\mu$ as well as the transition probabilities $p(x' | x, a, \mu)$. Since $\mu$ is fixed, as in the classical case, one obtains the
the Bellman equation: 
\begin{equation}
\label{eq:def-Q-functionMFG}
    Q^*_{\mu}(x,a) = f(x, a, \mu) + \gamma\sum_{x' \in \mathcal{X}} p(x' | x, a, \mu) \min_{a'} Q^*_\mu(x',a'), \qquad (x,a) \in \mathcal{X} \times \mathcal{A}. 
\end{equation}
This function characterizes the optimal cost-to-go for an agent starting at state $x$, using action $a$ for the first step, and then acting optimally for the rest of the time steps, while the population distribution is given by $\mu$ (for every time step). Note that $\min_{a}Q^*_{\mu}(x,a)=
\min_\alpha J^{AMFG}(\alpha; \mu)$ in the notation of Section \ref{sec:asymptotic}.

\subsection{Action-value function for Asymptotic MFC}\label{subsec: Q-A-MFC}

For MFC, it is not obvious how to use the same $Q$-function because, as noticed earlier, the distribution appearing in the definition of MFC is directly linked to the control and not fixed a priori. One possibility is to look at MFC as an MDP on the space of distributions and then to introduce a $Q$-function which takes a distribution as an input~\cite{CarmonaLauriereTan-2019-MQFRL,gu2019dynamic,gu2020qlearning,motte2019mean}. 

We take a different route and consider a modified Q- function as follows. For an admissible control $\alpha(x)$, we define the MKV- dynamics $p(x' | x, a, \mu^\alpha)$ so that $\mu^\alpha$ is the limiting distribution of the associated process $(X^\alpha _n)$. We define the control $\tilde\alpha$ by
\begin{align}\label{def:tildealpha}
    \tilde\alpha (x')=\left\{
    \begin{array}{lll}
      a &\mbox{if} &x'=x, \\
        \alpha(x) &\mbox{for} &x'\neq x.
    \end{array}\right.
\end{align}
Note that $\tilde\alpha$ depends on $x$ and $a$.
Our modified $Q$-function is given by
\begin{align*}%
    Q^\alpha(x,a)=f(x,a, \mu^{\tilde\alpha})+\EE\left[\sum_{n=1}^\infty \gamma^n f(X_{n},\alpha(X_{n}), \mu^\alpha) \,\Big\vert\, X_{0}=x,A_{0}=a\right].
\end{align*}
We then obtain that the optimal $Q^*(x,a)=\min_\alpha Q^\alpha(x,a)$ satisfies the Bellman equation
\begin{equation}
\label{eq:def-Q-functionMFC}
    Q^*(x,a) = f(x, a, \tilde{\mu}^*) + \gamma\sum_{x' \in \mathcal{X}} p(x' | x, a, \tilde{\mu}^*) \min_{a'} Q^*(x',a'), \qquad (x,a) \in \mathcal{X} \times \mathcal{A}, \end{equation}
where the optimal control $\alpha^*$ is given by   $\alpha^*(x)=\argmin_a Q^*(x,a)$, the control $\tilde{\alpha}^*$ is defined by \eqref{def:tildealpha} for $x$ and $a$, and 
$\tilde{\mu}^*:=\mu^{\tilde{\alpha}^*}$.
The optimal value function is $V^*(x)=\min_a Q^*(x,a)$ ($=J^{AMFC}(\alpha^*)$ in the notation of Section \ref{sec:asymptotic}).
The details of the derivation of these equations are given in Appendix \ref{appendix : MFC-RL}.

Note that, compared with the $Q_\mu$-function used for MFG, our MFC modified $Q$-function
involves the differences $\Delta_\mu f:=f(x,a,\tilde\mu)-f(x,a,\mu)$ and $\Delta_\mu p:=p(\cdot | x,a,\tilde\mu)-p(\cdot | x,a,\mu)$ which play the role of derivatives with respect to the probability distribution in the classical continuous time and space Mean Field Control problems.

\subsection{Unification through a two timescale approach}\label{sec: unified 2scale approach}

The goal is now to design a learning procedure which can approximate, for either MFG or MFC, not only $Q$ but also the corresponding $\mu$. For MFG, the usual fixed point iterations are on the distribution and at each iteration, the best response against this distribution (which can be deduced from the corresponding $Q$ table) is computed. For MFC, the iterations are on the control (here again, it can be deduced from the $Q$ table) and the distribution corresponding to this control is computed at each iteration. Instead of completely freezing the distribution (resp. the control) in the first case (resp. the second case), we can imagine that letting it evolve at a slow rate would still lead to the same limit. In other words, the definitions of MFG and MFC seem to lie at the two opposite sides of a spectrum. 

Based on this viewpoint, we consider the following iterative procedure, where both variables ($Q$ and $\mu$) are updated at each iteration but with different rates.  Starting from an initial guess $(Q_0, \mu_0) \in \RR^{|\mathcal{X}| \times |\mathcal{A}|} \times \Delta^{|\mathcal{X}|}$, define iteratively for $k=0,1,\dots$:

\begin{subequations}%
     \begin{empheq}[left=\empheqlbrace]{align}
      \label{eq1:2scale-mu-k}
    \mu_{k+1} &= \mu_{k} + \rho_{k}^\mu \mathcal{P}(Q_{k}, \mu_{k}),
    \\
     \label{eq1:2scale-Q-k}
    Q_{k+1} &= Q_{k} + \rho_{k}^Q \mathcal{T}(Q_{k}, \mu_{k}),
    \end{empheq}
\end{subequations}

where 
$$
\begin{cases}
 \mathcal{P}(Q, \mu)(x) = (\mu P^{Q,\mu})(x) - \mu(x), \qquad x \in \mathcal{X},\\
    \mathcal{T}(Q, \mu)(x,a) = f(x, a, \mu) + \gamma \sum_{x'} p(x' | x,a,\mu) \min_{a'}Q(x',a') - Q(x,a), \qquad (x,a) \in \mathcal{X} \times \mathcal{A}, 
\end{cases}
$$
and 
$$
    P^{Q,\mu}(x, x') = p(x' | x, \argmin_{a} Q(x, a), \mu), 
    \qquad \hbox{} 
    \qquad
    (\mu P^{Q,\mu})(x) = \sum_{x_0} \mu(x_0) P^{Q,\mu}(x_0,x),
$$
$ P^{Q,\mu}$ is the transition matrix when the population distribution is $\mu$ and the agent uses the optimal control according to $Q$.   The learning rates $\rho_{k}^\mu$ and $\rho_{k}^Q$ are assumed to satisfy usual Robbins-Monro type conditions, namely: $\sum_{k} \rho_{k}^\mu = \sum_{k} \rho_{k}^Q = +\infty$ and $\sum_{k} |\rho_{k}^\mu|^2 = \sum_{k} |\rho_{k}^Q|^2 < + \infty$.

If $\rho_{k}^\mu < \rho_{k}^Q$, the approximate $Q$-function evolves faster, while it is the converse if $\rho_{k}^\mu > \rho_{k}^Q$. This suggests that these two regimes should converge to different limit points.  These ideas have been studied by Borkar~\cite{borkar1997stochastic,MR2442439} in connection with reinforcement learning methods under the name of two timescales approach. More precisely, from Borkar~\cite[Chapter 6, Theorem 2]{MR2442439}, we expect to have the following two situations. We assume that the operators $\mathcal{T}$ and $\mathcal{P}$ are Lipschitz continuous, which, as explained in Appendix~\ref{appendix : borkar_assumptions}, can be obtained from the Lipschitz continuity of $f$ and $p$ in the model, as well as a slight modification of $\mathcal{P}$ to regularize the minimizer.

\begin{itemize}
    \item {\bf Two timescale approach for MFG}. 
    
    If $\rho_{k}^\mu/\rho_{k}^Q \to 0$ as $k\to+\infty$, the system \eqref{eq1:2scale-mu-k}--\eqref{eq1:2scale-Q-k} tracks the ODE system
    \begin{subequations}%
     \begin{empheq}[left=\empheqlbrace]{align*}
      \dot \mu_t &= \mathcal{P}(Q_t, \mu_t),\\
    \dot Q_t &= \frac{1}{\epsilon} \mathcal{T}(Q_t, \mu_t),
    \end{empheq}
    \end{subequations}
    where $\rho_{k}^\mu /\rho_{k}^Q$  is thought of being of order $\epsilon\ll 1$.
    We consider, for any fixed $\mu$, the ODE 
    $$
    \dot Q_t = \frac{1}{\epsilon} \mathcal{T}(Q_t,  \mu),
    $$ 
    and we assume it has a globally asymptotically stable equilibrium $Q_{ \mu}$. 
    In particular, $\mathcal{T}(Q_{ \mu},  \mu) = 0$, meaning by \eqref{eq:def-Q-functionMFG} that $Q_\mu$ is the value function of an infinitesimal agent facing the crowd distribution $ \mu$. We further assume that $Q_\mu$ is Lipschitz continuous with respect to $\mu$. Convergence to $Q_\mu$ can be obtained following standard arguments for Q-learning (see, \textit{e.g.}, \cite[Section 10.3]{MR2442439}) and the Lipschitz continuity of $Q_\mu$ can be guaranteed through Lipschitz continuity of $f, p$ and the minimizer  in~\eqref{eq:optimal-Q-fct}.  Then the first ODE becomes 
    $$
        \dot \mu_t = {\mathcal{P}}(Q_{\mu_t}, \mu_t).
    $$
     Assuming it has a globally asymptotically stable equilibrium  $\mu_{\infty}$, this distribution satisfies
    $$
        {\mathcal{P}}(Q_{\mu_\infty}, \mu_\infty) = 0.
    $$
    This condition implies that $\mu_\infty$ and the associated control given by $\hat\alpha(x) = \argmin_{a} Q_{\mu_\infty}(x, a)$ form a Nash equilibrium.  From~\cite[Chapter 6, Theorem 2]{MR2442439}, the system \eqref{eq1:2scale-mu-k}--\eqref{eq1:2scale-Q-k} with discrete time updates also converges to this Nash equilibrium when $\rho_{k}^\mu/\rho_{k}^Q \to 0$ as $k\to+\infty$. 
    \vskip .4cm
    \item {\bf Two timescale approach for MFC}.
    
    If $\rho_{k}^Q/\rho_{k}^\mu \to 0$ as $k\to+\infty$, the system \eqref{eq1:2scale-mu-k}--\eqref{eq1:2scale-Q-k} tracks the ODE system
    \begin{subequations}%
     \begin{empheq}[left=\empheqlbrace]{align*}
      \dot \mu_t &= \frac{1}{\epsilon}  \mathcal{P}(Q_t, \mu_t),
    \\
    \dot Q_t &= {\mathcal{T}}(Q_t, \mu_t),
   \end{empheq}
    \end{subequations}
     where $\rho_{k}^Q /\rho_{k}^\mu$  is thought of being of order $\epsilon\ll 1$.
    We consider, for any fixed $Q$, the ODE
    $$
        \dot \mu_t = \frac{1}{\epsilon}  \mathcal{P}(Q, \mu_t),
    $$
    and we assume it has a globally asymptotically stable equilibrium $\mu_{Q}$. 
    In particular, $\mathcal{P}(Q,  \mu_Q) = 0$, meaning that $\mu_{Q}$ is the asymptotic distribution of a population in which every agent uses the control $ \alpha(x) = \argmin_{a} Q(x,a)$. We further assume that $\mu_{Q}$ is Lipschitz continuous with respect to $Q$. 
    Then the second ODE becomes 
    $$
       \dot Q_t(x,a) = {\mathcal{T}}(Q_t(x,a), \widetilde{\mu}_{Q_t}),
    $$
    where $\widetilde{\mu}_{Q_t}$ is defined by \eqref{def:tildealpha} at $(x,a)$ for $\alpha(\cdot)= \argmin_{a'}Q_t(\cdot,a')$. This is consistent with the update of $Q$ and what the algorithm proposed in Section \ref{sec:algo} does.
     Assuming this ODE has a globally asymptotically stable equilibrium  $Q_{\infty}$, this $Q$-table satisfies
    $$
        {\mathcal{T}}(Q_\infty, \widetilde{\mu}_{Q_\infty}) = 0.
    $$
    This last condition means that $Q_\infty=Q^*$ satisfies the MFC Bellman equation \eqref{eq:def-Q-functionMFC}, and that the control $\alpha^*(x) = \argmin_{a} Q_\infty(x, a)$ is an MFC optimum for the asymptotic formulation and the induced optimal distribution is $\mu_{Q_\infty}$.   From~\cite[Chapter 6, Theorem 2]{MR2442439}, the system \eqref{eq1:2scale-mu-k}--\eqref{eq1:2scale-Q-k} with discrete time updates also converges to this social optimum when $\rho_{k}^Q/\rho_{k}^\mu \to 0$ as $k\to+\infty$. 
\end{itemize}

\subsection{Stochastic approximation}

The above (deterministic) algorithm relies on the operators $\mathcal{P}$, $\mathcal{T}$  which, in many practical situations are not known, for instance because the agent does not know for sure the dynamics or the reward function. In such situations, the agent can only rely on random samples (more details are provided in the next section). The algorithm can be modified to account for such stochastic approximations. Indeed, let us assume that, for any $Q,\mu,x,a$, the agent can know the value $f(x, a, \mu)$ and can sample a realization of the random variable
$$
    X'_{x,a,\mu} \sim p(\cdot | x,a,\mu).
$$
Then, she can compute the realization of the following random variables $\widecheck{\mathcal{T}}_{Q,\mu,x,a}$ and $\widecheck{\mathcal{P}}_{Q,\mu,x,a}$ taking values respectively in $\RR$ and $\Delta^{|\mathcal{X}|}$:
$$
    \widecheck{\mathcal{T}}_{Q,\mu,x,a} = f(x, a, \mu) + \gamma  \min_{a'}Q(X'_{x,a,\mu},a') - Q(x,a),
$$
and
$$
    \widecheck{\mathcal{P}}_{Q,\mu,x,a}(x'') =  \mathbf{1}_{\{X'_{x,a,\mu}=x''\}} - \mu(x''), \qquad \forall x'' \in \mathcal{X}. 
$$
Observe that
\begin{equation}
\label{eq:link-checkT-calT}
    \EE[\widecheck{\mathcal{T}}_{Q,\mu,x,a}] 
    = \sum_{x'} p(x' | x,a,\mu) 
    \left[ f(x, a, \mu) + \gamma  \min_{a'}Q(x',a') - Q(x,a) \right]
    = \mathcal{T}(Q, \mu)(x,a),
\end{equation}

and
$$
    \EE[\widecheck{\mathcal{P}}_{Q,\mu,x,a}(x'') ]
    = \sum_{x'} p(x' | x,a,\mu) 
    \left(\mathbf{1}_{\{x'=x''\}}- \mu(x'')\right)
    = p(x'' | x,a,\mu)- \mu(x''). 
$$
If the starting point $x$ comes from a random variable $X \sim \mu$ and if $a$ is chosen to be an optimal action at $X$ according to a given table $Q$, i.e., $a \in \argmin_{\mathcal A} Q(X,\cdot)$, then we obtain
\begin{align}
    \notag
    \EE[\widecheck{\mathcal{P}}_{Q,\mu,X,\argmin_a Q(X,a)}(x'') ]
    &= \sum_x \mu(x) \sum_{x'} p(x' | x,\argmin_a Q(x,a),\mu) 
    \left(\mathbf{1}_{\{x'=x''\}}- \mu(x'')\right)
    \\
    \notag
    &= \sum_x \mu(x) \left(p(x'' | x,\argmin_a Q(x,a),\mu)- \mu(x'')\right)
    \\
    \notag
    &= (\mu P^{Q,\mu})(x'')- \mu(x'')\\
    &= \mathcal{P}(Q, \mu)(x'').
\label{eq:link-checkP-calP}
\end{align}

We can thus replace the deterministic updates~\eqref{eq1:2scale-mu-k}--\eqref{eq1:2scale-Q-k} by the following stochastic ones, starting from some initial $Q_0, \mu_0$: for $k=0, 1, \dots$, 
\begin{subequations}%
     \begin{empheq}[left=\empheqlbrace]{align}
      \label{eq:2scale-mu-k-sto}
    \mu_{k+1}(x) 
    &= \mu_{k}(x) + \rho_{k}^\mu \widecheck{\mathcal{P}}_{Q_{k}, \mu_{k},X_k,\argmin_a Q(X_k,a)}(x)
    \\
    \notag&= \mu_{k}(x) + \rho_{k}^\mu \mathcal{P}(Q_{k}, \mu_{k})(x) + \mathbf{P}^k(x), \qquad \forall x \in \mathcal{X}
    \\
     \label{eq:2scale-Q-k-sto}
    Q_{k+1}(x,a)
    &= Q_{k}(x,a) + \rho_{k}^Q \widecheck{\mathcal{T}}_{Q_{k}, \mu_{k},x,a}
    \\
    \notag&= Q_{k} + \rho_{k}^Q \mathcal{T}(Q_{k}, \mu_{k})(x,a) + \mathbf{T}^k(x,a), \qquad \forall (x,a) \in \mathcal{X} \times \mathcal{A},
    \\
    \notag
    X_k &\sim \mu_k,
     \end{empheq}
\end{subequations}
where we introduced the notation: 
$$
    \mathbf{P}^k(x) = \rho_{k}^\mu \Big( \widecheck{\mathcal{P}}_{Q_{k}, \mu_{k},X_k,\argmin_a Q_{k}(X_k, a)}(x) - \mathcal{P}(Q_{k}, \mu_{k})(x)\Big), \qquad \forall x,
$$
and
$$
    \mathbf{T}^k(x,a)  = \rho_{k}^Q \Big( \widecheck{\mathcal{T}}_{Q_{k}, \mu_{k},x,a} - \mathcal{T}(Q_{k}, \mu_{k})(x,a)\Big), \qquad \forall (x,a),
$$
with $X_k$ sampled from $\mu_k$.
Note that $\mathbf{T}^k$ and $\mathbf{P}^k$ are martingales by the above remarks, see~\eqref{eq:link-checkT-calT}--\eqref{eq:link-checkP-calP}. Hence under suitable conditions, we expect convergence to hold by classical stochastic approximation results~\cite{MR2442439}.

However, the procedure~\eqref{eq:2scale-mu-k-sto}--\eqref{eq:2scale-Q-k-sto} is synchronous (it updates all the coefficients of the Q-table and the distribution at each iteration $k$) and it requires having access to a generative model, i.e., to a simulator which can provide samples of transitions drawn according to $p(\cdot| x,a,\mu_k )$ for arbitrary state $x$. In the next section, we propose a procedure which works even with a more restricted setting, which uses episodes: In each episode, the learner is constrained to follow the trajectory sampled by the environment without choosing arbitrarily its state.

\section{Reinforcement Learning Algorithm}\label{sec:algo}

As recalled in the Introduction, RL studies the algorithms to solve a Markov decision process (MDP) based on trials and errors. An MDP can be described through the interactions of an agent with an environment. At each time $n$, the agent observes its current state $X_{n} \in \mathcal{X}$ and chooses an action $A_{n} \in \mathcal{A}$. Due to the agent's action, the environment provides the new state of the agent $X_{n+1}$ and incurs a cost $f_{n+1}$. The goal of the agent is to find an optimal strategy (or policy) $\pi^*$ which assigns to each state an action in order to minimize the aggregated discounted costs. The idea is then to design methods which allow the agent to learn (an approximation of) $\pi^*$ by making repeated use of the environment's outputs but without knowing how the environment produces the new state and the associated cost. A detailed overview of this field can be found in 
 \cite{sutton2018reinforcement} (although RL methods are often presented with reward maximization objectives, we consider cost minimization problems for the sake of consistency with the MFG literature).

  As presented in Section \ref{sec:Q_classical_MDP}, the optimal strategy can be derived from the optimal action-value function. However $Q^*$ is a priori unknown. In order to learn $Q^*$ by trials and errors, an approximate version $Q$ of the table $Q^*$ is constructed through an iterative procedure. At each step, an action is taken, which leads to a cost and to a new state. On the one hand, it is interesting to act efficiently in order to avoid high costs, and on the other hand it is important to improve the quality of the table $Q$ by trying actions and states which have not been visited many times so far. This is the so-called exploitation--exploration trade-off. The trade-off between exploration of the unknown environment and exploitation of the currently available information can be taken care of by an $\epsilon$-greedy policy based on $Q$. The algorithm chooses the action that minimizes the immediate cost with probability $1-\epsilon$, and a random action otherwise, as in (\ref{intro: pi}) with an $\argmin$.

\subsection{U2-MF-QL : Unified Two Timescales Mean Field Q-learning }

In order to apply the RL paradigm to  mean field problems, the first step consists in defining the connection between these two frameworks.  
In a MFG (resp. a MFC) the goal of a typical agent is to find the pair $(\hat\alpha,\hat\mu)$ (resp. $(\alpha^*,\mu^*)$) where $\hat\alpha:\mc{X}\mapsto \mc{A}$ (resp. $\alpha^*:\mc{X}\mapsto \mc{A}$) represents the  equilibrium (resp. optimal) strategy  which assigns at each state the equilibrium (resp. optimal) action in order to minimize the aggregated discounted costs and $\hat \mu$ (resp. $\mu^*$) is the ergodic distribution of the population at equilibrium (resp. optimum). 
The traditional definition of an MDP based on the agent--environment pair is augmented with the  distribution of the population. In this new framework, the agent corresponds to the representative player of the mean field problem.  

We now define the type of environment to which the agent is assumed to have access. A key difference with prior works on RL for mean field problems is that we do not assume that agent can witness the evolution of the population's distribution. Instead, the environment estimates the distribution of the population by exploiting the symmetry property of the problem. Indeed, when the system is at equilibrium the law of the representative player matches the distribution of the population. As showed in the diagram of Figure~\ref{plot: MF-MDP}, at each time $n$, the agent observes its current state $X_{n} \in \mathcal{X}$ and then chooses an action $A_{n} \in \mathcal{A}$. An approximation of the distribution $\mu_{n}$ is computed by the environment based on the observed states of the representative player. Provided with the choice of the action and the estimate of the distribution, the environment generates the new state of the agent $X_{n+1}$ and assigns a cost $f_{n+1}$.  

\begin{figure}[H]
\centering
 \includegraphics[width=.6\linewidth]{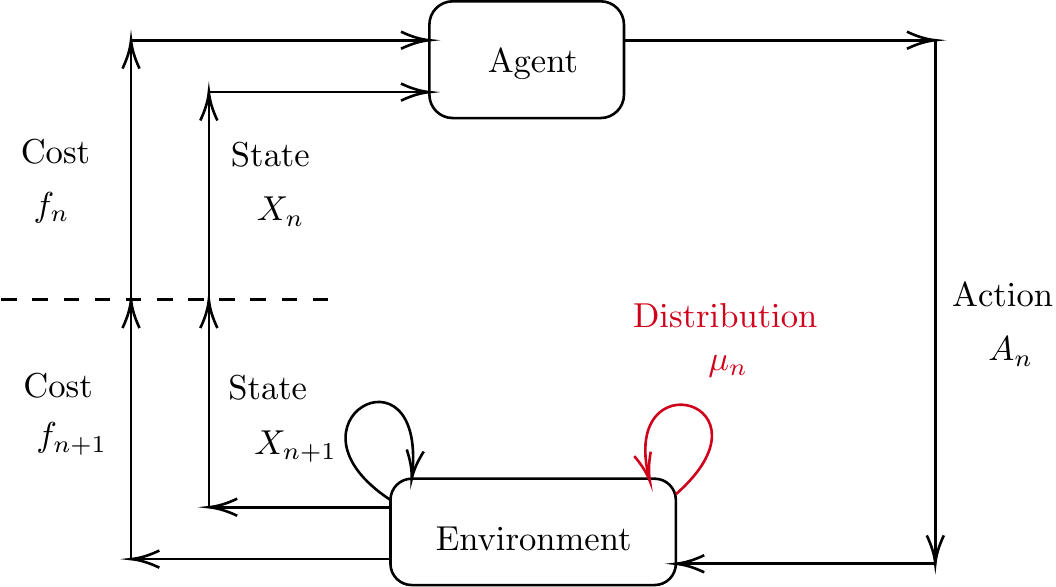}
  \caption{MDP with Mean Field interactions: Interaction of the representative agent with the environment. When the current state of the representative agent is $X_{n}$, given an action $A_{n}$, the environment produces an estimate of the distribution $\mu_{n}$, the new state $X_{n+1}$ and incurs a cost $f_{n+1}$ calculated by starting from the current state of the environment $X_{n}$ and using the transition controlled by $A_{n}$ and parameterized by $\mu_{n}$. }
  \label{plot: MF-MDP}
  
\end{figure}

The algorithm is designed to solve infinite horizon problems through an online approach, i.e. interacting with the environment. The learning procedure is based on splitting the infinite horizon in successive episodes in order to promote the exploration of the environment. The first episode is initialized based on the initial distribution of the representative player. Within a given episode, the agent updates her strategy at each learning step aiming to optimize the expected aggregated cost based on the current estimate of the distribution of the population $\mu_n$. Changes in the representative player's strategy have an effect on the population requiring to update $\mu_n$  accordingly. After an assigned number of steps ${T}$, the episode is terminated. A new episode is initialized based on the current version of the environment represented by the estimate of the population obtained at the last time point of the previous episode. One may think at the initialization step as a change in the choice of the representative player who provides the data flow. As the number of episodes increases, one expects the distribution of the representative player to converge to the limiting distribution. Within a given learning step, the environment computes an estimate of $\mu_n$ based on the current state of the agent $X_{n}$, provides the next state $X_{n+1}$ and assigns the cost $f_{n+1}$ given the triple $(X_{n},A_{n},\mu_{n})$. In other words, the environment consists of the dynamics of the agent and the cost structure. The case of our interest corresponds to the one in which the dynamics of the agent and the cost structure are unknown.  In this way, introducing the RL paradigm is equivalent to define a data driven approach to solve mean field models which may scale their applicability to real world problems.

In contrast with standard Q-learning, since in the mean field framework the cost function also depends on the distribution of the population, the goal here consists in learning the optimal strategy along with the corresponding ergodic distribution of the population, i.e. $(\hat \alpha, \hat \mu)$ in the MFG setting and $(\alpha^*, \mu^*)$ in the MFC setting. Based on the intuition provided in Section~\ref{sec: 2scale} related to the two timescale approach, we propose Algorithm~\ref{algo:U2MFQL}. At each step, we update the Q-table at the observed state-action pair $Q(X_{n},A_{n})$. With a different learning rate, the estimate of the distribution is updated based on the operator $\boldsymbol{\delta}:\mc{X} \mapsto \Delta^{|\mc{X}|} $ which maps the next observed state $X_{n+1} \in \mc{X}$ to the corresponding one-hot vector measure.  
To be specific, we identify the simplex $\Delta^{|\mc{X}|}$ with the subset $\left\{\left[\mu(x_i)\right]_{i=0,\dots, |\mc{X}|-1} \,:\,\mu(x_i)\in[0,1]\,\text{and}\, \sum_i \mu(x_i) = 1\right\}$ of $\RR^{|\mc{X}|}$. Then $\boldsymbol{\delta}$ is the function which associates to each element of $\mc{X} = \{x_0,\dots, x_{|\mc{X}|-1}\}$ the corresponding element of the canonical basis $(e_0,\dots,e_{|\mc{X}|-1})$ of $\RR^{|\mc{X}|}$, i.e.,  for each $i =0,\dots,|\mc{X}|-1$, $\boldsymbol{\delta}(x_i) = e_{i}$, which is an element of $\Delta^{|\mc{X}|}$ by the above identification. In order to learn the limiting distribution of the population through successive learning episodes, an estimate $\mu_{n_i}$ is computed for each step $n_i$ based on the sample $X^k_{n_i}$ collected from episodes $k=1,2,\dots.$ This approach attempts to minimize the correlation of the sampled states. The update rule presented in algorithm \ref{algo:U2MFQL} allocates more weight on the most recent samples allowing to forget progressively the initial sample that were obtained by a distribution far from the limiting one. At convergence, one may expect each $\mu_{n_i}$ to be an estimate of the limiting distribution.

The algorithm returns both an approximation $\mu_T^k$ of the distribution and an approximation $Q^k$ of the Q-function, from which an approximation of the optimal control can be recovered as $x \mapsto \argmin_{a \in \mc{A}}Q^k(x,a)$.

\begin{algorithm}[h]
\caption{Unified Two Timescales Mean Field Q-learning - Tabular version\label{algo:U2MFQL}} 
\begin{algorithmic}[1]
\REQUIRE $T$ : number of time steps in a learning episode, 
\\ $\mathcal{{X}} =\{ x_0, \dots, x_{|\mathcal{{X}}|-1}\}$ : finite state space, \\ 
$\mathcal{{A}} =\{ a_0, \dots, a_{|\mathcal{{A}}|-1}\}$ : finite action space, \\
$\mu_0$ : initial distribution of the representative player,\\ 
$\epsilon$ : parameter related to the $\epsilon-$greedy policy,\\
$tol_{\mu}$, $tol_Q$ : break rule tolerances.
\STATE \textbf{Initialization}: $Q^0(x,a)=0$ for all $(x,a)\in \mathcal{{X}}\times \mathcal{{A}}$,  $\mu^0_{n}=\bracket{\frac{1}{|\mathcal{{X}}|}, \dots,\frac{1}{|\mathcal{{X}}|}}$ for $n=0, \dots, T$ 

\FOR{each episode $k=1,2,\dots$}
\STATE \textbf{Initialization:} Sample $X^k_{0} \sim \mu^{k-1}_{T}$ and set $Q^k \equiv Q^{k-1}$  
\FOR{$ n \gets 0$ to $T-1$} 
\STATE
\textbf{Update} $\mu$: \\
$\mu^{k}_{n} = \mu^{k-1}_{n} + \rho^{\mu}_{k} (\boldsymbol{\delta}(X^k_{n}) - \mu^{k-1}_{n})$ where $\boldsymbol{\delta}(X^k_{n}) = \bracket{\mathbf{1}_{x_0}(X^k_{n}), \dots, \mathbf{1}_{x_{|\tilde{\mc{X}}|-1}}(X^k_{n})} $
\STATE
\textbf{Choose action $A^k_{n}$} using the $\epsilon$-greedy policy derived from $Q^k(X^k_{n},\cdot)$  
\\
\textbf{Observe cost $f_{n+1}=f(X^k_{n},A^k_{n},\mu^{k}_{n})$} and state $X^k_{n+1}$ provided by the environment %

\STATE
\textbf{Update} $Q$:\\ $Q^k(X^k_{n},A^k_{n})= Q^k(X^k_{n},A^k_{n})+\rho^Q_{k,n,X^k_{n},A^k_{n}} [f_{n+1} +\gamma \min_{a'\in \mc{A}} Q^k(X^k_{n+1},a')-Q^k(X^k_{n},A^k_{n})]$
\ENDFOR
\IF{ $ \delta(\mu_T^{k-1},\mu_T^{k}) \leq tol_{\mu}$ and $\Vert Q^k-Q^{k-1}\Vert_{1,1}<tol_Q$}
\STATE break
\ENDIF
\ENDFOR
\RETURN $(\mu^k,Q^k)$
\end{algorithmic}
\end{algorithm}

The Unified Two Timescales Mean Field Q-learning (U2-MF-QL) algorithm represents a unified approach to solve mean field problems.  On the one hand, by choosing the learning rate for the distribution of the population slower than the one for the Q-table, we obtain the solution to the MFG problem. Similarly to the scheme presented in Section~\ref{sec: 2scale}, the iterations in $Q$ perceive the quantity $\mu$ as quasi-static mimicking the freezing of the flow of measures characteristic in the solving scheme of a MFG.
On the other hand, by choosing the learning rate for the mean-field term faster than the one for the Q-table, we obtain the solution to the MFC problem. Indeed, this choice of the parameters guarantees that the distribution changes instantaneously for each variation of the control function (Q-table) replicating the structure of the MFC problem.

\subsection{Application to continuous problems}
\label{sec:app-continuous}
Although it is presented in a setting with finite state and action spaces, the application of the algorithm U2-MF-QL can be extended to continuous problems. Such adaptation requires truncation and discretization procedures to time, state and action spaces which should be calibrated based on the specific problem.

In practice, the learning episode will correspond to a uniform discretization $\tau = \{{t_n} \}_{n\in \{0,\dots, |\tau|-1\}} $ of a time interval $[0,T]$ with $T$ large enough. The environment will provide the new state and reward at these discrete times. We assume that $T$ is large enough to reach the ergodic regime. The continuous state space  will be represented as the disjoint union of equally sized neighbors. Each of them will be identified by its centroid and it will correspond to a row of the $Q$ table. Likewise, actions will be provided to the environment in a finite set $ \mc{A} = \{ a_0, \dots, a_{|\mc{A}|-1} \} \subset \RR^k $, and the distribution $\mu$ will be estimated on the set of centroids $\mc{X}=\{ x_0, \dots, x_{|\mc{X}|-1} \} \subset \RR^k $ identifying $\mu(x_i)$ as the probability of the neighbor centered in $x_i$. Then Algorithm~\ref{algo:U2MFQL} is ran on those spaces.

We will use the benchmark linear-quadratic models given in continuous time and space for which we have explicit formulas given in Appendix~\ref{appendix : calculations}. In that case, we use an Euler discretization. We do not address here the error of approximation since the purpose of this comparison with a benchmark is mainly for illustration.

\section{Numerical experiments}\label{sec:results}

In this section we illustrate our algorithm on a benchmark problem which admits an analytical solution. 

\subsection{Benchmark problem}
We illustrate our algorithm 
on the following model, in which the mean-field interactions are through the first moment. We take $d=k=1$, 
\begin{equation}\label{eq: benchmark}
  f(x,\alpha,\mu) = \frac{1}{2}\alpha^2 + c_1 \left( x- c_2 m\right)^2 + c_3 \left( x- c_4 \right)^2 + c_5 m^2,
    \qquad
    b(x,\alpha,\mu) = \alpha,   
\end{equation}
where $m = \int_{\RR} x \mu(x) dx$. Here the parameters $c_2, c_4 \in \RR$ and $c_1, c_3, c_5 \in \RR_+$ are constant such that $c_1+c_3 - c_1  c_2 \neq 0$. In this model the drift is simply the control, while the running cost can be understood as follows: the first term is a quadratic cost for controlling the diffusion, which penalizes high velocity, the second term incorporates mean field interactions and encourages the agents to be close to $c_2 m$ (if $c_2=1$, this has a mean-reverting effect), the third term creates an incentive for each agent to be close to the target position $c_4$, and the fourth term penalizes the population when its mean $m$ is far away from zero. We thus obtain a complex combination of various effects, which can be balanced depending on the choice of parameters. 

We consider both the corresponding MFG and MFC problems in the asymptotic formulation. The details on the solutions of these problems and their connection to the non-asymptotic formulation are given in the appendix.

\subsection{Numerical results}

We present the results obtained by applying the U2-MF-QL algorithm to the mean field problems based on the running cost and drift specified in \eqref{eq: benchmark}.  These results show how the algorithm successfully learns the MFG solution or the MFC solution based on simply tuning the learning rates. Moreover, this shows that the algorithm manages to solve problems defined on continuous time and continuous state, action spaces even though it is conceived for discrete problems. Such applications require to apply truncation and discretization procedures to time, state and actions which should be calibrated on a problem base. 

We consider the problem defined by the choice of parameters: $c_1=0.25$, $c_2=1.5$, $c_3=0.50$, $c_4=0.6$, $c_5=5$, discount parameter $\beta=1$ and volatility $\sigma = 0.3$. The infinite time horizon is truncated at time $T=20$. The continuous time is discretized using step  $\Delta t= 10^{-2}$. Recall that $\gamma$ in the discrete time setting corresponds to $e^{-\beta \Delta t}$ in the continuous time setting. The action space is given by $\mc{A} = \{ a_0=-1, \dots, a_{N_{\mc{A}}}=1 \}$ and the state space by $\mc{X} = \{ x_0=-2+x_c, \dots, x_{N_{\mc{X}}}=2+x_c \}$, where $x_c$ is the center of the state space. The step size for the discretization of the spaces  $\mc{X}$ and $\mc{A}$ is given by $\Delta_{.} = \sqrt{\Delta t} = 10^{-1} $.  The state space $\mc{X}$ and the action space $\mc{A}$ have been chosen large enough to make sure that the state is within the boundary most of the time. In practice, this would have to be calibrated in a model-free way through experiments.  In this example, for the numerical experiments, we used the knowledge of the model. In particular, we choose $x_c=0.5$ for both examples.  Note that if the problem under consideration is posed on finite spaces, this issue does not occur since the domain is fixed.  The exploitation-exploration trade off is tackled on each episode  using an $\epsilon-$greedy policy, see~\eqref{intro: pi}. In particular, the value of $\epsilon$ is fixed to $0.15$.

\noindent
We present the following results for both the MFG and MFC benchmark examples:
\begin{enumerate}
	\item learning rates analyses;
	\item learning of the controls and the ergodic distribution;
	\item empirical error analyses;
	\item empirical analyses of the stopping criteria.
\end{enumerate}

\subsubsection{Learning rates analyses}

It is important to observe that even if in the MFC case the choice of $\rho^\mu_k$ below does not satisfy the classical Robbins-Monro summability condition recalled in Section \ref{sec: unified 2scale approach}, the numerical convergence of the algorithm is obtained suggesting that these requirements may be relaxed in this framework. Failing in satisfying these conditions generates a noisy approximation of the distribution $\mu$ in the MFC problem. However, averaging over the last 10k episodes allows to minimize such noise as showed in the Figures below. Based on the theoretical results given in \cite{even2003learning}, we define the learning rates appearing in Algorithm~\ref{algo:U2MFQL} as follows: 
\begin{equation} \label{eqn : rates}
 \rho^Q_{k,n,x,a}=\frac{1}{\parentheses{1+\#|(x,a,k,n)|}^{\omega^Q}},     \quad \quad \rho^{\mu}_k=\frac{1}{(1+k)^{\omega^{\mu}}}, 
\end{equation}
 where $\#|(x,a,k,n)|$ is the number of times that the algorithm visited state $x$ and performed action $a$ until episode $k$ and time $t_n$. The exponent $\omega^Q$ can take values in $(\frac{1}{2},1).$ The value of $\omega^{\mu}$ is chosen depending on the value of $\omega^Q$ and the cooperative or non-cooperative nature of the problem we want to solve. The algorithm is run over $80 \times 10^{3}$ episodes over the interval $[0,T]$.
 
\vskip 6pt
\noindent
\textbf{Figures \ref{fig:rates_MFG_500}, \ref{fig:rates_MFC_500}, \ref{fig:rates_MFG_80k}, \ref{fig:rates_MFC_80k}: comparison of the learning rates.} The solution of the MFG benchmark is reached based on the choice $(\omega^Q,\omega^{\mu})=(0.55,0.85)$, such that $\rho^{\mu}<\rho^Q $. As pointed out in section \ref{sec: unified 2scale approach}, by satisfying this relation the $Q$-function evolves faster than the estimation of the distribution mimicking the solving scheme of a MFG. On the other hand, the solution of the MFC benchmark can be obtained by opting for the pair of parameters $(\omega^Q,\omega^{\mu})=(0.65,0.15)$ such that $\rho^{\mu}>\rho^Q $. In figures \ref{fig:rates_MFG_500}, \ref{fig:rates_MFC_500}, \ref{fig:rates_MFG_80k}, \ref{fig:rates_MFC_80k}, we suppose that $\#|(x,a,k,1)| = k $.  The $x-$axis refers to the episode. The $y-$axis represents the rate evaluated at episode $k$. 
\begin{figure}[H]
\centering
\begin{minipage}{.5\textwidth}
  \centering
  \includegraphics[width=.9\linewidth]{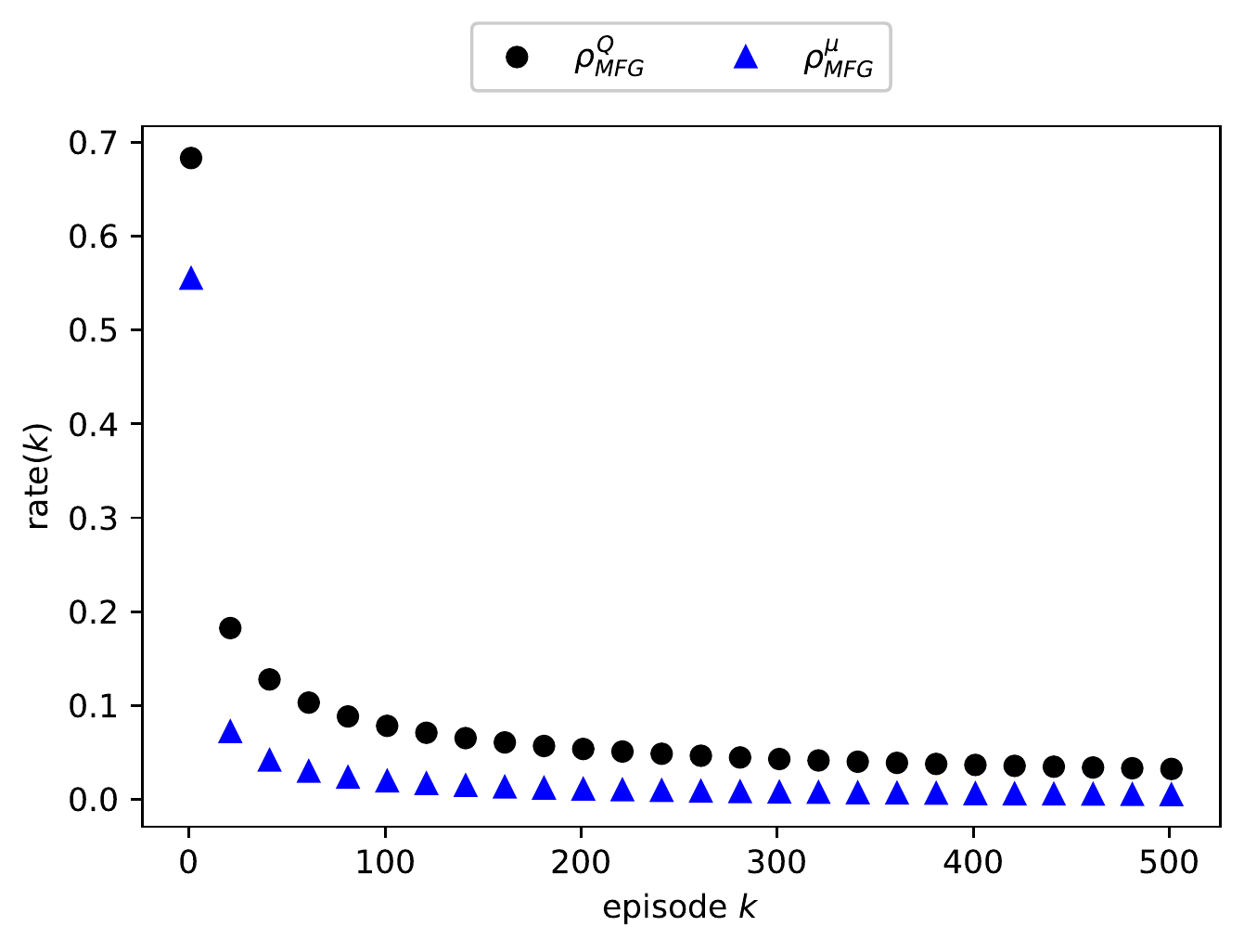}
  \caption{MFG: learning rates over the first $500$ episodes 
  }
  \label{fig:rates_MFG_500}
\end{minipage}%
\begin{minipage}{.5\textwidth}
  \centering
  \includegraphics[width=.9\linewidth]{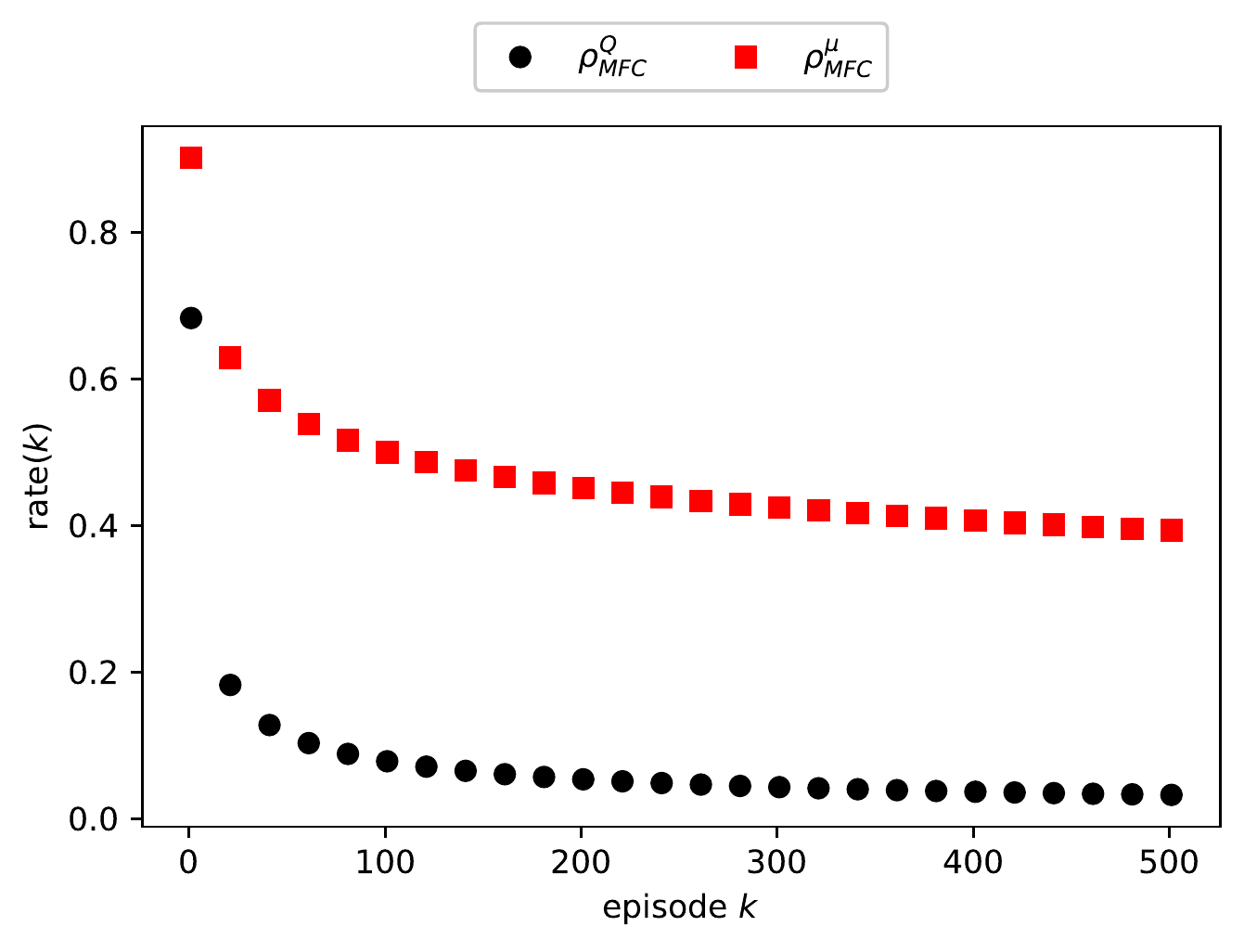}
  \caption{MFC: learning rates over the first $500$ episodes}
  \label{fig:rates_MFC_500}
\end{minipage}
\end{figure}

\begin{figure}[H]
\centering
\begin{minipage}{.5\textwidth}
  \centering
  \includegraphics[width=.9\linewidth]{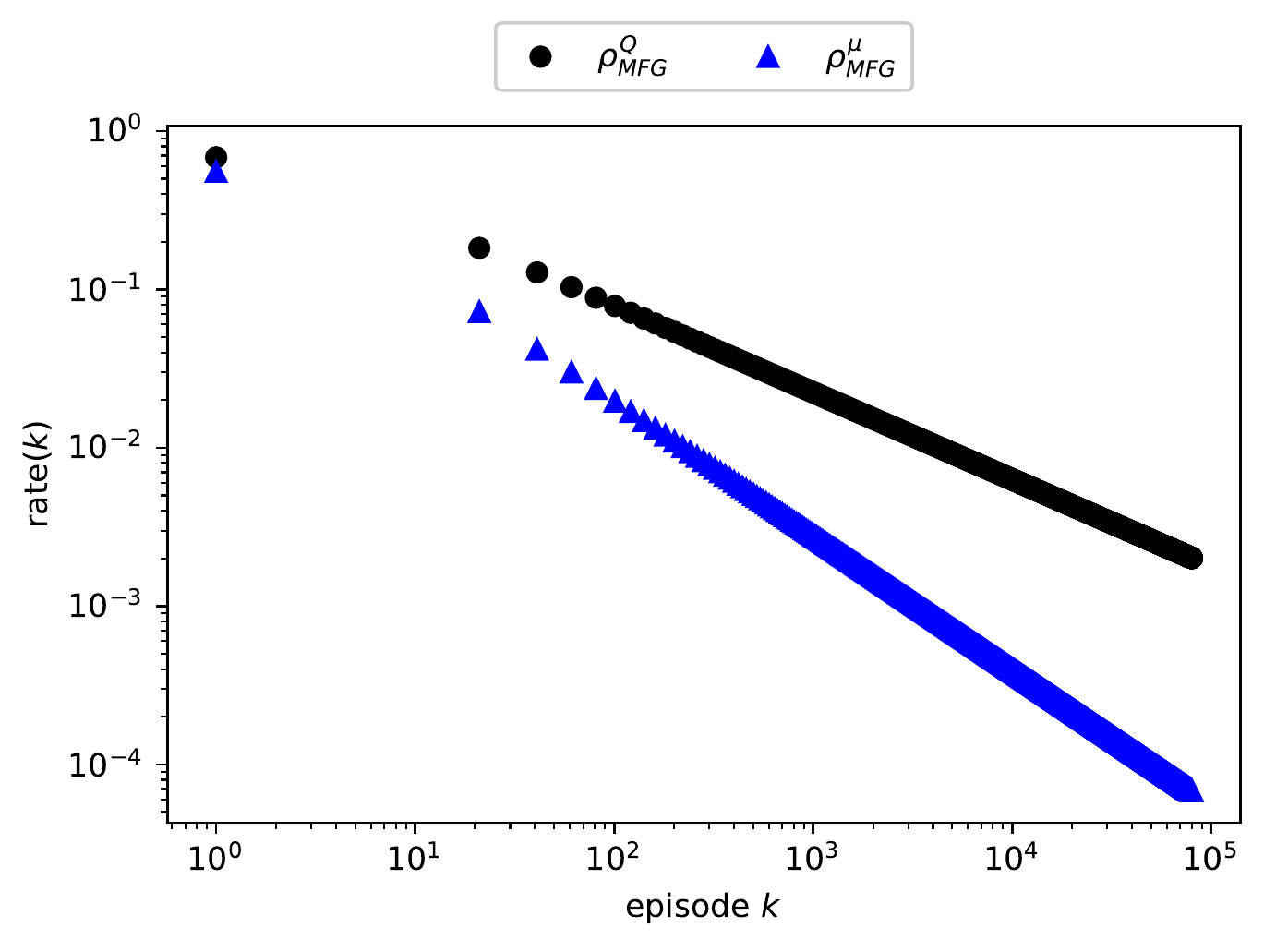}
  \caption{MFG: learning rates over $80 \times 10^{3}$ episodes 
  }
  \label{fig:rates_MFG_80k}
\end{minipage}%
\begin{minipage}{.5\textwidth}
  \centering
  \includegraphics[width=.9\linewidth]{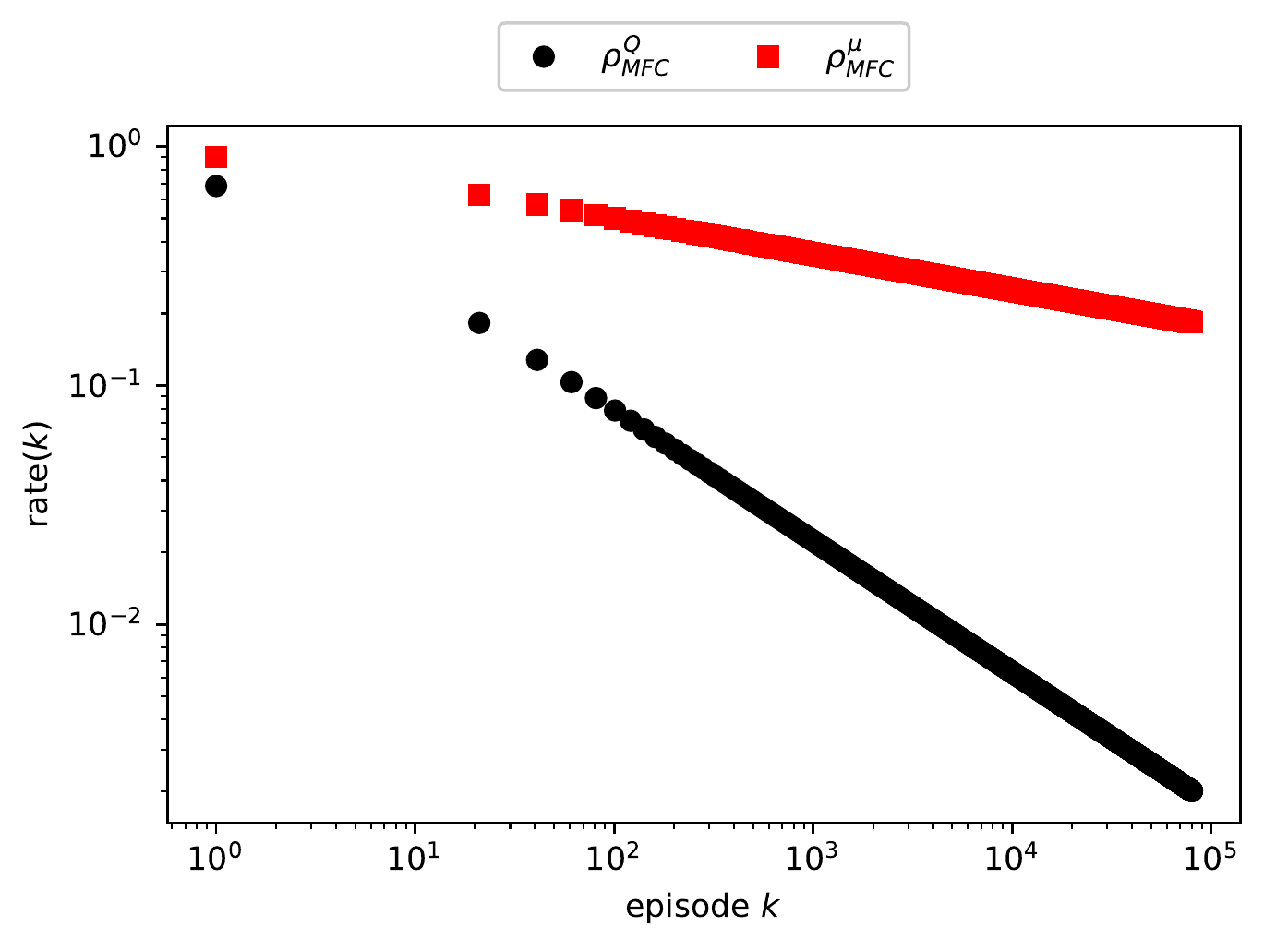}
  \caption{MFC: learning rates over $80 \times 10^{3}$ episodes }
  \label{fig:rates_MFC_80k}
\end{minipage}
\end{figure}
\vskip 6pt
\noindent
\textbf{Figures \ref{fig:2scale_MFG_low}, \ref{fig:2scale_MFG_mean}, \ref{fig:2scale_MFC_low}, \ref{fig:2scale_MFC_mean}: Empirical check of the two timescale conditions.} The U2-MF-QL algorithm is based on an asynchronous QL approach which makes use of different learning rates for each $Q(x,a)$ based on the number of visits to the relative  state-action pair. An empirical check of the two timescale conditions presented in section \ref{sec: unified 2scale approach} is presented in the following plots. The number of visits to each state depends on their proximity to the mean of the ergodic distribution. As a proof of concept, the learning rates for two different states in the MFG and MFC frameworks are analyzed after $80  \times 10^3$ learning epochs. The plots on the left are relative to the state on the left bound of $\mc{X}$, while the plots on the right are relative to the closest state to the theoretical mean. Each plot shows the value of the learning rates $\rho_k^{\mu}$ and  $\rho_{k,n,x,a}^{Q}$ together with the counter of visits to each pair $(x,a)$. The two timescale conditions are satisfied in each plot. The number of visits changes from order $10^2$ for the state on the border of $\mc{X}$ to order $10^7$ for the closest state to the ergodic mean. The $x-$axis refers to the action. The left $y-$axis represents the learning rate. The right $y-$axis represents the counter of visits for each state-action pair. 
\begin{figure}[H]
\centering
\begin{minipage}{.5\textwidth}
  \centering
  \includegraphics[width=.9\linewidth]{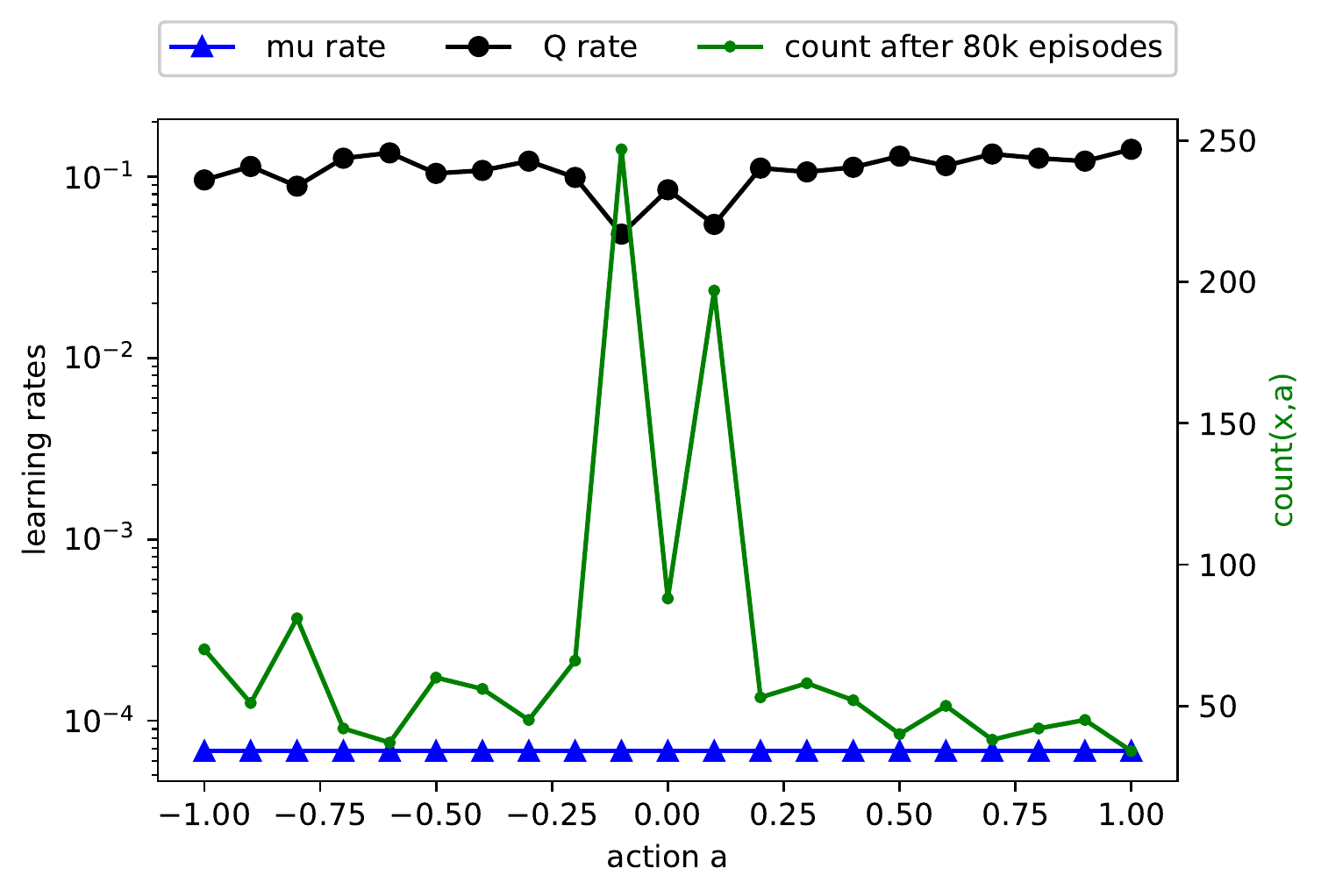}
  \caption{MFG: comparison learning rates for state $x = -1.50$
  }
  \label{fig:2scale_MFG_low}
\end{minipage}%
\begin{minipage}{.5\textwidth}
  \centering
  \includegraphics[width=.9\linewidth]{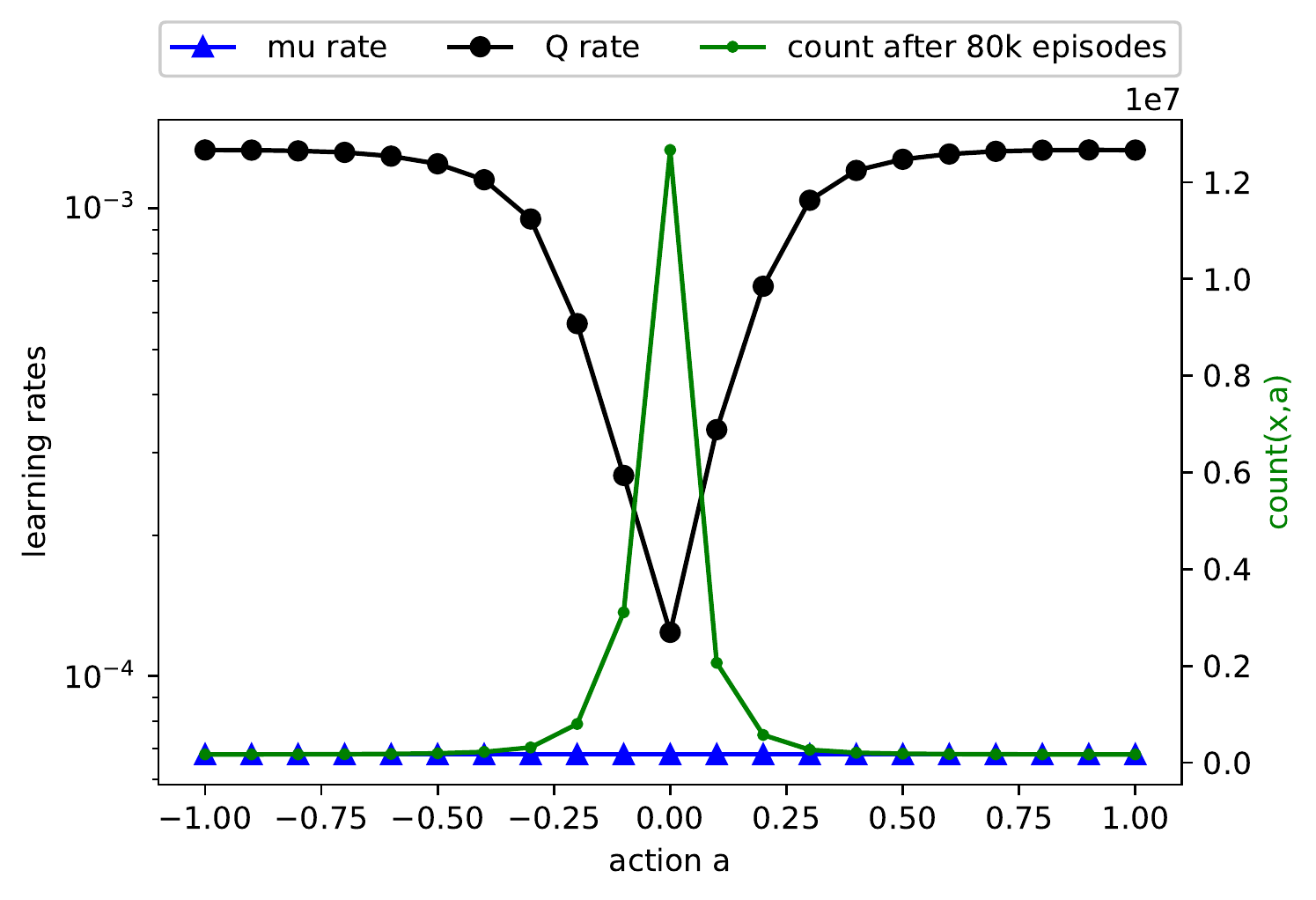}
  \caption{MFG: comparison learning rates for state $x = 0.80$}
  \label{fig:2scale_MFG_mean}
\end{minipage}
\end{figure}

\begin{figure}[H]
\centering
\begin{minipage}{.5\textwidth}
  \centering
  \includegraphics[width=.9\linewidth]{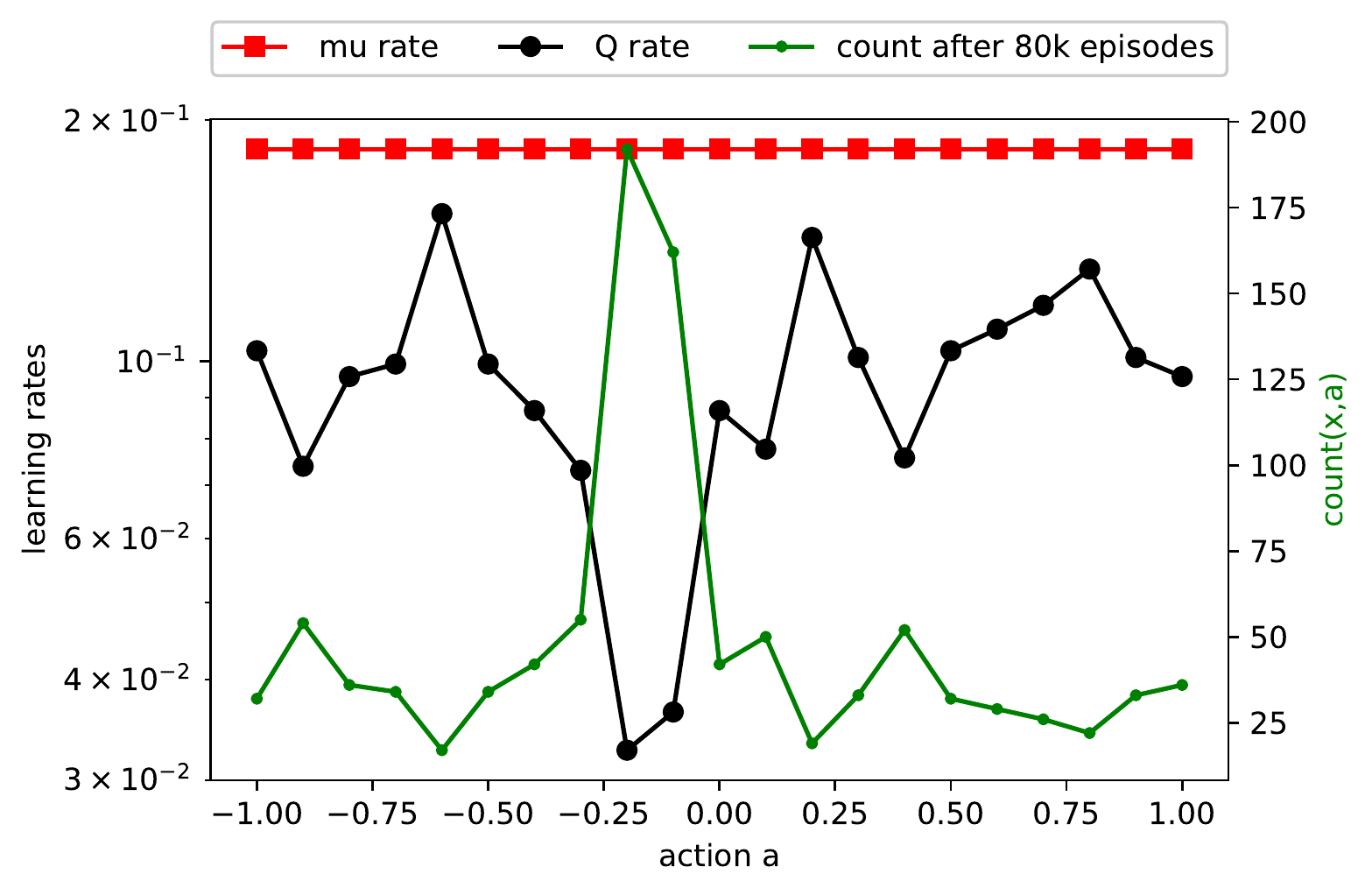}
  \caption{MFC:  comparison learning rates for state $x = -1.50$
  }
  \label{fig:2scale_MFC_low}
\end{minipage}%
\begin{minipage}{.5\textwidth}
  \centering
  \includegraphics[width=.9\linewidth]{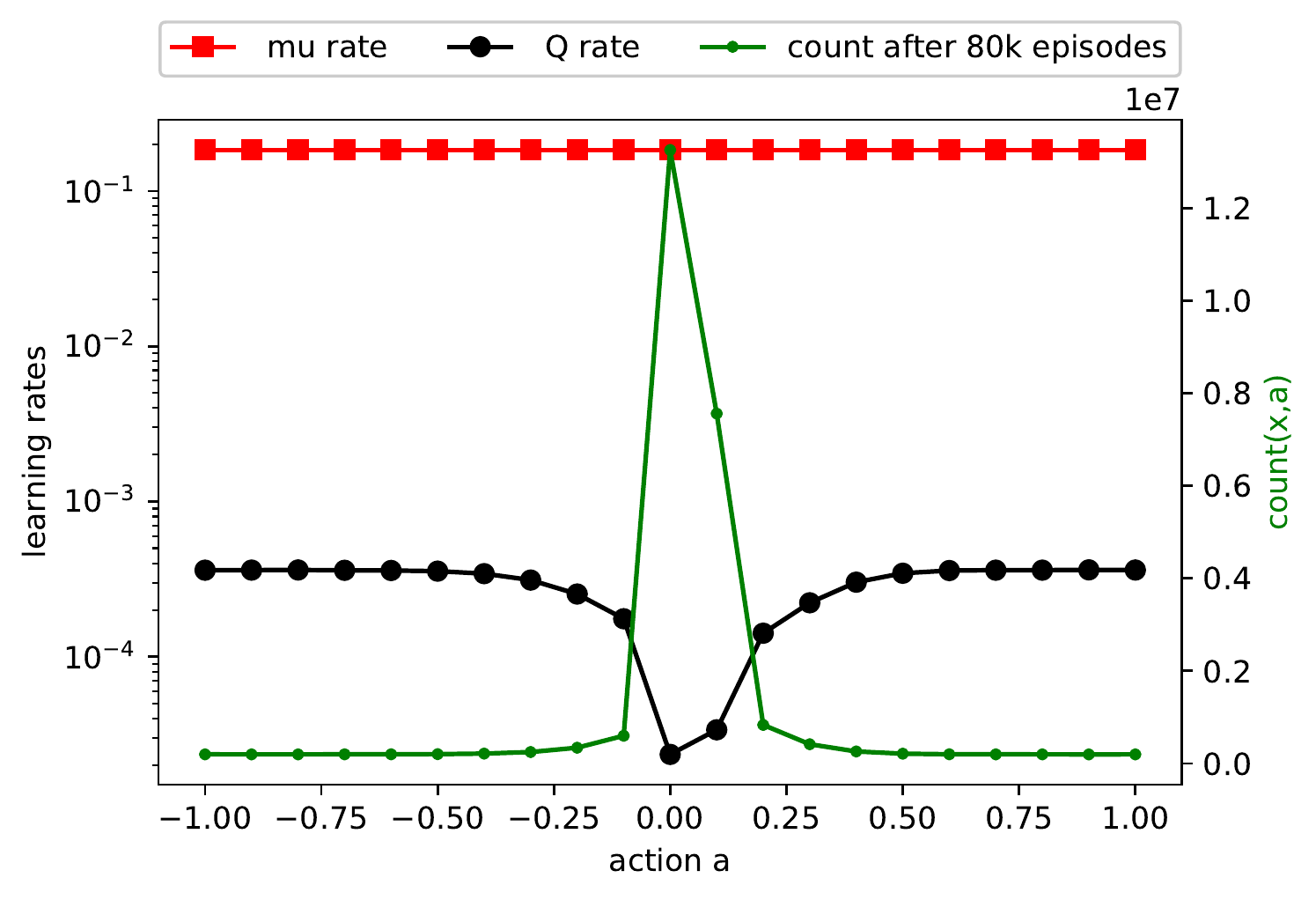}
  \caption{MFC:  comparison learning rates for state $x = 0.10$ }
  \label{fig:2scale_MFC_mean}
\end{minipage}
\end{figure}

\subsubsection{Learning of the controls and the ergodic distribution}
\textbf{Figures \ref{fig:control_MFG}, \ref{fig:control_MFC}, \ref{fig:control_2_MFG}, \ref{fig:control_2_MFC}, \ref{fig:v_fct_MFG}, \ref{fig:v_fct_MFC}:  controls, distributions and value functions learned by the algorithm.} The controls and the distribution learned by the algorithm are compared with the theoretical solution obtained in the appendix \ref{appendix : calculations}. As presented in Section \ref{sec: 2scale}, the control $\alpha(x)$ is obtained as the $\arg\min_aQ(x,a)$. Similarly, the value function $V(x)$ can be recovered as $\min_aQ(x,a)$. The $x-$axis represents the state variable $x$. In Figures \ref{fig:control_MFG}, \ref{fig:control_MFC}, \ref{fig:control_2_MFG}, \ref{fig:control_2_MFC}, \ref{fig:v_fct_MFG}, \ref{fig:v_fct_MFC}, the left $y-$axis relates to the action $\alpha(x)$. The right $y-$axis refers to the probability mass $\mu(x)$. The red (resp. blue) line shows the theoretical control function for the MFG (resp. MFC) problem. The black dots are the controls learned by the algorithm. Note that the peak of the distribution $\mu$ is not located at the same point $x$ for MFG and MFC.  Note that the peak of the distribution $\mu$ is not located at the same point $x$ for MFG and MFC. In Figures  \ref{fig:control_MFG}, \ref{fig:control_2_MFG}, the $y-$axis corresponds to the value function $V(x)$. The continuous lines refer to the theoretical solution. The black dots are the numerical approximation recovered by the $Q$-function. We observe that the algorithm converges to different solutions based on the choice of the pair $(\omega^Q,\omega^{\mu})$. On the left, the choice $(\omega^Q,\omega^{\mu})=(0.55,0.85)$ produces the approximation of the solution of the MFG. On the right, the set of parameters  $(\omega^Q,\omega^{\mu})=(0.65,0.15)$ lets the algorithm learn the solution of the MFC problem. In Figures \ref{fig:control_MFG}
, \ref{fig:control_MFC} the learned controls and the learned ergodic distribution is averaged over $10$ runs. In Figures \ref{fig:control_2_MFG}
, \ref{fig:control_2_MFC} the learned controls and the learned distribution $\mu_T$ is averaged over $10$ runs and the last $ 10^4$ episodes.
\begin{figure}[H]
\centering
\begin{minipage}{.5\textwidth} 
  \centering 
  \includegraphics[width=.9\linewidth]{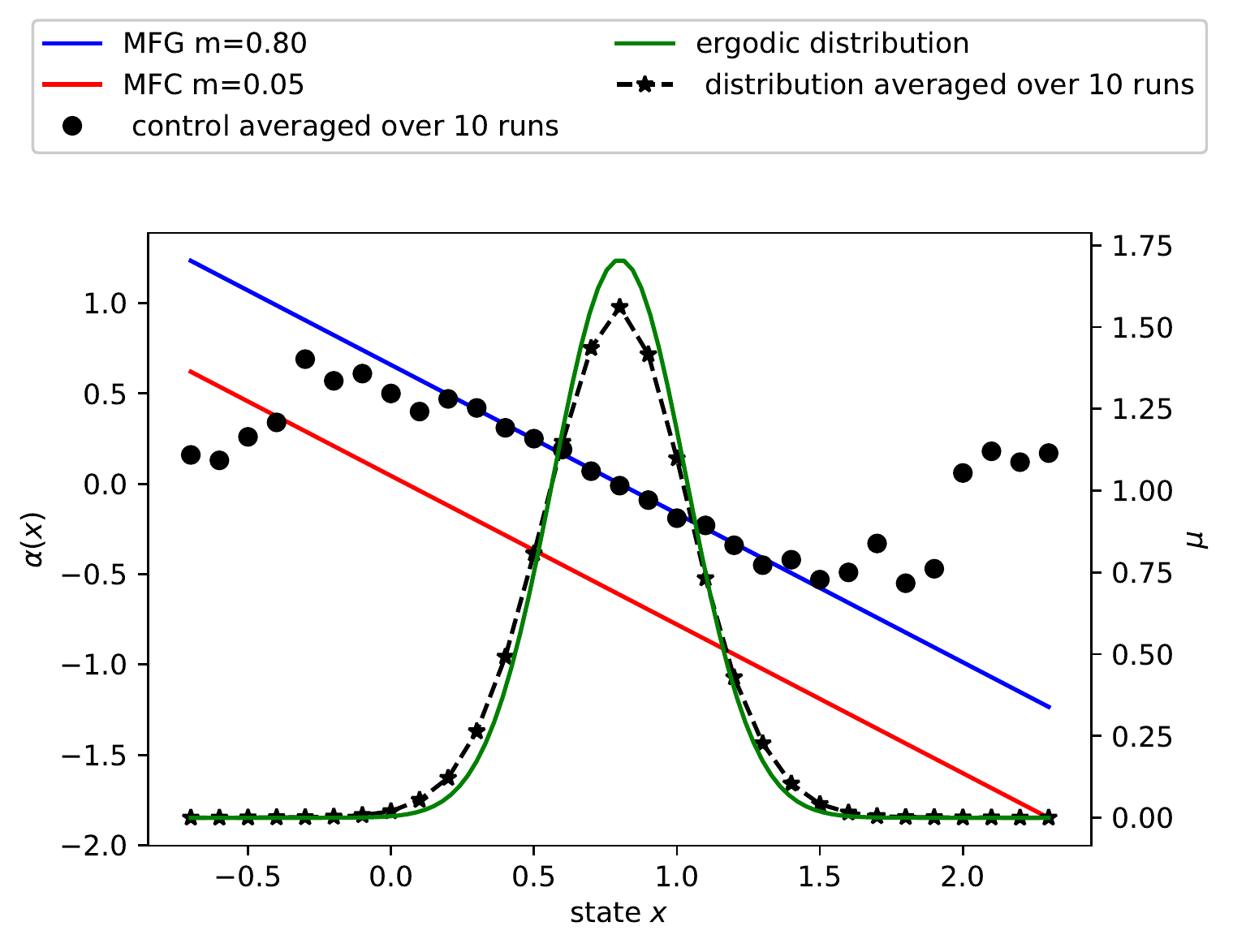}
  \caption{MFG: results averaged over $10$ runs
  }
  \label{fig:control_MFG}
\end{minipage}%
\begin{minipage}{.5\textwidth} 
  \centering 
  \includegraphics[width=.9\linewidth]{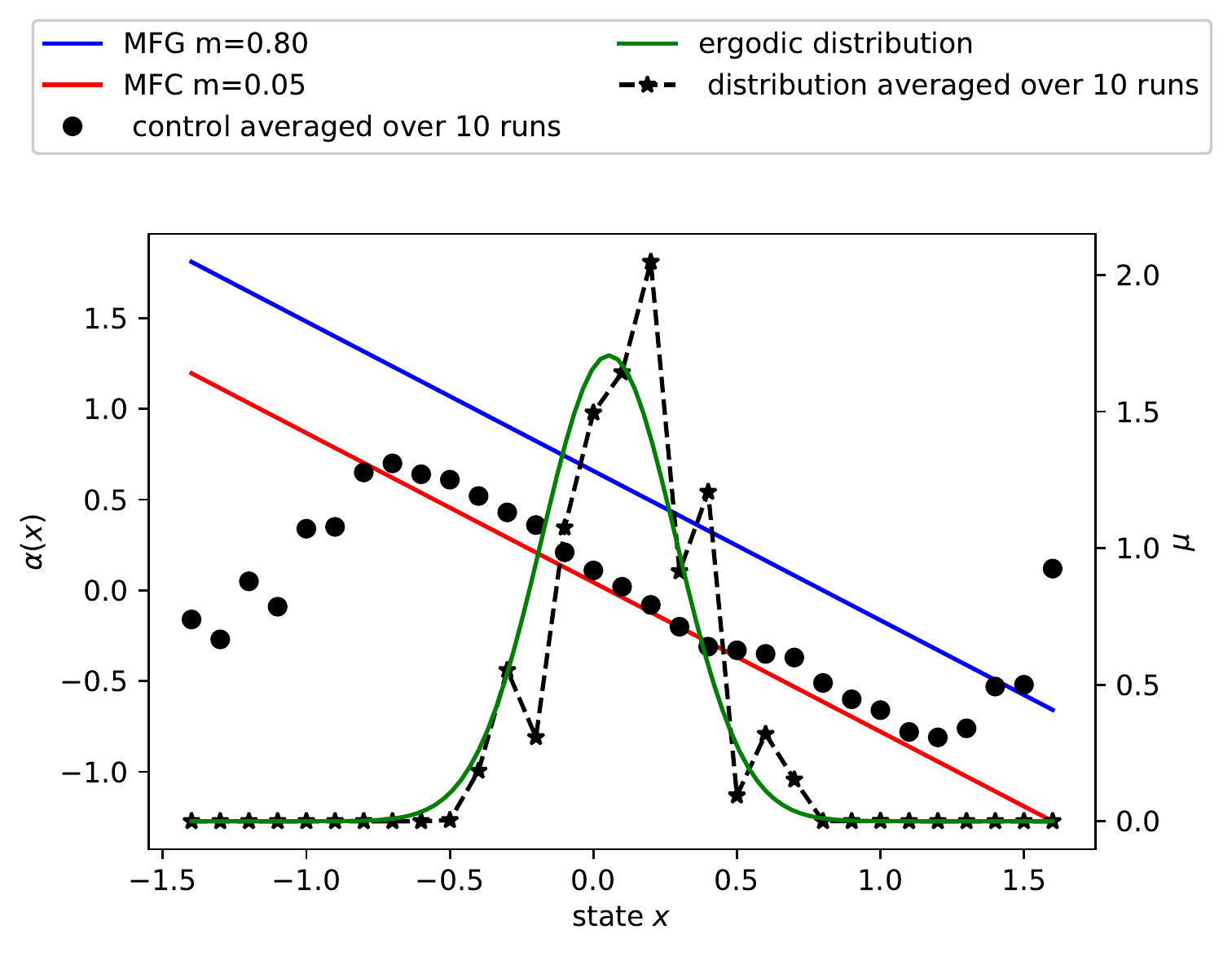}
  \caption{MFC: results averaged over $10$ runs
  }
  \label{fig:control_MFC}
\end{minipage}%
\end{figure}

\begin{figure}[H]
\centering
\begin{minipage}{.5\textwidth} 
  \centering 
  \includegraphics[width=.9\linewidth]{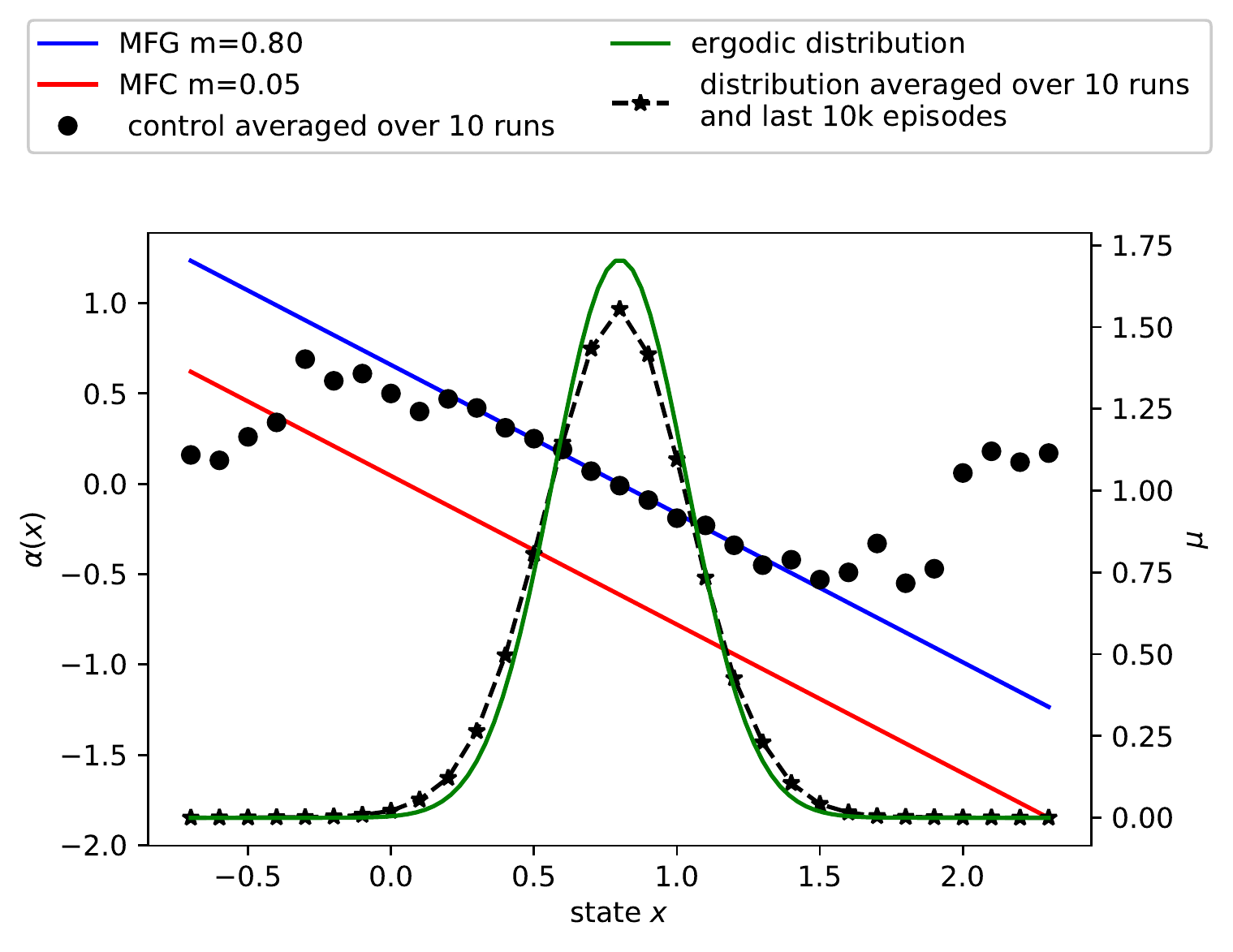}
  \caption{MFG: results averaged over $10$ runs and last $10k$ episodes 
  }
  \label{fig:control_2_MFG}
\end{minipage}%
\begin{minipage}{.5\textwidth} 
  \centering 
  \includegraphics[width=.9\linewidth]{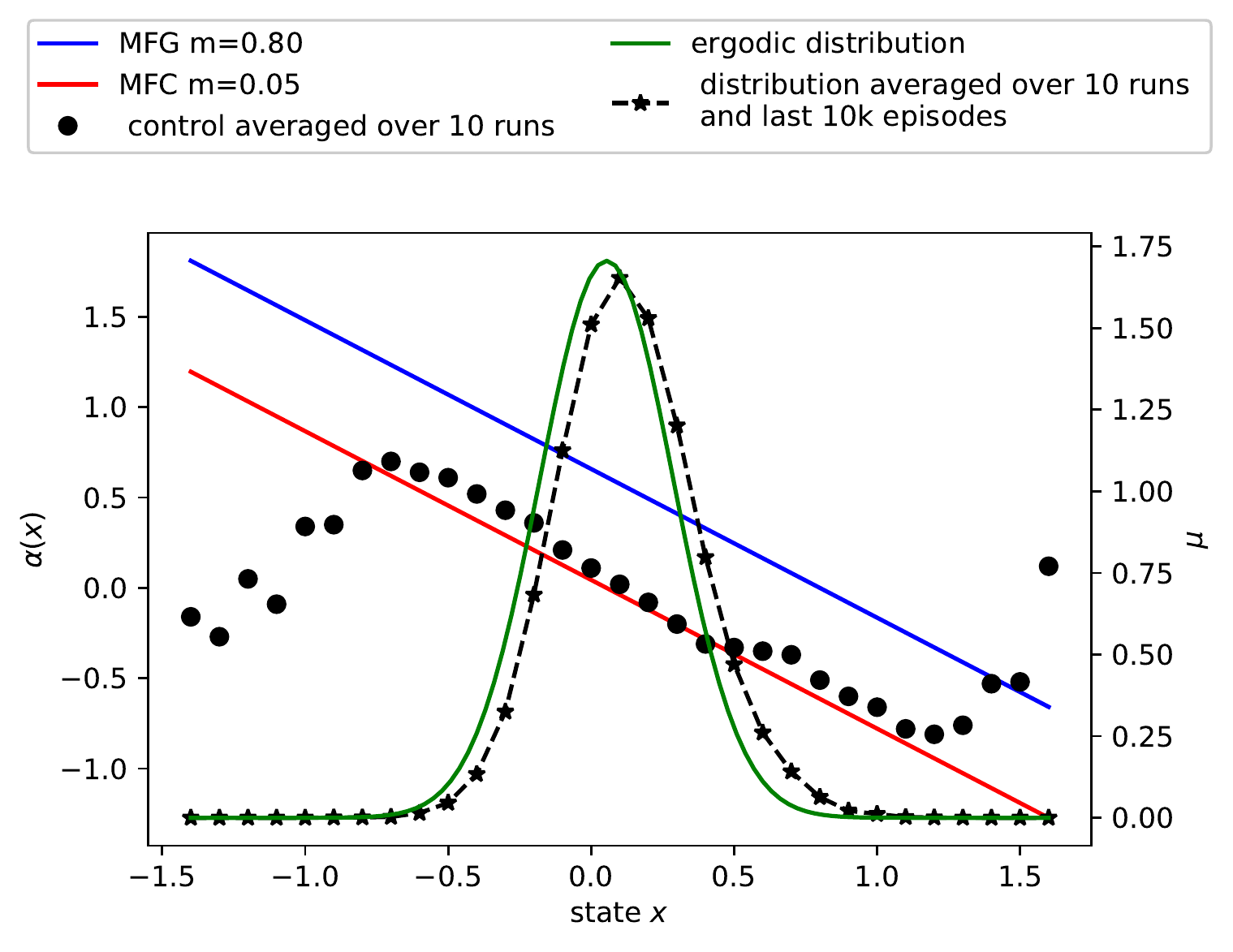}
  \caption{MFC: results averaged over $10$ runs and last $10k$ episodes 
  }
  \label{fig:control_2_MFC}
\end{minipage}%
\end{figure}

\begin{figure}[H]
\centering
\begin{minipage}{.5\textwidth} 
  \centering 
  \includegraphics[width=.9\linewidth]{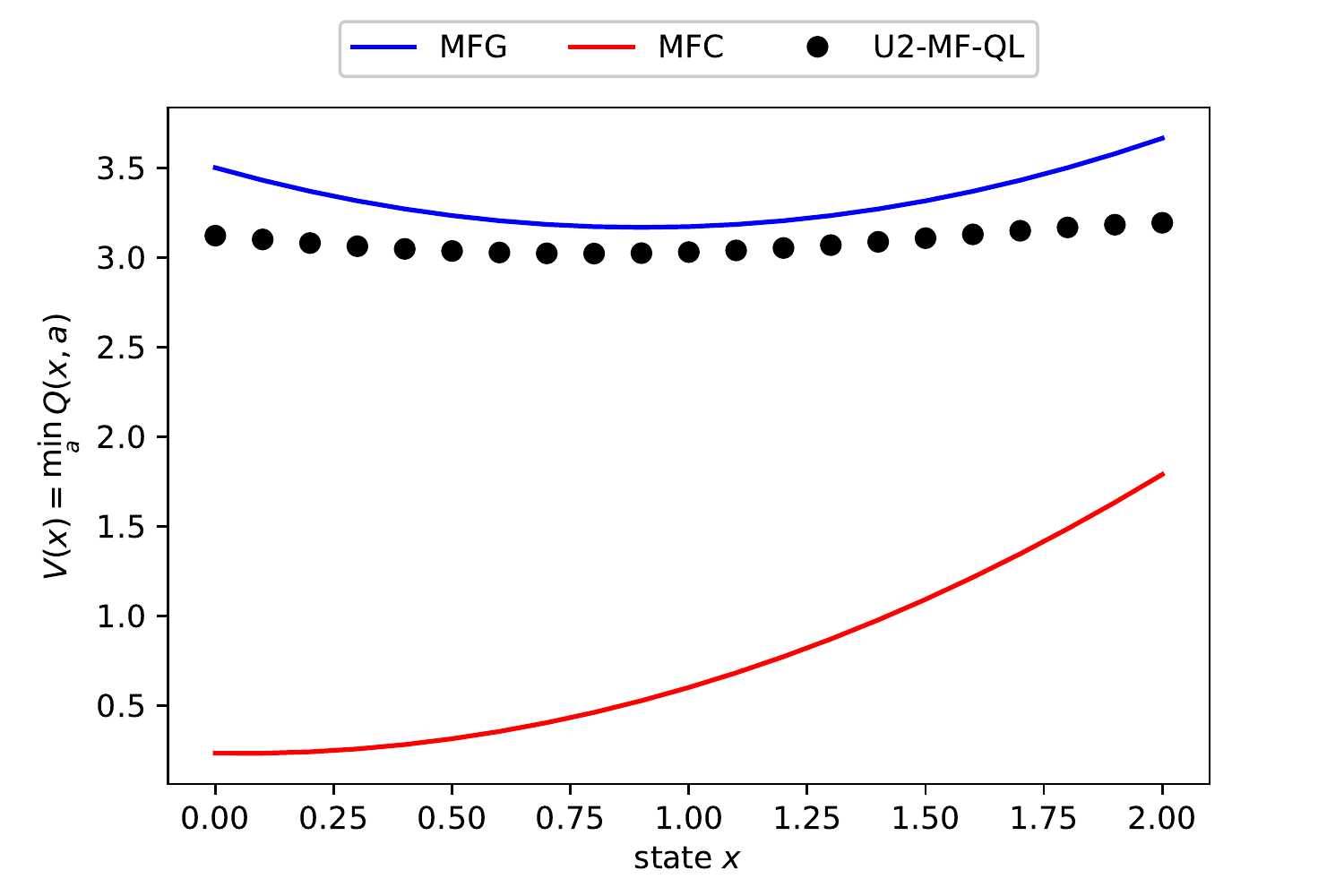}
  \caption{MFG: value function
  }
  \label{fig:v_fct_MFG}
\end{minipage}%
\begin{minipage}{.5\textwidth} 
  \centering 
  \includegraphics[width=.9\linewidth]{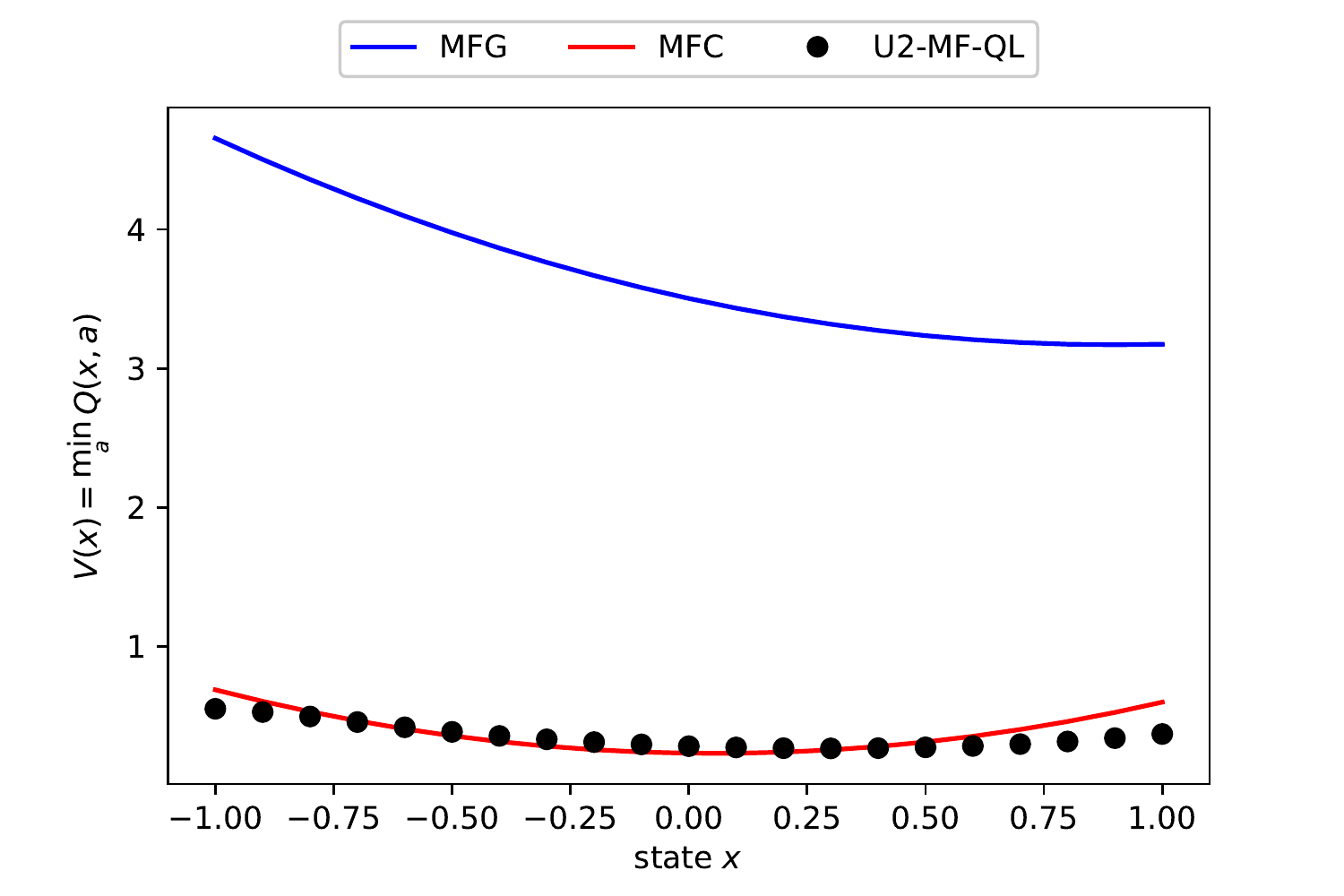}
  \caption{MFC: value function 
  }
  \label{fig:v_fct_MFC}
\end{minipage}%
\end{figure}

\subsubsection{Empirical error analyses}
\textbf{Figures \ref{fig:control_error_MFG_first}, \ref{fig:control_error_MFC_first}: MSE error on the control.} 
A metric used to evaluate the numerical results consists in the mean squared error (MSE) of the controls learned by episode $k$  with respect to the theoretical solution presented in Appendix~\ref{appendix : calculations}. In particular, this metric considers the states $x \in \mathcal{X}$ where the ergodic distribution $\hat\mu$ is mostly concentrated.  Let $\mathcal{C}_{MFG}\subset \mathcal{X}$ be centered in $\hat m$ s.t. $\hat \mu (\mathcal{C}_{MFG}) = 0.99$, then the mean squared error by episode k for run i and its average over all runs are defined as
$$
\text{MSE}_{\alpha}(i,k) = \frac{1}{{|\mc{C}_{MFG}|}}\sum_{j=0}^{|\mc{C}_{MFG}|-1} (\alpha^{i,k}(x_j)-\hat\alpha(x_j))^2, \quad \text{MSE}_{\alpha}(k) = \frac{1}{\#runs}\sum_{i=0}^{\#runs} \text{MSE}_{\alpha}(i,k). 
$$
The $x-$axis represents the number of episodes used for learning. The $y-$axis represents the mean squared error averaged over $10$ runs (solid line) and its standard deviation (shaded region). 
\begin{figure}[H]
\centering
\begin{minipage}{.5\textwidth}
  \centering
  \includegraphics[width=.9\linewidth]{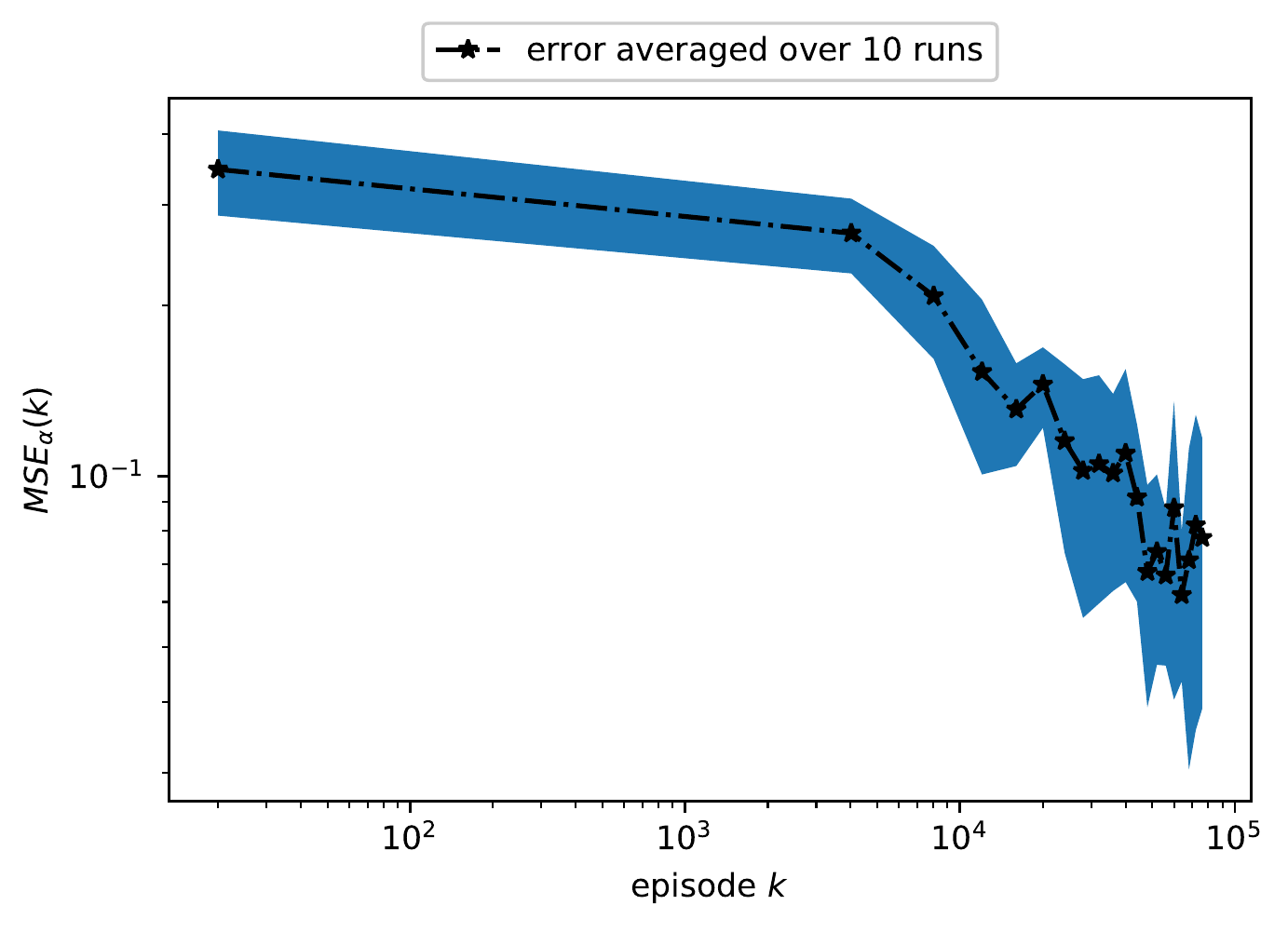}
  \caption{MFG: squared root of $\text{MSE}_{\alpha}(k)$ }
  \label{fig:control_error_MFG_first}
\end{minipage}%
\begin{minipage}{.5\textwidth}
  \centering
  \includegraphics[width=.9\linewidth]{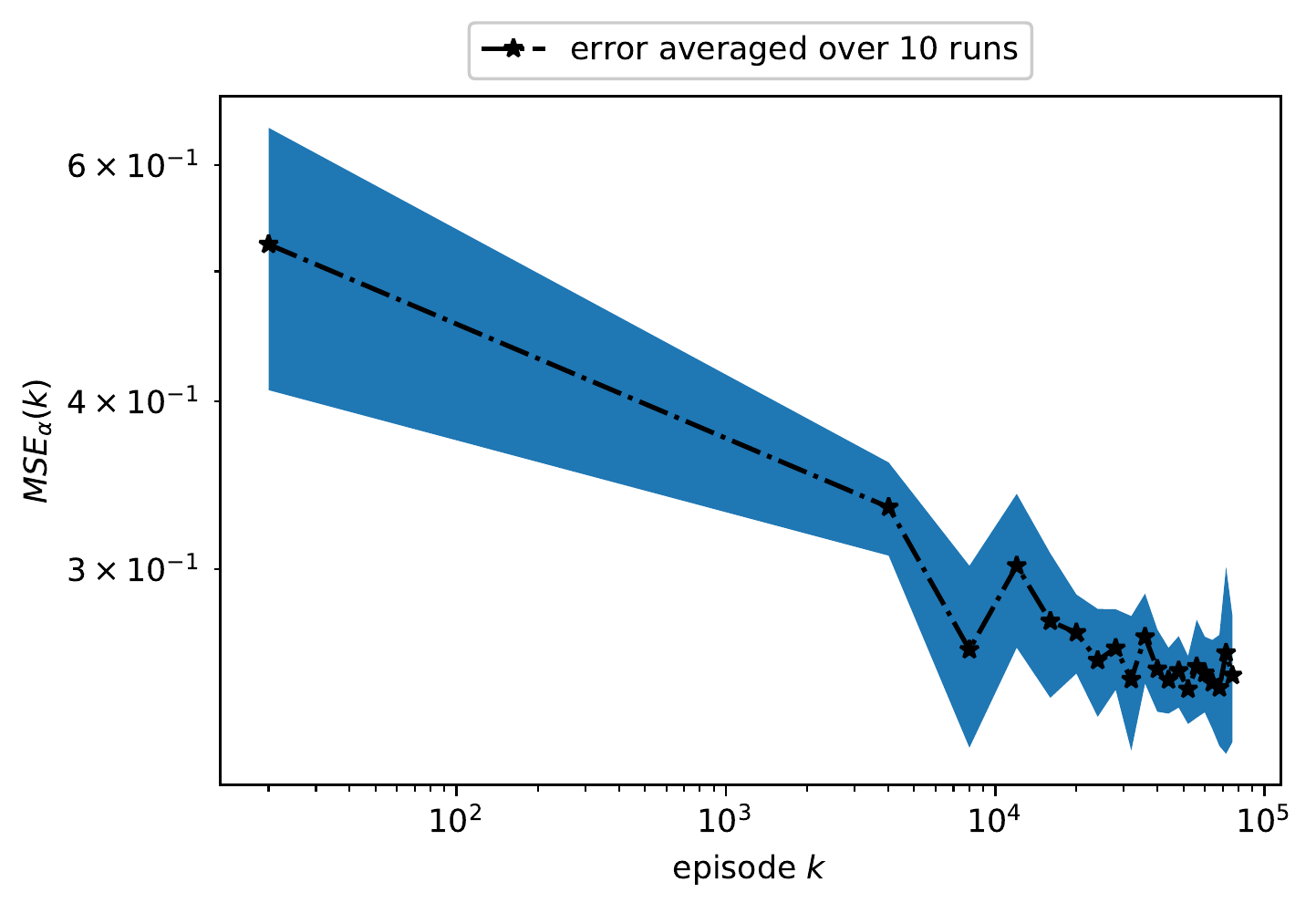}
  \caption{MFC: squared root of $\text{MSE}_{\alpha}(k)$}
  \label{fig:control_error_MFC_first}
\end{minipage}
  
\end{figure}

\vskip 6pt
\noindent
\textbf{Figures \ref{fig:error_mean_MFG}, \ref{fig:error_mean_MFC}:  MSE on the ergodic mean.} 
A metric used to evaluate the numerical results consists in the squared error of the ergodic mean learned by episode $k$ compared with its theoretical value obtained in Appendix \ref{appendix : calculations} averaged over the total numbers of runs, i.e. 
$$
    \text{MSE}_{m}(k) = \frac{1}{\#runs}\sum_{i=0}^{\#runs} (m_T^{i,k}-\hat m )^2 . 
$$

The $x-$axis represents the number of episodes used for learning. The $y-$axis represents the error averaged over $10$ runs (solid line) and its standard deviation (shaded region).  For the MFG, the error in the approximation of the ergodic mean reduces both in mean and standard deviation by increasing the number of episodes. For the MFC case, an oscillating behavior is  observed. The choice of $\omega_{\mu}=0.15$ in the learning rates defined in \ref{eqn : rates}  allows to quicker adjustment of the mean by allocating more weights on the most recent sample. In this way, the algorithm mimics the nature of the MFC problem at the expense of a slower and more oscillating convergence.

\begin{figure}[H]
\centering
\begin{minipage}{.5\textwidth}
  \centering
  \includegraphics[width=.9\linewidth]{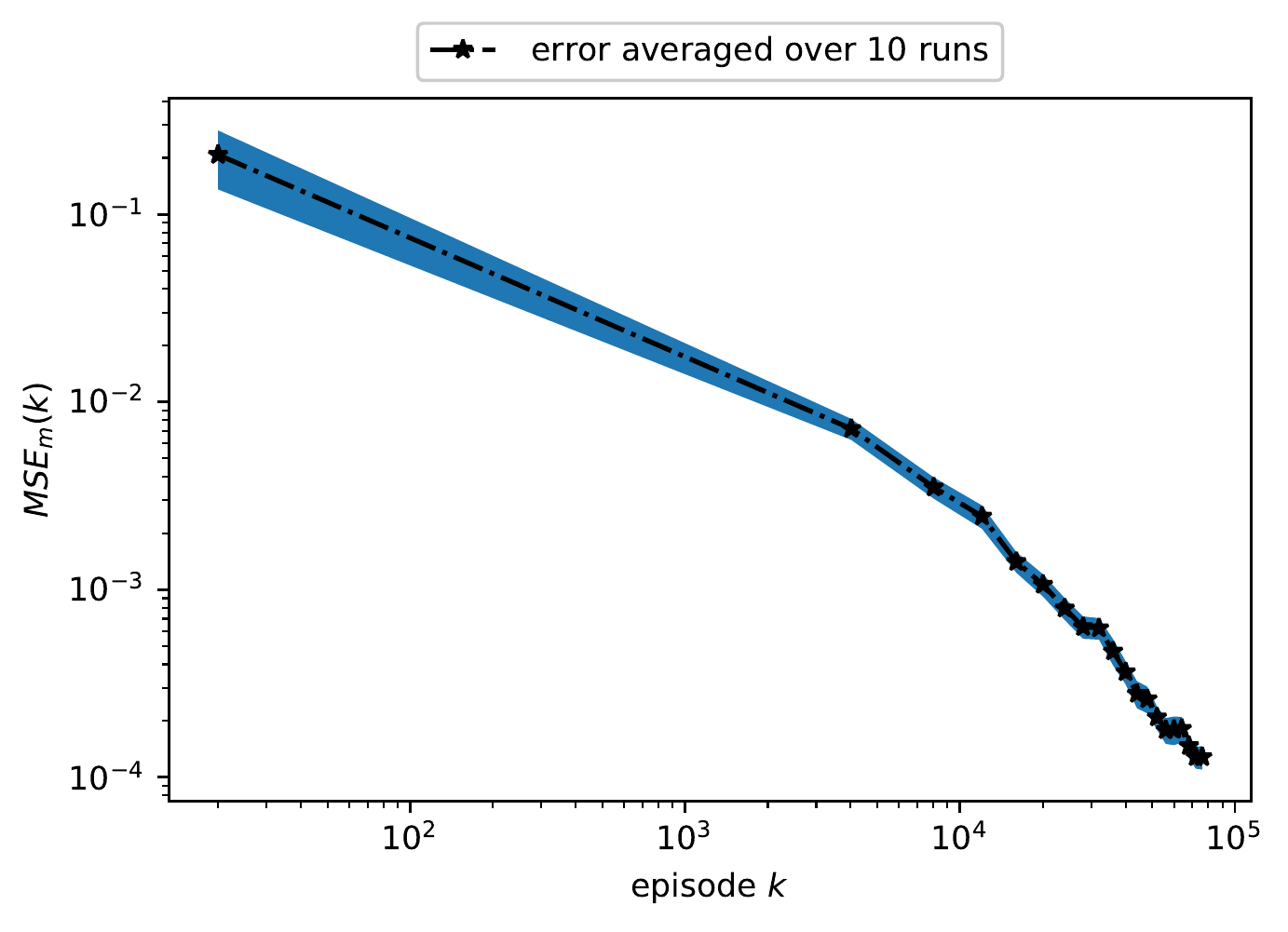}
  \caption{MFG: mean sqared error on $\hat m$ }
  \label{fig:error_mean_MFG}
\end{minipage}%
\begin{minipage}{.5\textwidth}
  \centering
  \includegraphics[width=.9\linewidth]{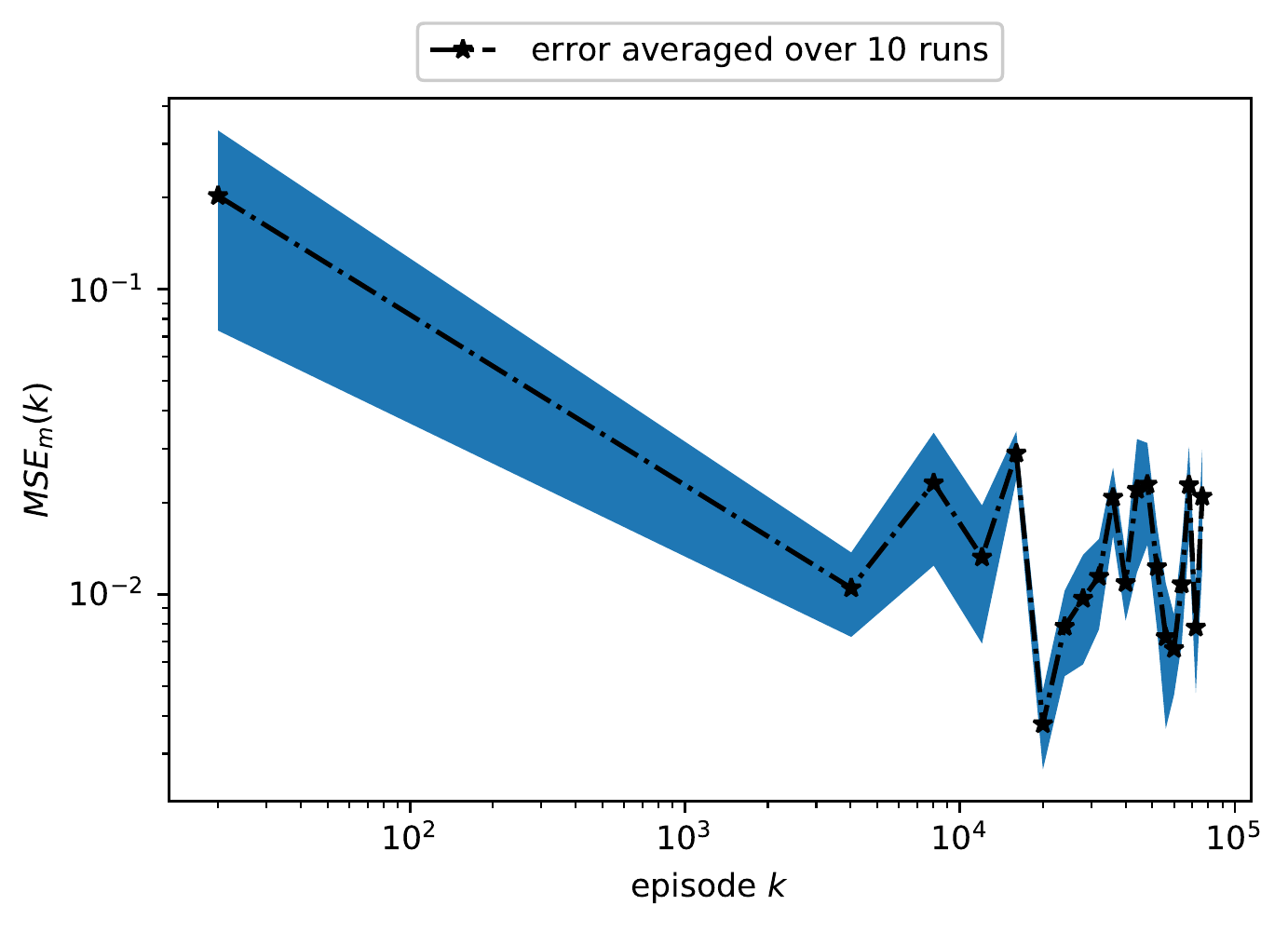}
  \caption{MFC: mean sqared error on $\hat m$ }
  \label{fig:error_mean_MFC}
\end{minipage}
  
\end{figure}

\subsubsection{Empirical analyses of the stopping criteria}

\textbf{Figures \ref{fig:stop_mu_MFG}, \ref{fig:stop_mu_MFC}, \ref{fig:stop_Q_MFG}, \ref{fig:stop_Q_MFC}:  stopping criteria.} The goal of the the U2-MF-QL is to obtain a good approximation of the optimal controls and the ergodic distribution. As presented in algorithm \ref{algo:U2MFQL}, the stopping criteria is based on the analyses of the progresses in learning the optimal $Q$ function and the ergodic distibution. The total variation and the $1,1$-norm between the start and the end of each episode is evaluated for the distribution and the $Q-$table respectively as follows
\begin{equation*}
    \delta(\mu_T^{k-1},\mu_T^{k}) = \sum_{x_i \in \mc{X}} \abs{\mu_T^{k}(x_i)-\mu_T^{k-1}(x_i)}, \quad \quad
    \Vert Q^k - Q^{k-1} \Vert_{1,1} = \sum_{i,j} \abs{Q_{i,j}^k - Q_{i,j}^{k-1}}.
\end{equation*}
The algorithm stops when the increments are not significant anymore based on a threshold given as input. The value of the threshold depends on the user's needs and it may be calibrated by a trial and error approach. The remaining plots show how these quantities decrease as the number of episodes increase.
The $x-$axis represents the number of episodes used for learning. The $y-$axis represents the value of the total variation.
\begin{figure}[H]
\centering
\begin{minipage}{.5\textwidth}
  \centering
  \includegraphics[width=.9\linewidth]{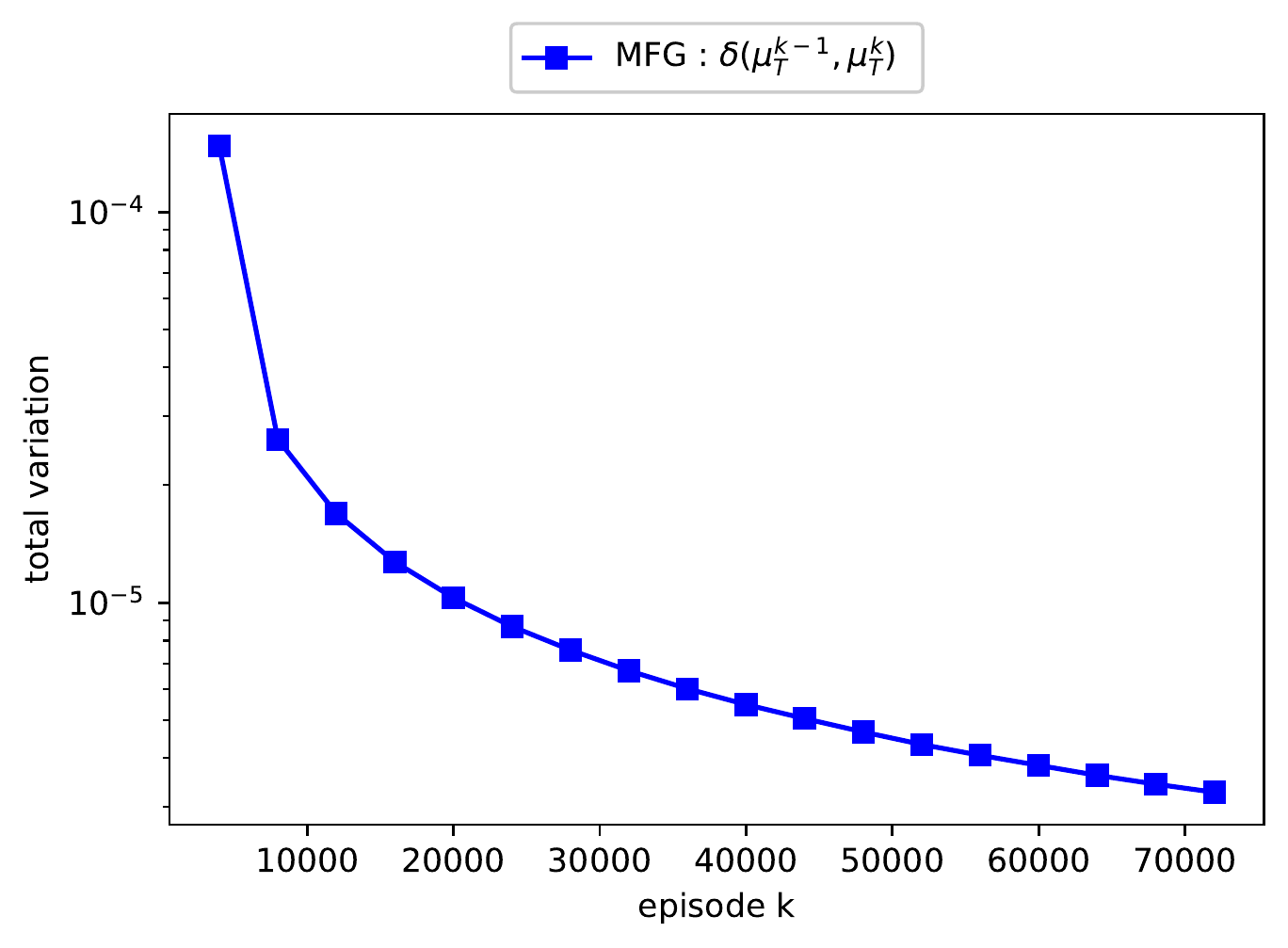}
  \caption{MFG: total variation on $\mu$ }
  \label{fig:stop_mu_MFG}
\end{minipage}%
\begin{minipage}{.5\textwidth}
  \centering
  \includegraphics[width=.9\linewidth]{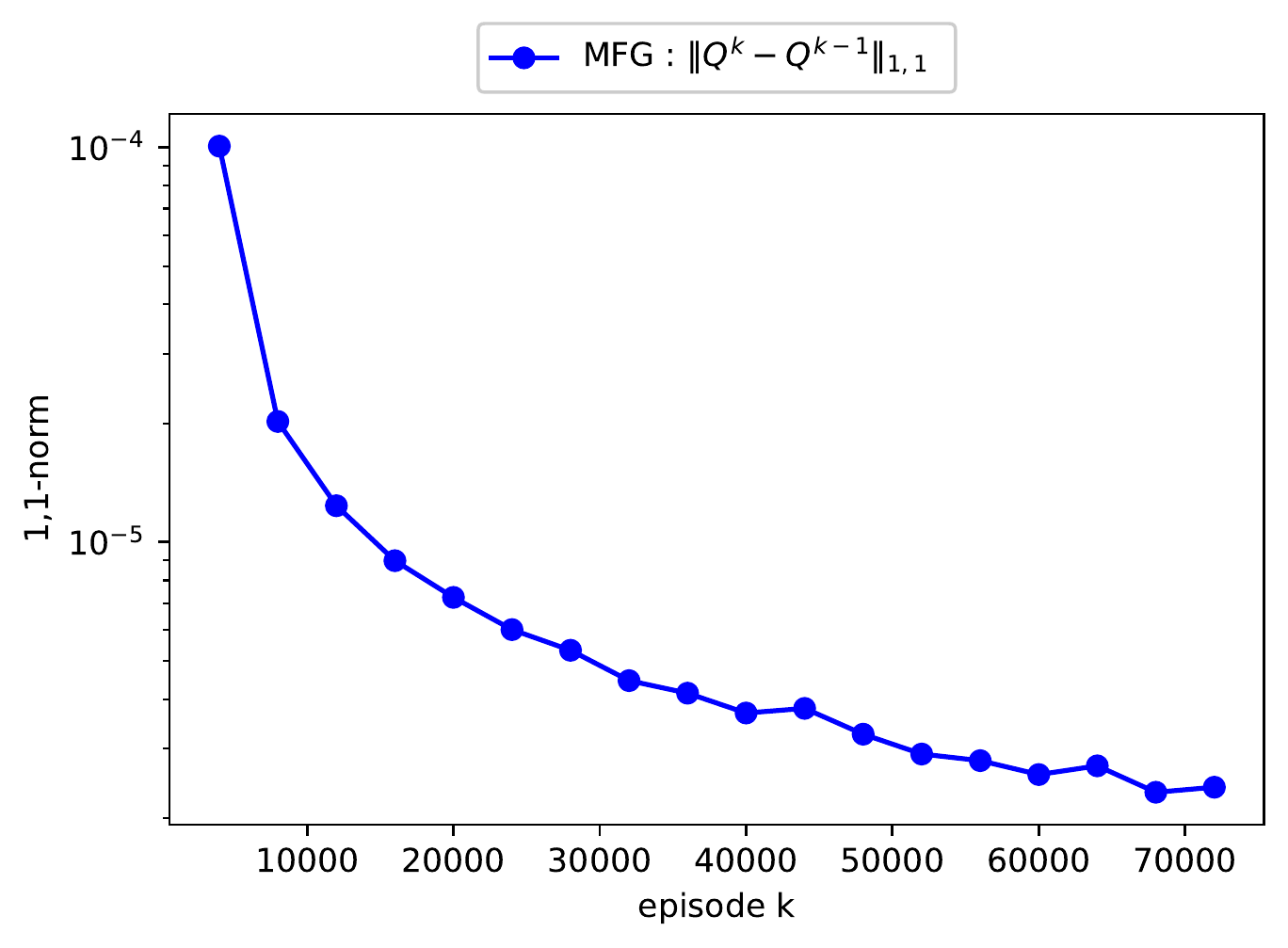}
 \caption{MFG: total variation on $Q$ }
   \label{fig:stop_Q_MFG}
\end{minipage}%
\end{figure}

\begin{figure}[H]
\centering
\begin{minipage}{.5\textwidth}
  \centering
  \includegraphics[width=.9\linewidth]{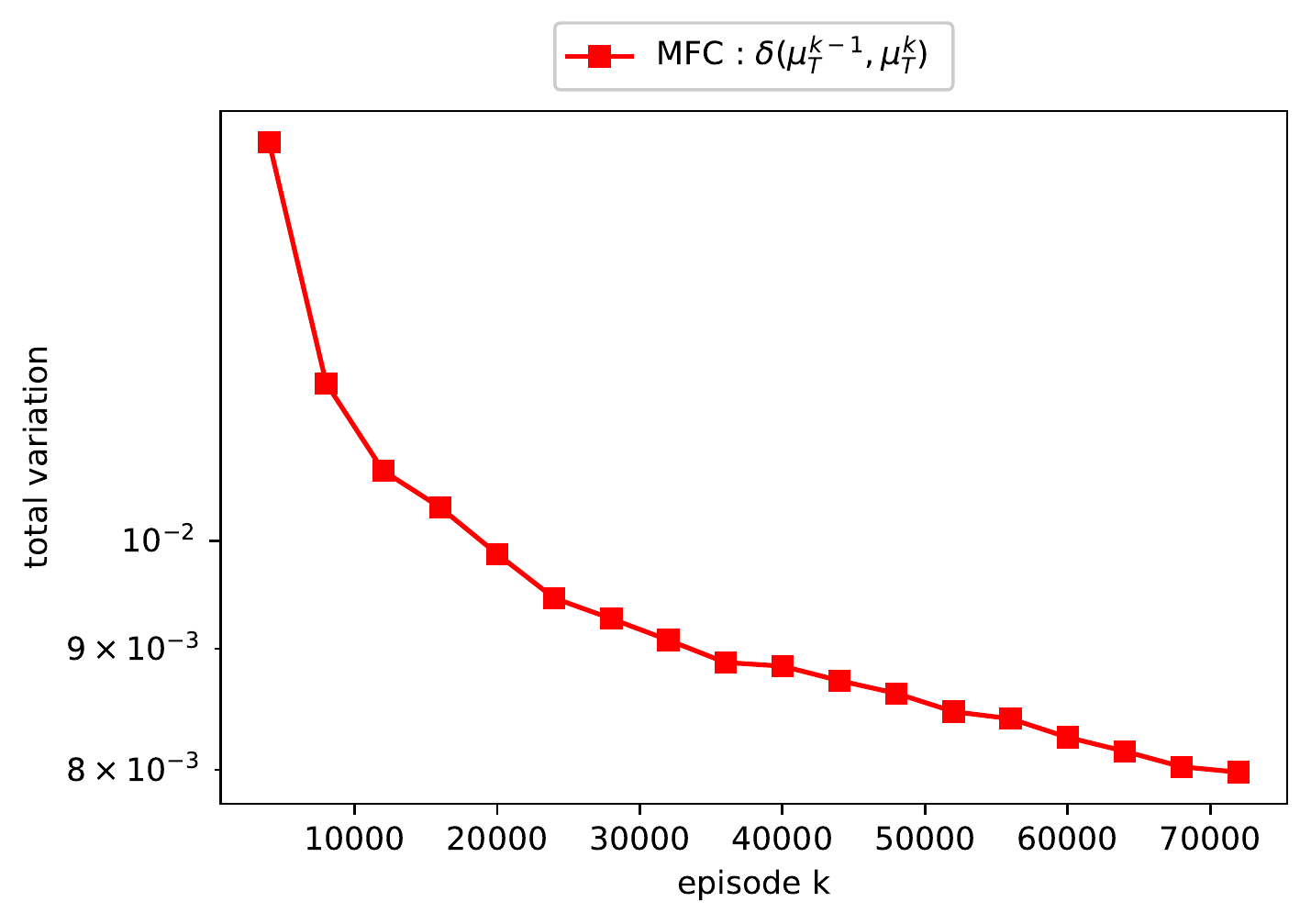}
  \caption{MFC: total variation on $\mu$ }
  \label{fig:stop_mu_MFC}
\end{minipage}%
\begin{minipage}{.5\textwidth}
  \centering
  \includegraphics[width=.9\linewidth]{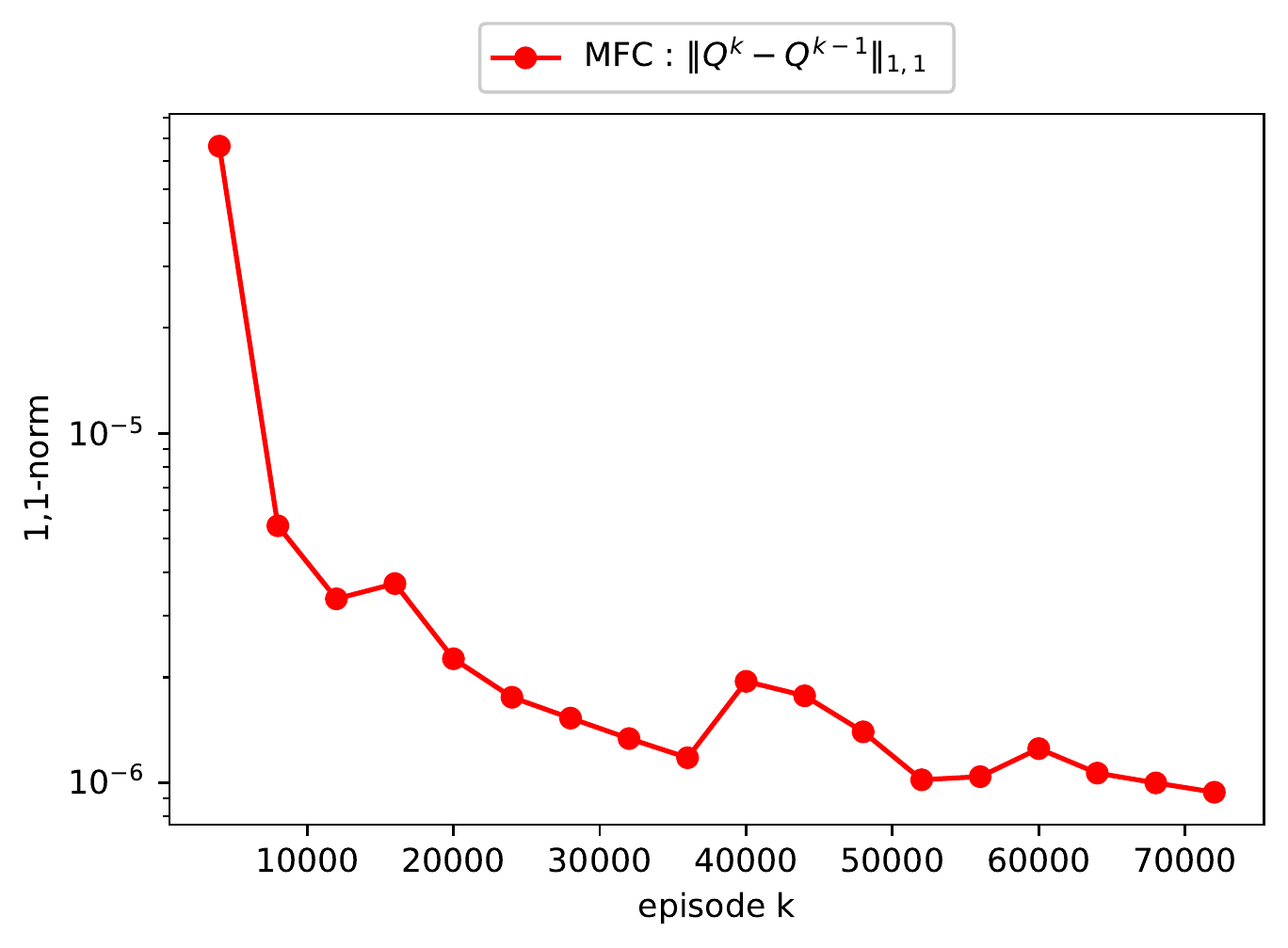}
 \caption{MFC: total variation on $Q$ }
   \label{fig:stop_Q_MFC}
\end{minipage}  
\end{figure}

\bibliographystyle{plain}
\bibliography{references}

\begin{thebibliography}{10}

\bibitem{anahtarci2020q}
Berkay Anahtarci, Can~Deha Kariksiz, and Naci Saldi.
\newblock Q-learning in regularized mean-field games.
\newblock {\em arXiv preprint arXiv:2003.12151}, 2020.

\bibitem{AngiuliFouqueLauriere-Handbook2021}
Andrea Angiuli, Jean-Pierre Fouque, and Mathieu Lauri\`ere.
\newblock Reinforcement learning for mean field games, with applications to
  economics.
\newblock 2021.

\bibitem{bellman2015applied}
Richard~E Bellman and Stuart~E Dreyfus.
\newblock {\em Applied dynamic programming}, volume 2050.
\newblock Princeton university press, 2015.

\bibitem{MR3134900}
Alain Bensoussan, Jens Frehse, and Sheung Chi~Phillip Yam.
\newblock {\em Mean field games and mean field type control theory}.
\newblock Springer Briefs in Mathematics. Springer, New York, 2013.

\bibitem{borkar1997stochastic}
Vivek~S Borkar.
\newblock Stochastic approximation with two time scales.
\newblock {\em Systems \& Control Letters}, 29(5):291--294, 1997.

\bibitem{MR2442439}
Vivek~S. Borkar.
\newblock {\em Stochastic approximation}.
\newblock Cambridge University Press, Cambridge; Hindustan Book Agency, New
  Delhi, 2008.
\newblock A dynamical systems viewpoint.

\bibitem{MR3608094}
Pierre Cardaliaguet and Saeed Hadikhanloo.
\newblock Learning in mean field games: the fictitious play.
\newblock {\em ESAIM Control Optim. Calc. Var.}, 23(2), 2017.

\bibitem{carmona2018probabilisticI-II}
Ren{\'e} Carmona and Fran{\c{c}}ois Delarue.
\newblock {\em Probabilistic Theory of Mean Field Games with Applications
  I-II}.
\newblock Springer, 2018.

\bibitem{carmona2019convergence-I}
Ren{\'e} Carmona and Mathieu Lauri{\`e}re.
\newblock {Convergence Analysis of Machine Learning Algorithms for the
  Numerical Solution of Mean Field Control and Games: I--The Ergodic Case}.
\newblock {\em arXiv preprint arXiv:1907.05980}, 2019.

\bibitem{carmona2019convergence-II}
Ren{\'e} Carmona and Mathieu Lauri{\`e}re.
\newblock {Convergence Analysis of Machine Learning Algorithms for the
  Numerical Solution of Mean Field Control and Games: II--The Finite Horizon
  Case}.
\newblock {\em arXiv preprint arXiv:1908.01613}, 2019.

\bibitem{CarmonaLauriereTan-2019-LQMFRL}
Ren{\'e} Carmona, Mathieu Lauri{\`e}re, and Zongjun Tan.
\newblock Linear-quadratic mean-field reinforcement learning: Convergence of
  policy gradient methods.
\newblock Preprint, 2019.

\bibitem{CarmonaLauriereTan-2019-MQFRL}
Ren{\'e} Carmona, Mathieu Lauri{\`e}re, and Zongjun Tan.
\newblock Model-free mean-field reinforcement learning: Mean-field {MDP} and
  mean-field {Q}-learning.
\newblock Preprint, 2019.

\bibitem{elie2020convergence}
Romuald Elie, Julien Perolat, Mathieu Lauri{\`e}re, Matthieu Geist, and Olivier
  Pietquin.
\newblock On the convergence of model free learning in mean field games.
\newblock In {\em in proc. of AAAI}, 2020.

\bibitem{even2003learning}
Eyal Even-Dar and Yishay Mansour.
\newblock Learning rates for q-learning.
\newblock {\em Journal of machine learning Research}, 5(Dec):1--25, 2003.

\bibitem{fouque2019deep}
Jean-Pierre Fouque and Zhaoyu Zhang.
\newblock Deep learning methods for mean field control problems with delay.
\newblock {\em Frontiers in Applied Mathematics and Statistics}, 6(11), 2020.

\bibitem{fu2019actor}
Zuyue Fu, Zhuoran Yang, Yongxin Chen, and Zhaoran Wang.
\newblock Actor-critic provably finds nash equilibria of linear-quadratic
  mean-field games.
\newblock {\em arXiv preprint arXiv:1910.07498}, 2019.

\bibitem{gu2019dynamic}
Haotian Gu, Xin Guo, Xiaoli Wei, and Renyuan Xu.
\newblock Dynamic programming principles for learning mfcs.
\newblock {\em arXiv preprint arXiv:1911.07314}, 2019.

\bibitem{gu2020qlearning}
Haotian Gu, Xin Guo, Xiaoli Wei, and Renyuan Xu.
\newblock Mean-field controls with {Q}-learning for cooperative {MARL}:
  Convergence and complexity analysis.
\newblock {\em arXiv preprint arXiv:2002.04131}, 2020.

\bibitem{guo2019learning}
Xin Guo, Anran Hu, Renyuan Xu, and Junzi Zhang.
\newblock Learning mean-field games.
\newblock In {\em Advances in Neural Information Processing Systems}, pages
  4966--4976, 2019.

\bibitem{Han-Hu-2020}
Jiequn Han and Ruimeng Hu.
\newblock Deep fictitious play for finding markovian nash equilibrium in
  multi-agent games.
\newblock {\em arxiv.org/abs/1912.01809}, 2020.

\bibitem{MR2352434}
Minyi Huang, Peter~E. Caines, and Roland~P. Malham{\'e}.
\newblock Large-population cost-coupled {LQG} problems with nonuniform agents:
  individual-mass behavior and decentralized {$\epsilon$}-{N}ash equilibria.
\newblock {\em IEEE Trans. Automat. Control}, 52(9):1560--1571, 2007.

\bibitem{MR2346927}
Minyi Huang, Roland~P. Malham{\'e}, and Peter~E. Caines.
\newblock Large population stochastic dynamic games: closed-loop
  {M}c{K}ean-{V}lasov systems and the {N}ash certainty equivalence principle.
\newblock {\em Commun. Inf. Syst.}, 6(3):221--251, 2006.

\bibitem{MR2295621}
Jean-Michel Lasry and Pierre-Louis Lions.
\newblock Mean field games.
\newblock {\em Jpn. J. Math.}, 2(1):229--260, 2007.

\bibitem{mguni2018decentralised}
David Mguni, Joel Jennings, and Enrique~Munoz de~Cote.
\newblock Decentralised learning in systems with many, many strategic agents.
\newblock In {\em Thirty-Second AAAI Conference on Artificial Intelligence},
  2018.

\bibitem{motte2019mean}
M{\'e}d{\'e}ric Motte and Huy{\^e}n Pham.
\newblock Mean-field markov decision processes with common noise and open-loop
  controls.
\newblock {\em arXiv preprint arXiv:1912.07883}, 2019.

\bibitem{perrin2020continuousfp}
Sarah Perrin, Julien P{\'e}rolat, Mathieu Lauri{\`e}re, Matthieu Geist, Romuald
  Elie, and Olivier Pietquin.
\newblock {Fictitious Play for Mean Field Games: Continuous Time Analysis and
  Applications}.
\newblock In preparation, 2020.

\bibitem{SubramanianMahajan-2018-RLstatioMFG}
Jayakumar Subramanian and Aditya Mahajan.
\newblock Reinforcement learning in stationary mean-field games.
\newblock In {\em Proceedings. 18th International Conference on Autonomous
  Agents and Multiagent Systems}, 2019.

\bibitem{sutton2018reinforcement}
Richard~S Sutton and Andrew~G Barto.
\newblock {\em Reinforcement learning: An introduction}.
\newblock MIT press, 2018.

\bibitem{watkins1989learning}
Christopher John Cornish~Hellaby Watkins.
\newblock {\em Learning from delayed rewards}.
\newblock PhD thesis, King's College, Cambridge, 1989.

\bibitem{xie2020provable}
Qiaomin Xie, Zhuoran Yang, Zhaoran Wang, and Andreea Minca.
\newblock Provable fictitious play for general mean-field games.
\newblock {\em arXiv preprint arXiv:2010.04211}, 2020.

\bibitem{yang2018deep}
Jiachen Yang, Xiaojing Ye, Rakshit Trivedi, Huan Xu, and Hongyuan Zha.
\newblock Deep mean field games for learning optimal behavior policy of large
  populations.
\newblock In {\em International Conference on Learning Representations}, 2018.

\bibitem{yang2018mean}
Yaodong Yang, Rui Luo, Minne Li, Ming Zhou, Weinan Zhang, and Jun Wang.
\newblock Mean field multi-agent reinforcement learning.
\newblock In {\em International Conference on Machine Learning}, pages
  5567--5576, 2018.

\end{thebibliography}

\appendix

\section{Theoretical solutions for the benchmark examples}\label{appendix : calculations}

In this appendix the solutions of the following benchmark problems are presented for the linear-quadratic models given by (\ref{eq: benchmark}).
\begin{itemize}
    \item[A.1]  Non-asymptotic Mean Field Game,
    \item[A.2]  Asymptotic Mean Field Game,
    \item[A.3]  Stationary Mean Field Game,
    \item[A.4]  Non-asymptotic Mean Field Control,
    \item[A.5]  Asymptotic Mean Field Control.
    \item[A.6]  Stationary Mean Field Control.
\end{itemize}
In particular, we check that the relations \eqref{controls-MFG} and \eqref{controls-MFC} are satisfied. The explicit formulas for the optimal controls (AMFG and AMFC) are used as benchmarks for our algorithm.

\subsection { Solution for non-asymptotic MFG}  
We present the solution for the following MFG problem
\begin{enumerate}
    
    \item Fix $\boldsymbol{m}=(m_t)_{t\geq 0} \subset \mR$ and solve the stochastic control problem:
    \begin{align*}
    \min_{\boldsymbol{ \alpha}}J^{\boldsymbol{m}}(\boldsymbol{ \alpha})&=\min_{\boldsymbol{ \alpha}} \mathbb{E}\left[\int_0^{\infty} e^{-\beta t}f(X^{\boldsymbol{ \alpha}}_t,\alpha_t,m_t)dt \right]=\\
&=\min_{\boldsymbol{ \alpha}}\E{\int_0^{+\infty}e^{-\beta t }\parentheses{\frac{1}{2}\alpha_t^2 + c_1 \parentheses{ X_t^{\boldsymbol{ \alpha}}- c_2 m_t}^2 + c_3 \parentheses{X_t^{\boldsymbol{ \alpha}}- c_4}^2+c_5 m_t^2}dt } ,    
\\
\text{subject to }&\\
 dX^{\boldsymbol{ \alpha}}_t&=\alpha_t dt +\sigma dW_t, \\
    X^{\boldsymbol{ \alpha}}_0&\sim\mu_0.    
\end{align*}
 
   \item Find the fixed point, $\boldsymbol{\hat m}=(\hat m_t)_{t\geq 0}$, such that $\E{X_t^{\boldsymbol{ \hat \alpha}}}=\hat m_t$ for all $ t\geq 0$.
\end{enumerate}
This problem can be solved by two equivalent approaches: PDE and FBSDEs. Both approaches start by solving the  problem defined by a finite horizon $T$. Then, the solution to the infinite horizon problem is obtained by taking the limit  $T$ goes to infinity. Let $V^{\boldsymbol{m}^T,T}(t,x)$ be the optimal value function for the finite horizon problem conditioned on $X_0=x$, i.e. 
\begin{equation*}
V^{\boldsymbol{m}^T,T}(t,x)=\inf_{\boldsymbol{ \alpha}}J^{\boldsymbol{m},x}(\boldsymbol{ \alpha})=\inf_{\boldsymbol{ \alpha}}\E{\int_t^{T}e^{-\beta s }f(X_s^{\boldsymbol{\alpha}},\alpha_s,m^T_s) ds\Big|X_0^{\boldsymbol{ \alpha}}=x}, \quad V^{\boldsymbol{m}^T,T}(T,x)=0.  
\end{equation*}  where $\boldsymbol{m}^T=\{ m_t^T \}_{0 \leq t \leq T}\subset \mR.$ Let's consider the following ansatz with its derivatives
\begin{equation}
\begin{split}
    V^{\boldsymbol{m}^T,T}(t,x) &= \Gamma_2^T (t) x^2 +  \Gamma_1^T (t) x  + \Gamma_0^T (t), 
    \\
    \partial_t V^{\boldsymbol{m}^T,T}(t,x) &= \dot{\Gamma}_2^T (t) x^2 +  \dot{\Gamma}_1^T (t) x  + \dot{\Gamma}_0^T (t), \\
    \partial_x V^{\boldsymbol{m}^T,T}(t,x) &= 2 \Gamma_2^T (t) x +  \Gamma_1^T (t) , \\
    \partial_{xx}V^{\boldsymbol{m}^T,T}(t,x) &=2 \Gamma_2^T (t), 
\end{split}
\end{equation}
Then, the HJB equation for the value function reads:
\begin{align*}
&\partial_t V^{\boldsymbol{m}^T,T} - \beta V^{\boldsymbol{m}^T,T} + \inf_{\alpha} \{\mathcal{A}^X V^{\boldsymbol{m}^T,T} + f(x,\alpha,m^T)\}\\
&=\partial_t V^{\boldsymbol{m}^T,T} - \beta V^{\boldsymbol{m}^T,T} 
\\
&\qquad+ \inf_{\alpha} \left\{\alpha \partial_x V^{\boldsymbol{m}^T,T} +\frac{1}{2}\sigma^2  \partial_{xx} V^{\boldsymbol{m}^T,T} + \frac{1}{2}\alpha^2 + c_1 (x- c_2 m^T)^2 + c_3 (x - c_4)^2 +c_5 (m^T)^2\right\}\\
&=\partial_t V^{\boldsymbol{m}^T,T} - \beta V^{\boldsymbol{m}^T,T} 
\\
&\qquad+ \left \{ - {\partial_x V^{\boldsymbol{m}^T,T}}^2 +\frac{1}{2}\sigma^2  \partial_{xx}V^{\boldsymbol{m}^T,T} + \frac{1}{2}{\partial_x V^{\boldsymbol{m}^T,T}}^2 + c_1  (x- c_2 m^T)^2+ c_3 (x - c_4 )^2+ c_5 (m^T)^2\right \}\\
&=\partial_t V^{\boldsymbol{m}^T,T} - \beta V^{\boldsymbol{m}^T,T} -  \frac{1}{2}{\partial_x V^{\boldsymbol{m}^T,T}}^2 +\frac{1}{2}\sigma^2  \partial_{xx} V^{\boldsymbol{m}^T,T} + c_1 (x-c_2 m^T)^2 + c_3 (x - c_4)^2 + c_5 (m^T)^2= 0, 
\end{align*}
where in the third line we evaluated the infimum at $\hat\alpha^{T}= -V^{\boldsymbol{m}^T,T}_x$. The following ODEs system is obtained by replacing the ansatz and its derivatives in the HJB equation:
\begin{equation} \label{system_conditions_ODE_MFG}
\begin{cases}
{\dot \Gamma}_2^T -2({{ \Gamma}^T_2})^2 - \beta { \Gamma}^T_2 + c_1 + c_3 =0, \quad  &{ \Gamma}_2^T(T) = 0, \\
{\dot \Gamma}^T_1 = (2  { \Gamma}^T_2 + \beta ) { \Gamma}^T_1 + 2 c_1 c_2 m^T + 2 c_3 c_4, \quad  &{ \Gamma}^T_1(T)=0, \\
{\dot \Gamma}^T_0 = \beta { \Gamma}^T_0 + \frac{1}{2}({{ \Gamma}^T_1})^2 - \sigma^2 { \Gamma}^T_2 -c_3 {c_4}^2  - (c_1 {c_2}^2 +c_5) ({m^T})^2 , \quad  &{ \Gamma}^T_0(T) = 0,\\
\dot m^T = - 2 {\Gamma}^T_2  m^T - { \Gamma}^T_1, \quad  &m^T(0)= \E{\mu_0}=m_0,\\
\end{cases}
\end{equation}
where the last equation is obtained by considering the expectation of $X_t^{\boldsymbol{\alpha}}$ after replacing $\hat\alpha^{T} = -\partial_x V^{\boldsymbol{m}^T,T} = - (\Gamma^T_2 x + \Gamma^T_1)$.
The first equation is a Riccati equation. In particular, the solution $\Gamma^T_2$ converges to $\hat\Gamma_2=\frac{-\beta + \sqrt{\beta^2+8(c_1 + c_3)}}{4}$ as $T$ goes to infinity. The second and fourth ODEs are coupled and they can be written in matrix notation as 

\begin{gather} \label{eqn: non_hom_system}
 \dot{\wideparen{\begin{pmatrix} 
 m^T \\ \Gamma^T_1
 \end{pmatrix}}}
 = \begin{bmatrix}
 - 2 \Gamma^T_2  &  -1  \\
  2 c_1 c_2 & 2 \Gamma^T_2 + \beta   
   \end{bmatrix}
   \begin{pmatrix} 
   m^T \\ \Gamma^T_1
    \end{pmatrix} + \begin{pmatrix} 
   0 \\ 2 c_3 c_4
    \end{pmatrix}, \quad
    \begin{pmatrix} 
   m^T(0) \\ \Gamma^T_1(T)
    \end{pmatrix}
    =
    \begin{pmatrix} 
   m_0 \\ 0
    \end{pmatrix}.
\end{gather}
We start by solving the homogeneous equation, i.e.
\begin{gather}\label{matr_system}
 \dot{\wideparen{\begin{pmatrix} 
 m^T \\ \Gamma_1^T
 \end{pmatrix}}}
 = K_t^T \begin{pmatrix} 
   m^T \\ \Gamma^T_1
    \end{pmatrix} \coloneqq
 \begin{bmatrix}
 - 2 \Gamma^T_2  &  -1  \\
  2 c_1 c_2 & 2 \Gamma^T_2 + \beta   
   \end{bmatrix}
   \begin{pmatrix} 
   m^T \\ \Gamma^T_1
    \end{pmatrix} 
   , \quad
    \begin{pmatrix} 
   m^T(0) \\ \Gamma^T_1(T)
    \end{pmatrix}
    =
    \begin{pmatrix} 
   m_0 \\ 0
    \end{pmatrix}.
\end{gather}

We introduce the propagator $P^T$, i.e.
   \begin{gather} 
 {\begin{pmatrix} 
 m^T \\ \Gamma_1^T
 \end{pmatrix}}
 = P^T_t
    \begin{pmatrix} 
   m^T(0) \\ \Gamma_1^T(0)
    \end{pmatrix}.
\end{gather}
By deriving $\begin{pmatrix} 
 m^T \\ \Gamma_1^T
 \end{pmatrix}$ and expressing the initial conditions in terms of the inverse of $P^T$ and 
$\begin{pmatrix} 
 m^T \\ \Gamma_1^T
 \end{pmatrix}$, we obtain 
  \begin{gather}
 \dot{\wideparen{\begin{pmatrix} 
 m^T \\ \Gamma_1^T
 \end{pmatrix}}}
 = \dot{P^T_t}
    \begin{pmatrix} 
   m^T(0) \\ \Gamma_1^T(0)
    \end{pmatrix}=\dot{P^T_t}
    ({P^T_t})^{-1}\begin{pmatrix} 
   m^T \\ \Gamma_1^T
    \end{pmatrix}.
\end{gather} 
By comparing the last system with \eqref{matr_system}, we obtain 
\begin{equation}
\begin{cases}
\dot{P^T_t} &= K^T_t P^T_t\\
 P^T_0 &= \mbb{I}_2
\end{cases}
\end{equation}
where $\mbb{I}_2$ is the identity matrix in dimension 2. 
The solution is given by $P^T_t=e^{\int_0^t K^T_s ds } \coloneqq e^{L^T_t}.$ In particular, the exponent is equal to
\begin{equation}
L^T_t= \int_0^t K^T_s ds =  \begin{bmatrix}
 - 2  \int_0^t \Gamma_2^T(s) ds &  -t  \\
  2 c_1 c_2 t & 2  \int_0^t \Gamma_2^T(s) ds  + \beta t  
   \end{bmatrix}= \begin{bmatrix}
   g_t^T & d_t \\ b_t & a_t^T
   \end{bmatrix}.
\end{equation}
We evaluate the exponential $P^T(t)= e^{L^T_t}$ by using the Taylor's expansion and diagonalizing the matrix $L^T_t$. The eigenvalues/eigenvectors of $L^T_t$ are given by
\begin{equation}
\lambda^T_{1\backslash 2,t} \coloneqq \frac{a_t^T+g_t^T \pm \sqrt{(a_t^T-g_t^T)^2 + 4b_t d_t}}{2}, \quad v^T_{1,t}\coloneqq \begin{pmatrix} 
  d_t \\ \lambda^T_{1,t} - g_t^T
    \end{pmatrix}, \quad v^T_{2,t}\coloneqq \begin{pmatrix} 
  d_t \\ \lambda^T_{2,t} - g^T_t
    \end{pmatrix}.
\end{equation}
$P_t$ is obtained by
\begin{equation}
\begin{split}
P^T_t &= \begin{pmatrix} 
 p^T_t(1,1) &  p^T_t(1,2) \\
 p^T_t(2,1) &  p^T_t(2,2) 
    \end{pmatrix}\\
    &= e^{L^T_t}= \sum_{k=0}^{\infty}   \begin{bmatrix} 
 v^T_{1,t} & v^T_{2,t}
    \end{bmatrix} \frac{\begin{pmatrix} 
 \lambda^T_{1,t} & 0 \\ 0 &   \lambda^T_{2,t}
    \end{pmatrix}^k} {k! }\begin{bmatrix} 
 v^T_{1,t} & v^T_{2,t}
    \end{bmatrix}^{-1} \coloneqq \\
    & \coloneqq S^T_t \sum_{k=0}^{\infty} \frac{{D^T_t}^k}{k!} ({S^T_t})^{-1}=\\
    &= S^T_t \begin{pmatrix} 
 e^{\lambda^T_{1,t}} & 0 \\ 0 &    e^{\lambda^T_{2,t}}
    \end{pmatrix}({S^T_t})^{-1}=\\
    &=\frac{1}{d_t(\lambda^T_{2,t} - \lambda^T_{1,t})} 
    \begin{pmatrix} 
 d_t e^{\lambda^T_{1,t}}(\lambda^T_{2,t} - g^T_t)+ d_t e^{\lambda^T_{2,t}}(g^T_t-\lambda^T_{1,t}) & d_t^2(e^{\lambda^T_{2,t}}-e^{\lambda^T_{1,t}}) \\ 
 (\lambda^T_{1,t} - g^T_t) (\lambda^T_{2,t} - g^T_t) (e^{\lambda^T_{1,t}}-e^{\lambda^T_{2,t}}) &  
 d_t e^{\lambda^T_{2,t}}(\lambda^T_{2,t} - g^T_t) + d_t e^{\lambda^T_{1,t}}(g^T_t-\lambda^T_{1,t})
    \end{pmatrix}.
    \end{split}
\end{equation}
In order to solve the non homogeneous case, we introduce an extra term $\begin{pmatrix} 
   h_1^T \\ h_2^T
    \end{pmatrix}$, i.e.
 \begin{gather} 
 {\begin{pmatrix} 
 m^T \\ \Gamma_1^T
 \end{pmatrix}}
 = P^T_t
    \begin{pmatrix} 
   h^T_1 \\ h^T_2
    \end{pmatrix}.
\end{gather}

By deriving $ {\begin{pmatrix} 
 m^T \\ \Gamma_1^T
 \end{pmatrix}}$, we obtain

\begin{gather}\label{eqn: var_const}
 \dot{\wideparen{\begin{pmatrix} 
 m^T \\ \Gamma_1^T
 \end{pmatrix}}}
 = \dot{P}^T_t
    \begin{pmatrix} 
   h^T_1 \\ h^T_2
    \end{pmatrix}+P_t^T \dot{\wideparen{\begin{pmatrix} 
 h^T_1 \\ h^T_2
 \end{pmatrix}}}={K_t^T}{P^T_t}
    \begin{pmatrix} 
   h_1^T \\ h_2^T
    \end{pmatrix}+{P^T_t}\dot{\wideparen{\begin{pmatrix} 
 h^T_1 \\ h^T_2
 \end{pmatrix}}}= K_t^T  {\begin{pmatrix} 
 m^T_t \\ \Gamma_1^T
 \end{pmatrix}}+{P^T_t}\dot{\wideparen{\begin{pmatrix} 
 h^T_1 \\ h^T_2
 \end{pmatrix}}}.
\end{gather}
By comparing (\ref{eqn: non_hom_system})  with (\ref{eqn: var_const}), we obtain
\begin{gather}
    \dot{\wideparen{\begin{pmatrix} 
 h^T_1 \\ h^T_2
 \end{pmatrix}}}=(P_t^T)^{-1}\begin{pmatrix} 
 0 \\ 2c_4 c_4
 \end{pmatrix}=\frac{1}{|P_t^T|}\begin{pmatrix} 
 p_t^T(2,2) & -p_t^T(1,2) \\ -p_t^T(2,1) & p_t^T(1,1)
 \end{pmatrix}\begin{pmatrix} 
 0 \\ 2c_3 c_4
 \end{pmatrix}.
\end{gather}
By integration we obtain 
\begin{equation}\label{eqn : h's functions}
\begin{split}
    h_1^T(t)&=h_1^T(0)-2c_3c_4\int_0^t \frac{p_s^T(1,2)}{|P_s^T|}ds,\\
    h_2^T(t)&=h_2^T(0)+2c_3c_4\int_0^t  \frac{p_s^T(1,1)}{|P_s^T|}ds,
\end{split}
\end{equation}
where $h_1^T(0)=m_0$ and $h_2^T(0)=\Gamma_1^T(0)$.

We use the terminal condition $\Gamma_1^T(T)=0$ to obtain an evaluation of $h_2^T(0)=\Gamma_1^T(0)$ in terms of $P^T_T$ and $m_0$, i.e. 
\begin{equation} \label{eq:F_0}
\begin{split}
\Gamma_1^T(T)&=p^T_T(2,1)h^T_1(T) + p^T_T(2,2) h^T_2(T)=0, \\
\Gamma_1^T(T)&=p^T_T(2,1)\parentheses{m_0-2c_3c_4\int_0^T \frac{p_s^T(1,2)}{|P_s^T|}ds} + p^T_T(2,2) \parentheses{\Gamma_1^T(0)+2c_3c_4\int_0^T \frac{p_s^T(1,1)}{|P_s^T|} ds}=0, \\
\Gamma_1^T(0)&= -\frac{p^T_T(2,1)}{p^T_T(2,2)} \parentheses{m_0-2c_3c_4\int_0^T \frac{p_s^T(1,2)}{|P_s^T|}ds} -2c_3c_4\int_0^T \frac{p_s^T(1,1)}{|P_s^T|} ds .
\end{split}
\end{equation}
In order to evaluate the limit of $\Gamma_1^T(0)$ as $T$ goes to infinity, we  analyze the different terms separately. First, we evaluate the following limit:
\begin{equation}
\lim_{T\to \infty}\frac{1}{T}\int_0^T \Gamma_2^T(s) ds =\lim_{T\to \infty} \Gamma_2^T(s_1)= \hat\Gamma_2, \quad s_1 \in [0,T] ,
\end{equation}
where we applied the mean value integral theorem and $\hat\Gamma_2=\frac{-\beta +\sqrt{\beta^2+8(c_1+c_3)}}{4}$ is the limit of the solution of the Riccati equation obtained previously, i.e. $\hat\Gamma_2=\lim_{T\to\infty} \Gamma_2^T(s).$ We recall that 
$$\lambda^T_{2,T}-\lambda^T_{1,T}=\sqrt{(a^T_T-g^T_T)^2+4b^T_T d_T}= T \sqrt{\parentheses{\frac{4}{T}\int_0^T \Gamma^T_2(s)ds + \beta}^2 - 8 c_1 c_2 }>0$$ which goes to infinity as $T$ goes to $\infty$ when the term under square root is well defined.  We observe that 
\begin{equation}
\begin{split}
\hat{g}_t&\coloneqq  \lim_{T\to \infty}g^T_t = \lim_{T\to \infty}- 2  \int_0^t \Gamma_2^T(s) ds =  - 2  \hat\Gamma_2 t \coloneqq g t,\\ 
b_t&=2c_1c_2 t,\\ 
\hat{a}_t&\coloneqq \lim_{T\to \infty}a^T_t = \lim_{T\to \infty} 2  \int_0^t \Gamma_2^T(s) ds + \beta t= 2 \hat\Gamma_2 t   + \beta t ,\\  
d_t &= - t,\\
\hat\lambda_{1\backslash 2,t}&\coloneqq \lim_{T\to \infty}\lambda^T_{1\backslash 2,t}  = \frac{\hat{a}_t+\hat{g}_t \pm \sqrt{(\hat{a}_t-\hat{g}_t)^2 + 4b_t d_t}}{2} = t \frac{\beta \pm \sqrt{(4 \hat\Gamma_2 + \beta)^2-8c_1c_2}}{2}\coloneqq t \lambda_{1\backslash 2} ,  \\
\hat{P}_t&\coloneqq \lim_{T\to \infty}P^T_t = \\
&=\frac{1}{d_t(\hat\lambda_{2,t} - \hat\lambda_{1,t})} 
\begin{pmatrix} 
 d_t e^{\hat\lambda_{1,t}}(\hat\lambda_{2,t} - \hat{g}_t)+ d_t e^{\hat\lambda_{2,t}}(\hat{g}_t-\hat\lambda_{1,t}) & d_t^2(e^{\hat\lambda_{2,t}}-e^{\hat\lambda_{1,t}}) \\ 
 (\hat\lambda_{1,t} - \hat{g}_t) (\hat\lambda_{2,t} - \hat{g}_t) (e^{\hat\lambda_{1,t}}-e^{\hat\lambda_{2,t}}) &  
 d_t e^{\hat\lambda_{2,t}}(\hat\lambda_{2,t} - \hat{g}_t) + d_t e^{\hat\lambda_{1,t}}(\hat{g}_t-\hat\lambda_{1,t})
    \end{pmatrix}.
\end{split}
\end{equation}
To evaluate $\hat\Gamma_1(0)=\lim_{T\to \infty} \Gamma^T_1(0)$, we study  the limit of the remaining terms:

\begin{equation}\label{eqn : limits}
    \begin{split}
        \lim_{T\mapsto \infty}-\frac{p^T_T(2,1)}{p^T_T(2,2)}&=\lim_{T\mapsto \infty}\frac{(\lambda^T_{1,T} - g^T_T) (\lambda^T_{2,T} - g^T_T) (e^{\lambda^T_{2,T}}-e^{\lambda^T_{1,T}})}
{ d_T e^{\lambda^T_{2,T}}(\lambda^T_{2,T} - g^T_T) + d_T e^{\lambda^T_{1,T}}(g^T_T-\lambda^T_{1,T})}=\\
&=\lim_{T\mapsto \infty}\frac{1}{\frac{d_T}{(\lambda^T_{1,T}-g^T_T)(1-e^{\lambda^T_{1,T}-\lambda^T_{2,T}})}+\frac{d_T}{(\lambda^T_{2,T}-g^T_T)(1-e^{\lambda^T_{2,T}-\lambda^T_{1,T}})}}=\\
&= -(\lambda_1 -g)=\\
&=-(\lambda_1 + 2 \hat\Gamma_2),\\
 \lim_{T\mapsto \infty}\int_0^T \frac{p_s^T(1,2)}{|P_s^T|}ds&= \lim_{T\mapsto \infty}\int_0^T \frac{d_s(e^{\lambda^T_{2,s}}-e^{\lambda^T_{1,s}})}{(\lambda^T_{2,s}-\lambda^T_{1,s})(e^{\lambda^T_{1,s}+\lambda^T_{2,s}})} ds = \\
 &=\frac{1}{\lambda_2-\lambda_1}\parentheses{\frac{1}{\lambda_2}-\frac{1}{\lambda_1}}\\
 \lim_{T\mapsto \infty}\int_0^T \frac{p_s^T(1,1)}{|P_s^T|} ds &= \lim_{T\mapsto \infty}  \int_0^T\frac{1}{e^{\lambda^T_{1,s}+\lambda^T_{2,s}}} \parentheses{e^{\lambda^T_{1,s}}\frac{\lambda_{2,s}^T-g_s^T}{\lambda^T_{2,s}-\lambda^T_{1,s}} +e^{\lambda^T_{2,s}} \frac{g_s^T-\lambda_{1,s}^T}{\lambda^T_{2,s}-\lambda^T_{1,s}} }ds=\\
 &=\frac{\lambda_2-g}{\lambda_2(\lambda_2-\lambda_1)}+\frac{g-\lambda_1}{\lambda_1(\lambda_2-\lambda_1)}.
    \end{split}
\end{equation}
Finally, the value of $\hat\Gamma_1(0)$ is given by
\begin{equation}
\hat\Gamma_1(0)  = - (\lambda_1 - g) m_0 -2\frac{c_3c_4}{\lambda_2}.
\end{equation}

Given $\hat\Gamma_1(0)$, we evaluate the limit as $T$ goes to $\infty$ of \eqref{eqn : h's functions}, i.e.

\begin{equation}
\begin{split}
    h_1(t)\coloneqq \lim_{T\mapsto \infty} h_1^T(t)&=m_0- 2c_3c_4 \lim_{T\mapsto \infty} \int_0^t \frac{ p_s^T(1,2)}{|P_s^T|}ds = \\
    &=m_0+ 2 \frac{c_3 c_4}{\lambda_2 - \lambda_1} \parentheses{\frac{1}{\lambda_2}e^{-t \lambda_2}-\frac{1}{\lambda_1}e^{-t \lambda_1}+\frac{1}{\lambda_1}-\frac{1}{\lambda_2}} ,\\
    h_2(t)\coloneqq \lim_{T\mapsto \infty} h_2^T(t)&=\lim_{T\mapsto \infty} \parentheses{\Gamma_1^T(0)+2c_3c_4  \int_0^t \frac{p_s^T(1,1)}{|P_s^T|} ds} =\\
   &= \hat\Gamma_1(0)+ 2 \frac{c_3 c_4}{\lambda_2 - \lambda_1} \parentheses{\frac{\lambda_2-g}{\lambda_2}(1-e^{-t\lambda_2})+\frac{g-\lambda_1}{\lambda_1}(1-e^{-t\lambda_1})}.
\end{split}
\end{equation}

We can conclude that
\begin{equation} %
\begin{split}
 \hat{m}_t&=\lim_{T\to\infty}m^T_t=\\
&= \hat{p}_t(1,1) h_1(t) + \hat{p}_t(1,2)h_2(t)=\\
&= \parentheses{m_0+2 \frac{c_3c_4}{\lambda_2-\lambda_1}\parentheses{\frac{1}{\lambda_1}-\frac{1}{\lambda_2}} }e^{t\lambda_{1}}+2 \frac{c_3c_4}{\lambda_2-\lambda_1}\parentheses{\frac{1}{\lambda_2}-\frac{1}{\lambda_1}},\\
\hat\Gamma_1(t)&=\lim_{T\to\infty}\Gamma_1^T(t)=\\
&= \hat{p}_t(2,1) h_1(t) + \hat{p}_t(2,2)h_2(t)=\\
&=m_0 (g-\lambda_1) e^{t\lambda_{1}}+2\frac{c_3c_4}{\lambda_2-\lambda_1}\parentheses{\frac{\lambda_2-g}{\lambda_2}-\frac{\lambda_1-g}{\lambda_1}}.\\
\end{split}
\end{equation}

Finally, the third ODE in \eqref{system_conditions_ODE_MFG} can be solved by plugging in the solution of the previous ones and integrating. Since our interest is into the evolution of the mean and the control function, we omit these calculations, but we recall that: 
\begin{equation}\label{formula-alpha-hat}
\hat\alpha_t=-(\hat\Gamma_2 x+\hat\Gamma_1(t)),
\quad \hat\Gamma_2=\frac{-\beta + \sqrt{\beta^2+8(c_1 + c_3)}}{4},
\end{equation}
and we observe that
\begin{equation}\label{formula--alpha-hat-infty}
 \lim_{t\to\infty}\hat\alpha_t=-(\hat\Gamma_2 x+\hat\Gamma_1), \quad   \hat\Gamma_1=-\frac{4c_1c_2\hat\Gamma_2}{\lambda_2} 
 =\frac{c_3c_4\hat\Gamma_2}{2(c_1+c_3-c_1c_2)}.
\end{equation}

\subsection { Solution for Asymptotic MFG} 
The asymptotic version of the problem presented above is given by:
 \begin{enumerate}
    	\itemsep -2pt
    	\item Fix $m \in \mR $ and solve the stochastic control problem:
    	\begin{align*}
    	\min_{\boldsymbol{ \alpha}}J^{m}(\boldsymbol{ \alpha})&=\min_{\boldsymbol{ \alpha}} \mathbb{E}\left[\int_0^{\infty} e^{-\beta t}f(X^{\boldsymbol{ \alpha}}_t,\alpha_t,m)dt \right]=\\
    	&=\min_{\boldsymbol{ \alpha}}\E{\int_0^{\infty}e^{-\beta t }\parentheses{\frac{1}{2}\alpha_t^2 + c_1 \parentheses{ X_t^{\boldsymbol{ \alpha}}-c_2 m}^2 + c_3 \parentheses{ X_t^{\boldsymbol{ \alpha}}-c_4}^2 + c_5 m^2}dt},   
    	\\
    	\text{subject to: }&
    	\quad  dX^{\boldsymbol{ \alpha}}_t=\alpha_t dt +\sigma dW_t, \quad X^{\boldsymbol{ \alpha}}_0\sim\mu_0.    
    	\end{align*}
   \item Find the fixed point, $\hat m$, such that $\hat m = \lim_{t \to +\infty} \E{  X^{\hat\alpha,\hat m}_t}$.
\end{enumerate}

Let $V^m(x)$ be the optimal value function given $m \in \mR$ and conditioned on $X_0=x$, i.e.
\begin{equation*}
V^m(x)=\inf_{\boldsymbol{ \alpha}}J^{m,x}(\boldsymbol{ \alpha})=\inf_{\boldsymbol{ \alpha}}\E{\int_0^{+\infty}e^{-\beta t }\parentheses{\frac{1}{2}\alpha_t^2 + c_1 \parentheses{ X_t^{\boldsymbol{ \alpha}}-c_2 m}^2 + c_3 \parentheses{ X_t^{\boldsymbol{ \alpha}}-c_4}^2 + c_5 m^2}\Big|X_0^{\boldsymbol{ \alpha}}=x}.  
\end{equation*}
We consider the following ansatz with its derivatives with respect to $x$:
\begin{align*}
V^m(x)&=\Gamma_2 x^2 + \Gamma_1 x +\Gamma_0, \\ 
\dot V^m(x)&= 2\Gamma_2 x + \Gamma_1, \\
\ddot V^m(x)&=2\Gamma_2. 
\end{align*}
Let's consider the HJB equation
\begin{align*}
&\beta V^m(x) - \inf_{\alpha} \{\mathcal{A}^X V^m(x) + f(x,\alpha,m)\}\\
&=\beta V^m(x) - \inf_{\alpha} \left\{\alpha \dot V(x) +\frac{1}{2}\sigma^2  \ddot V^m(x) + \frac{1}{2}\alpha^2 + c_1 (x- c_2 m)^2 + c_3 (x - c_4)^2 +c_5 m^2\right\}\\
&=\beta V^m(x) - \left \{ - ({\dot V^m})^2(x) +\frac{1}{2}\sigma^2  \ddot V^m(x) + \frac{1}{2}({\dot V^m})^2(x) + c_1  (x- c_2 m)^2+ c_3 (x - c_4 )^2+ c_5 m^2\right \}\\
&=\beta V^m(x)  + \frac{1}{2}({\dot V^m})^2(x) -\frac{1}{2}\sigma^2  \ddot V^m(x) - c_1 (x-c_2 m)^2 - c_3 (x - c_4)^2 - c_5 m^2= 0, 
\end{align*}
where in the third line we evaluated the infimum at $\hat\alpha(x)= -\dot V^m(x)$. Replacing the ansatz and its derivatives in the HJB equation, it follows that
\begin{equation*}
\parentheses{\beta \Gamma_2 + 2 \Gamma_2^2 - c_1 - c_3 }x^2 +(\beta \Gamma_1 +2\Gamma_2\Gamma_1+2c_1c_2 m +2c_3 c_4 )x +\beta \Gamma_0+\frac{1}{2}\Gamma_1^2-\sigma^2 \Gamma_2 -( c_1{c_2}^2+c_5) m^2 - c_3 {c_4}^2=0.
\end{equation*}
An easy computation gives the values 
 \begin{align*}
 \Gamma_2&=\frac{-\beta + \sqrt{\beta^2 +8 (c_1+c_3)}}{4},\\
 \Gamma_1&=- \frac{ 2c_1c_2m+2c_3 c_4}{\beta + 2\Gamma_2},\\
 \Gamma_0&=\frac{ c_5  m^2 + c_3 {c_4}^2+ c_1 {c_2}^2 m^2 +\sigma^2 \Gamma_2 -\frac{1}{2}\Gamma_1^2 }{\beta}.
 \end{align*}
 
 By plugging the control $\hat\alpha(x)=-(2\Gamma_2x+\Gamma_1)$ into the dynamics of $X_t$ and taking the expected value, we obtain an ODE for ${ m_t}$

 \begin{equation} \label{eq: ode_mt_amfg}
     \dot m_t= -(2\Gamma_2 m_t+\Gamma_1).
 \end{equation}
The solution of (\ref{eq: ode_mt_amfg}) is used to derive $m$ as follows

\begin{equation}\label{valueofm}
    \begin{split}
      m &=\lim_{t\mapsto \infty} m_t 
      =\lim_{t\mapsto \infty} -\frac{\Gamma_1}{2\Gamma_2} + \parentheses{m_0 + \frac{\Gamma_1}{\Gamma_2}} e^{-2 \Gamma_2 t}
      =-\frac{\Gamma_1}{2\Gamma_2}
      =\frac{2c_1 c_2 m+2 c_3 c_4}{2 \Gamma_2 (\beta + 2\Gamma_2)},\\
     m &= \frac{c_3 c_4}{\Gamma_2 (\beta + 2\Gamma_2) -c_1 c_2 }
    \end{split}
\end{equation}

To summarize, we derived that $\hat\alpha(x)=-(2\Gamma_2x+\Gamma_1)$ with $\Gamma_2=\hat\Gamma_2$ and $\Gamma_1=\hat\Gamma_1$ obtained in \eqref{formula--alpha-hat-infty}. In other words, we have checked that
\begin{equation*}
  \lim_{t\to\infty}\hat\alpha_t^{MFG}(x) =  \hat\alpha^{AMFG}(x), \quad \forall x, 
\end{equation*}
that is the first part of \eqref{controls-MFG} for this LQ MFG.

    \subsection{Solution for stationary MFG}\label{appendix:SMFG}
 
The only difference with the derivation above in the case of asymptotic MFG is that $m_t$ should be a constant which, from 
\eqref{eq: ode_mt_amfg}, should satisfy $2\Gamma_2 m+\Gamma_1=0$. Therefore, $m$ takes the same value as in \eqref{valueofm}, and we deduce
\begin{equation*}
  \hat\alpha^{SMFG}(x) =  \hat\alpha^{AMFG}(x), \quad \forall x, \end{equation*}
that is the second part of \eqref{controls-MFG} for this LQ MFG.

\subsection { Solution for non-asymptotic MFC}

We present the solution for the following non-asymptotic MFC problem
\begin{align*}
 \min_{\boldsymbol{ \alpha}}J(\boldsymbol{ \alpha})&=\min_{\boldsymbol{ \alpha}} \mathbb{E}\left[\int_0^{\infty} e^{-\beta t}f(X^{\boldsymbol{ \alpha}}_t,\alpha_t,\E{X_t^{\boldsymbol{\alpha}}})dt \right]\\
  &=\min_{\boldsymbol{ \alpha}}\E{\int_0^{+\infty}e^{-\beta t }\parentheses{\frac{1}{2}\alpha_t^2 + c_1 \parentheses{ X_t^{\boldsymbol{ \alpha}}-c_2 \E{X_t^{\boldsymbol{ \alpha}}}}^2 + c_3 \parentheses{ X_t^{\boldsymbol{ \alpha}}-c_4}^2 + c_5 \E{X_t^{\boldsymbol{ \alpha}}}^2}dt},\\
\text{subject to: }&
 \quad dX^{\boldsymbol{ \alpha}}_t=\alpha_t dt +\sigma dW_t ,\quad 
    X^{\boldsymbol{ \alpha}}_0\sim\mu_0.    
\end{align*}
Note that here the mean $\E{X_t^{\boldsymbol{ \alpha}}}$ of the population  changes instantaneously when $\boldsymbol{ \alpha}$ changes.

This problem can be solved by two equivalent approaches: PDE and FBSDEs. Both approaches start by solving the  problem defined by a finite horizon $T$. Then, the solution to the infinite horizon problem is obtained by taking the limit for $T$ goes to infinity. Let $V^T(t,x)$ be the optimal value function for the finite horizon problem conditioned on $X_0=x$, i.e.
\begin{equation*}
V^T(t,x)=\inf_{\boldsymbol{ \alpha}}J^{\boldsymbol{m^\alpha},x}(\boldsymbol{ \alpha})=\inf_{\boldsymbol{ \alpha}}\E{\int_t^{T}e^{-\beta s }f(X_s^{\boldsymbol{\alpha}},\alpha_s,m_s^{\boldsymbol{\alpha}})ds\Big|X_0^{\boldsymbol{ \alpha}}=x}, \quad V^T(T,x)=0.  
\end{equation*}  Let's consider the following ansatz with its derivatives
\begin{equation}
\begin{split}
    V^T(t,x) &= \Gamma_2^T (t) x^2 +  \Gamma_1^T (t) x  + \Gamma_0^T (t), \quad V^T(T,x)=0,\\
    \partial_t V^T(t,x) &= \dot{\Gamma}_2^T (t) x^2 +  \dot{\Gamma}_1^T (t) x  + \dot{\Gamma}_0^T (t), \\
    \partial_x V^T(t,x) &= 2 \Gamma_2^T (t) x +  \Gamma_1^T (t) , \\
    \partial_{xx} V^T(t,x) &=2 \Gamma_2^T (t), 
\end{split}
\end{equation}
Starting by the MFC-HJB equation (4.12) given in \cite{MR3134900}, we extended it to the asymptotic case as follows 
\begin{align*}
&\beta V^T -V_t^T - H\parentheses{t,x,\boldsymbol{\mu},\alpha} - \int_{\mR} \fdv{H}{\mu}\parentheses{t,h,\boldsymbol{\mu},-\partial_x V^T}(x)\mu_t(h)dh=0,
\end{align*}
where   $m_t=\int_{\mR} y \mu_t(dy)$ and $\alpha^*=-\partial_x V^T$. We have:
\begin{align*}
 H\parentheses{t,x,\boldsymbol{\mu},\alpha} &:=\inf_{\alpha}\curlyp{\mathcal{A}^X V^T + f\parentheses{t, x,\alpha,\boldsymbol{\mu}}}\\
 &=\inf_{\alpha}\curlyp{\alpha \partial_x V^T + \frac{1}{2}\sigma^2 \partial_{xx}V^T+\frac{1}{2}\alpha^2 +c_1 (x-c_2 m_t)^2 + c_3 (x-c_4)^2 + c_5 {m_t}^2}\\
 &=-\frac{1}{2}(\partial_x V^T)^2+ \frac{1}{2}\sigma^2 \partial_{xx}V^T+c_1 (x-c_2 m_t)^2 + c_3 (x-c_4)^2 + c_5 {m_t}^2,
 \end{align*}
 \begin{align*}
\fdv{H\parentheses{t,h,\boldsymbol{\mu},\alpha}}{\mu} (x)&=\fdv{}{\mu}\parentheses{c_1 (h-c_2 m_t)^2 + c_5 {m_t}^2}(x)\\
&=\fdv{}{\mu}\parentheses{c_1 \parentheses{h- c_2 \int_{\mR} y \mu_t(dy)}^2 +c_5 \parentheses{\int_{\mR} y \mu_t(dy)}^2}(x)\\
&=-2c_1 c_2 x\parentheses{ h-c_2\int_{\mR} y \mu_t(dy))} +2c_5x\int_{\mR} y \mu_t(dy)\\
& =- 2c_1 c_2 x(h -c_2 m_t) +2c_5xm_t,
\end{align*}
\begin{align*}
&{ \int_{\mR} \fdv{H}{\mu}\parentheses{t,h,\boldsymbol{\mu},-\partial_x V^T}(x)\mu_t(h)dh}= - 2c_1 c_2 x(m_t - c_2 m_t) +2c_5xm_t,
\end{align*}
and finally
\begin{align*}
	&\beta V^T -\partial_t V^T  +\frac{1}{2}(\partial_x ^T)^2- \frac{1}{2}\sigma^2 \partial_{xx} V^T- c_1 (x- c_2 m_t)^2 
	\\
	&\qquad - c_3 (x- c_4)^2 - c_5 {m_t}^2 + 2 c_1 c_2 x(m_t - c_2 m_t) - 2c_5xm_t=0 .
\end{align*}
The following system of ODEs is obtained by replacing the ansatz and its derivatives in the MFC-HJB:
\begin{equation} \label{system_conditions_ODE_MFC}
\begin{cases}
{\dot \Gamma}_2^T -2({{ \Gamma}^T_2})^2 - \beta { \Gamma}^T_2 + c_1 + c_3 =0, \quad  &{ \Gamma}_2^T(T) = 0, \\
{\dot \Gamma}^T_1 = (2  { \Gamma}^T_2 + \beta ) { \Gamma}^T_1 + (2 c_1 c_2 ( 2 - c_2) -2c_5)m_t^T + 2 c_3 c_4, \quad  &{ \Gamma}^T_1(T)=0, \\
{\dot \Gamma}^T_0 = \beta { \Gamma}^T_0 + \frac{1}{2}({{ \Gamma}^T_1})^2 - \sigma^2 { \Gamma}^T_2 -c_3 {c_4}^2  - (c_1 {c_2}^2 +c_5) ({m^T_t})^2 , \quad  &{ \Gamma}^T_0(T) = 0,\\
\dot{m}_t^T = - 2 {\Gamma}^T_2  m^T - { \Gamma}^T_1, \quad  &m^T(0)= \E{X^{\boldsymbol{\alpha}}_0}=m_0,\\
\end{cases}
\end{equation}
where the last equation is obtained by considering the expectation of $X_t^{\boldsymbol{\alpha}}$ after replacing $\alpha^*(x) = -\partial_x V^T(x) = - (\Gamma^T_2 x + \Gamma^T_1)$.
The first equation is a Riccati equation. In particular, the solution $\Gamma^T_2$ converges to $\Gamma^*_2=\frac{-\beta + \sqrt{\beta^2+8(c_1 + c_3)}}{4}$ as $T$ goes to infinity. The second and fourth ODEs are coupled and they can be written in matrix notation as 

\begin{gather} %
 \dot{\wideparen{\begin{pmatrix} 
 m^T \\ \Gamma^T_1
 \end{pmatrix}}}
 = \begin{bmatrix}
 - 2 \Gamma^T_2  &  -1  \\
  (2 c_1 c_2 ( 2 - c_2) -2c_5) & 2 \Gamma^T_2 + \beta   
   \end{bmatrix}
   \begin{pmatrix} 
   m^T \\ \Gamma^T_1
    \end{pmatrix} + \begin{pmatrix} 
   0 \\ 2 c_3 c_4
    \end{pmatrix}, \quad
    \begin{pmatrix} 
   m^T(0) \\ \Gamma^T_1(T)
    \end{pmatrix}
    =
    \begin{pmatrix} 
   m_0 \\ 0
    \end{pmatrix}.
\end{gather}

By similar calculations to the non-asymptotic MFG case, the following solutions can be obtained
\begin{equation} %
\begin{split}
 m_t^*&=\lim_{T\to\infty}m^T_t
= p^*_t(1,1) h_1(t) + p^*_t(1,2)h_2(t)\\
&= \parentheses{m_0+2 \frac{c_3c_4}{\lambda_2-\lambda_1}\parentheses{\frac{1}{\lambda_1}-\frac{1}{\lambda_2}} }e^{t\lambda_{1}}+2 \frac{c_3c_4}{\lambda_2-\lambda_1}\parentheses{\frac{1}{\lambda_2}-\frac{1}{\lambda_1}},\\
\Gamma_1^*(t)&=\lim_{T\to\infty}\Gamma_1^T(t)
= p^*_t(2,1) h_1(t) + p^*_t(2,2)h_2(t)\\
&=m_0 (g-\lambda_1) e^{t\lambda_{1}}+2\frac{c_3c_4}{\lambda_2-\lambda_1}\parentheses{\frac{\lambda_2-g}{\lambda_2}-\frac{\lambda_1-g}{\lambda_1}},\\
\end{split}
\end{equation}
where

\begin{equation}
\begin{split}
g& \coloneqq  - 2  \Gamma_2^* ,\\ 
b&\coloneqq2(c_1c_2 (2-c_2) -  c_5) ,\\ 
a&\coloneqq 2 \Gamma_2^*    + \beta  ,\\  
d &\coloneqq - 1,\\
\lambda_{1\backslash 2}&\coloneqq  \frac{a+g \pm \sqrt{(a-g)^2 + 4b d}}{2} = t \frac{\beta \pm \sqrt{(4 \Gamma_2^* + \beta)^2-8(c_1c_2 (2-c_2) -  c_5)}}{2} .  \\
\end{split}
\end{equation}

As in the MFG case, the third ODE in (\ref{system_conditions_ODE_MFC}) can be solved by plugging in the solution of the previous ones and integrating. Since our interest is into the evolution of the mean and the control function, we omit the calculation for this ODE.

\subsection { Solution for Asymptotic MFC} 
The asymptotic version of the problem presented above is given by:
\begin{align*}
 \min_{\boldsymbol{ \alpha}}J(\boldsymbol{ \alpha})&=\inf_{\boldsymbol{ \alpha}} \mathbb{E}\left[\int_0^{\infty} e^{-\beta t}f(X^{\boldsymbol{ \alpha}}_t,\alpha_t,m^{\boldsymbol{\alpha}})dt \right]\\
  &=\inf_{\boldsymbol{ \alpha}}\E{\int_0^{+\infty}e^{-\beta t }\parentheses{\frac{1}{2}\alpha_t^2 + c_1  \parentheses{ X_t^{\boldsymbol{ \alpha}}-c_2 m^{\boldsymbol{ \alpha}}}^2 + c_3 \parentheses{ X_t^{\boldsymbol{ \alpha}}- c_4}^2 +c_5 (m^{\boldsymbol{ \alpha}})^2}dt},\\
\text{subject to: }&
 \quad dX^{\boldsymbol{ \alpha}}_t=\alpha_t dt +\sigma dW_t ,\quad 
    X^{\boldsymbol{ \alpha}}_0\sim\mu_0,    
\end{align*}
where $m^{\boldsymbol{ \alpha}} = \lim_{t \to +\infty} \E{X^{\alpha}_t}.$ \\
Let $V(x)$ be the optimal value function conditioned on $X_0=x$, i.e.
\begin{equation*}
V(x)=\inf_{\boldsymbol{ \alpha}}J^{x}(\boldsymbol{ \alpha})=\inf_{\boldsymbol{ \alpha}}\E{\int_0^{+\infty}e^{-\beta t }\parentheses{\frac{1}{2}\alpha_t^2 + c_1 \parentheses{ X_t^{\boldsymbol{ \alpha}}-c_2 m^{\boldsymbol{ \alpha}}}^2 + c_3 \parentheses{ X_t^{\boldsymbol{ \alpha}}-c_4}^2 + c_5 (m^{\boldsymbol{ \alpha}})^2}dt\Big|X_0^{\boldsymbol{ \alpha}}=x}.  
\end{equation*}
We consider the following ansatz with its derivative
\begin{align*}
V(x)&=\Gamma_2 x^2 + \Gamma_1 x + \Gamma_0, \\ 
\dot V(x)&= 2\Gamma_2 x + \Gamma_1, \\
\ddot V(x)&=2\Gamma_2. \\
\end{align*}
Starting by the MFC-HJB equation (4.12) given in \cite{MR3134900}, we extended it to the asymptotic case as follows 
\begin{align*}
&\beta V(x) - H\parentheses{x,\mu^{\boldsymbol{ \alpha}},\alpha} - \int_{\mR} \fdv{H}{\mu}\parentheses{h,\mu^{\boldsymbol{ \alpha}},-\dot V(h)}(x)\mu^{\boldsymbol{ \alpha}}(h)dh=0,
\end{align*}
where  $m^{\boldsymbol{ \alpha}}=\int_{\mR} y \mu^{\boldsymbol{ \alpha}}(dy)$. We have:
\begin{align*}
 H\parentheses{x,\mu^{\boldsymbol{ \alpha}},\alpha} &:=\inf_{\alpha}\curlyp{\mathcal{A}^X V(x) + f\parentheses{ x,\alpha,\mu^{\boldsymbol{ \alpha}}}}\\
 &=\inf_{\alpha}\curlyp{\alpha \dot V(x) + \frac{1}{2}\sigma^2 \ddot V(x)+\frac{1}{2}\alpha^2 +c_1 (x-c_2 m^{\boldsymbol{ \alpha}})^2 + c_3 (x-c_4)^2 + c_5 (m^{\boldsymbol{ \alpha}})^2}\\
 &=-\frac{1}{2}\dot V(x)^2+ \frac{1}{2}\sigma^2 \ddot V(x)+c_1 (x-c_2 m^{\boldsymbol{ \alpha}})^2 + c_3 (x-c_4)^2 + c_5 (m^{\boldsymbol{ \alpha}})^2,
 \end{align*}
 \begin{align*}
\fdv{H\parentheses{h,\mu^{\boldsymbol{ \alpha}},\alpha}}{\mu} (x)&=\fdv{}{\mu}\parentheses{c_1 (h-c_2 m^{\boldsymbol{ \alpha}})^2 + c_5 (m^{\boldsymbol{ \alpha}})^2}(x)\\
&=\fdv{}{\mu}\parentheses{c_1 \parentheses{h- c_2 \int_{\mR} y \mu^{\boldsymbol{ \alpha}}(dy)}^2 +c_5 \parentheses{\int_{\mR} y \mu^{\boldsymbol{ \alpha}}(dy)}^2}(x)\\
&=-2c_1 c_2 x\parentheses{ h-c_2\int_{\mR} y \mu^{\boldsymbol{ \alpha}}(dy))} +2c_5x\int_{\mR} y \mu^{\boldsymbol{ \alpha}}(dy)
 =- 2c_1 c_2 x(h -c_2 m^{\boldsymbol{ \alpha}}) +2c_5xm^{\boldsymbol{ \alpha}},
 \end{align*}
 \begin{align*}
&{ \int_{\mR} \fdv{H}{\mu}\parentheses{h,\mu^{\boldsymbol{ \alpha}},-\dot V(h)}(x)\mu^{\boldsymbol{ \alpha}}(h)dh}= - 2c_1 c_2 x(m^{\boldsymbol{ \alpha}} - c_2 m^{\boldsymbol{ \alpha}}) +2c_5xm^{\boldsymbol{ \alpha}},
\end{align*}
and finally the HJB equation becomes:
\begin{align*}
&\beta V(x) +\frac{1}{2}\dot V(x)^2- \frac{1}{2}\sigma^2 \ddot V(x)- c_1 (x- c_2 m^{\boldsymbol{ \alpha}})^2 - c_3 (x- c_4)^2 - c_5 (m^{\boldsymbol{ \alpha}})^2 + 2 c_1 c_2 x(m^{\boldsymbol{ \alpha}} - c_2 m^{\boldsymbol{ \alpha}}) - 2c_5xm^{\boldsymbol{ \alpha}}=0 .
\end{align*}
A system of ODEs is obtained by replacing the ansatz and its derivatives in the MFC-HJB and cancelling terms in $x^2$, and $x$ and constant:
\begin{equation*}
\begin{split}
\parentheses{\beta \Gamma_2 + 2 \Gamma_2^2 - c_1 - c_3 }x^2 &+\parentheses{\beta \Gamma_1 +2\Gamma_2\Gamma_1+2 c_1 c_ 2 m^{\boldsymbol{ \alpha}} (2-c_2) +2c_3 c_4 - 2 c_5 m^{\boldsymbol{ \alpha}}}x \\
&+\beta \Gamma_0+\frac{1}{2}\Gamma_1^2-\sigma^2 \Gamma_2 -( c_1 {c_2}^2+c_5) (m^{\boldsymbol{ \alpha}})^2 - c_3 {c_4}^2=0.
\end{split}
\end{equation*}
An easy computation gives the values 
 \begin{align*}
 \Gamma_2&=\frac{-\beta + \sqrt{\beta^2 +8 (c_1+c_3)}}{4},\\
 \Gamma_1&= \frac{2 c_5 m^{\boldsymbol{ \alpha}} -2c_1 c_2m^{\boldsymbol{ \alpha}}(2-c_2)-2c_3 c_4}{\beta + 2\Gamma_2},\\
 \Gamma_0&=\frac{  c_5 (m^{\boldsymbol{ \alpha}})^2 + c_3 {c_4}^2+ c_1 {c_2}^2 (m^{\boldsymbol{ \alpha}})^2 +\sigma^2 \Gamma_2 -\frac{1}{2}\Gamma_1^2 }{\beta}.
 \end{align*}
 
 By plugging the control $\alpha^*(x)=-(2\Gamma_2x+\Gamma_1)$ into the dynamics of $X^{\boldsymbol{\alpha}}_t$ and taking the expected value, we obtain an ODE for $m^{\boldsymbol{ \alpha}}_t$
 
 \begin{equation} \label{eq: ode_mt_amfc}
     \dot m_t^{\boldsymbol{ \alpha}}= -(2\Gamma_2 m_t^{\boldsymbol{ \alpha}}+\Gamma_1).
 \end{equation}

The solution of (\ref{eq: ode_mt_amfc}) is used to derive $m$ as follows

\begin{equation}\label{valueofmalpha}
    \begin{split}
      m^{\boldsymbol{ \alpha}} &=\lim_{t\mapsto \infty} m_t^{\boldsymbol{ \alpha}} 
       =\lim_{t\mapsto \infty} \left(-\frac{\Gamma_1}{2\Gamma_2} + \parentheses{m_0 + \frac{\Gamma_1}{\Gamma_2}} e^{-2 \Gamma_2 t}\right)\\
      &=-\frac{\Gamma_1}{2\Gamma_2}
      =-\frac{2c_5 m^{\boldsymbol{ \alpha}} -2c_1 c_2m^{\boldsymbol{ \alpha}}(2-c_2)-2c_3 c_4}{2 \Gamma_2 (\beta + 2\Gamma_2)}\\
     m^{\boldsymbol{ \alpha}} &= \frac{c_3 c_4}{\Gamma_2 (\beta + 2\Gamma_2)+ c_5 -c_1 c_2 (2-c_2)}
    \end{split}
\end{equation}
We remark that the values of $m_t^{\boldsymbol{ \alpha}}$ and $\Gamma_1(t)$ obtained in the non-asymptotic case converge to $m^{\alpha}$ and $\Gamma_1$ respectively as $t$ goes to $\infty$. Therefore, we have obtained that
\begin{equation*}
  \lim_{t\to\infty}\alpha_t^{*MFC}(x) =  \alpha^{*AMFG}(x), \quad \forall x, 
\end{equation*}
that is the first part of \eqref{controls-MFC} for this LQ MFC problem.

 \subsection{Solution for stationary MFC}\label{appendix:SMFC}

The only difference with the derivation above in the case of asymptotic MFC is that $m^\alpha_t$ should be a constant which, from 
\eqref{eq: ode_mt_amfc}, should satisfy $2\Gamma_2 m^\alpha+\Gamma_1=0$. Therefore, $m^\alpha$ takes the same value as in \eqref{valueofmalpha}, and we deduce
\begin{equation*}
  \alpha^{*SMFG}(x) =  \alpha^{*AMFG}(x), \quad \forall x, \end{equation*}
that is the second part of \eqref{controls-MFC} for this LQ MFC problem .

\clearpage
\section{Lipschitz property of the 2 scale  operators}\label{appendix : borkar_assumptions}

\subsection{Generic setting}\label{appendix:Lip}

We modify the original operators using the softmin operator on $\mathbb{R}^{|\mathcal{A}|}$ defined as:
$$
    \softmin(z) = \left(\frac{e^{-z_i}}{\sum_j e^{-z_j}}\right)_{i=1,\dots,|\mathcal{A}|} \in \Delta^{|\mathcal{A}|}, \qquad z \in \mathbb{R}^{|\mathcal{A}|}.
$$
Intuitively, it gives a probability distribution on the indices $i=1,\dots,|\mathcal{A}|$ which has higher values on indices whose corresponding values are closer to be a minimum. In particular, the elements of $\min\{i=1,\dots,|\mathcal{A}| : z_i = \argmin_j z_j\}$ have equal weight and this weight is the largest among $\left(\frac{e^{-z_i}}{\sum_j e^{-z_j}}\right)_{i=1,\dots,|\mathcal{A}|}$. We recall that the function $\softmin$ is Lipschitz continuous for the $2$-norm. Denoting by $L_s$ its Lipschitz constant, it means that
$$
    \|\softmin(z) - \softmin(z')\|_2
    \le L_s \|z - z'\|_2, \qquad z, z' \in \mathbb{R}^{|\mathcal{A}|}.
$$
Moreover, since $|\mathcal{A}|$ is finite, all the norms on $\mathbb{R}^{|\mathcal{A}|}$ are equivalent so there exists a positive constant $c_{2,\infty}$ such that 
$$
    \|\softmin(z) - \softmin(z')\|_\infty
    \le L_s c_{2,\infty} \|z - z'\|_\infty, \qquad z, z' \in \mathbb{R}^{|\mathcal{A}|}.
$$

To alleviate the notation, we will write $Q(x) := (Q(x,a))_{a \in \mathcal{A}}$ for any $Q \in \RR^{|\mathcal{X}| \times |\mathcal{A}|}$. We also introduce a more general version $\underline{p}$ of the transition kernel $p$, which can take as an input a probability over actions instead of a single action: for $x,x' \in \mathcal{X}, \nu \in \Delta^{|\mathcal{A}|}, \mu \in \Delta^{|\mathcal{X}|}$, 
$$
    \underline{p}(x'|x,\nu,\mu) = \sum_{a} \nu(a) p(x'|x,a,\mu).
$$
Intuitively, this is the probability for a agent at $x$ to move to $x'$ when the population distribution is $\mu$ and the agent picks a random action following the distribution $\nu$.

We now consider the following iterative procedure, which is a slight modification of~\eqref{eq1:2scale-mu-k}--\eqref{eq1:2scale-Q-k}. Here again, both variables ($Q$ and $\mu$) are updated at each iteration but with different rates.  Starting from an initial guess $(Q_0, \mu_0) \in \RR^{|\mathcal{X}| \times |\mathcal{A}|} \times \Delta^{|\mathcal{X}|}$, define iteratively for $k=0,1,\dots$:

\begin{subequations}%
     \begin{empheq}[left=\empheqlbrace]{align}
    \label{eq:2scale-mu-k}
    \mu_{k+1} &= \mu_{k} + \rho_{k}^\mu \underline{\mathcal{P}}(Q_{k}, \mu_{k}),
    \\
     \label{eq:2scale-Q-k}
    Q_{k+1} &= Q_{k} + \rho_{k}^Q \mathcal{T}(Q_{k}, \mu_{k}),
    \end{empheq}
\end{subequations}

where 
$$
\begin{cases}
    \mathcal{T}(Q, \mu)(x,a) = f(x, a, \mu) + \gamma \sum_{x'} p(x' | x,a,\mu) \min_{a'}Q(x',a') - Q(x,a), \qquad (x,a) \in \mathcal{X} \times \mathcal{A},
    \\
    \underline{\mathcal{P}}(Q, \mu)(x) = (\mu \underline{P}^{Q,\mu})(x) - \mu(x), \qquad x \in \mathcal{X}, 
\end{cases}
$$
with 
$$
    \underline{P}^{Q,\mu}(x, x') = \underline{p}(x' | x, \softmin Q(x), \mu), 
    \qquad \hbox{ and } 
    \qquad
    (\mu \underline{P}^{Q,\mu})(x) = \sum_{x_0} \mu(x_0) \underline{P}^{Q,\mu}(x_0,x),
$$
is the transition matrix when the population distribution is $\mu$ and the agent uses an approximately optimal randomized control according to the soft-min of $Q$.  

\begin{lemma}
    Assume that $f$ is Lipschitz continuous with respect to $\mu$ and that $\underline{p}$ is Lipschitz continuous with respect to $\nu$ and $\mu$. Then 
    \begin{itemize}
        \item the operator $\mathcal{T}$ is Lipschitz continuous w.r.t. $\mu$ (with a Lipschitz constant possibly depending on $\|Q\|_\infty)$, and Lipschitz continuous in $Q$ (uniformly in $\mu$);
        \item the operator $\underline{\mathcal{P}}$ is Lipschitz continuous in both variables.
    \end{itemize}
    If $p$ is independent of $\mu$, then both $\mathcal{T}$ and $\underline{\mathcal{P}}$ are Lipschitz continuous. 
\end{lemma}
\begin{proof}
    Let us denote by $L_p$ and $L_f$ the  Lipschitz constants of $p$ and $f$  respectively. Let $(Q,\mu),(Q',\mu') \in \RR^{|\mathcal{X}| \times |\mathcal{A}|} \times \Delta^{|\mathcal{X}|}$. We first consider $\mathcal{T}$. We have
    \begin{align*}
        \|\mathcal{T}(Q,\mu) - \mathcal{T}(Q',\mu)\|_{\infty}
        &\le \gamma \sum_{x'} \max_{x,a} p(x' | x,a,\mu) \left|\min_{a'}Q(x',a') - \min_{a'}Q'(x',a')\right| + \left\| Q - Q'\right\|_{\infty}
        \\
        &\le (\gamma + 1) \left\| Q - Q'\right\|_{\infty}.
    \end{align*}
    Moreover,
    \begin{align*}
        \|\mathcal{T}(Q,\mu) - \mathcal{T}(Q,\mu')\|_{\infty}
        &\le
        |f(x, a, \mu) - f(x, a, \mu')| 
        \\
        &\qquad + \gamma \sum_{x'} |p(x' | x,a,\mu) - p(x' | x,a,\mu')| \, |\min_{a'}Q(x',a')|
        \\
        &\le
        (L_f 
        + \gamma L_p \|Q\|_\infty )|\mathcal{X}| \|\mu- \mu'\|_{\infty},  
    \end{align*}
    where $L_f$ and $L_p$ are respectively the Lipschitz constants of $f$ and $p$ with respect to $\mu$. If $p$ is independent of $\mu$, we obtain 
    \begin{align*}
        \|\mathcal{T}(Q,\mu) - \mathcal{T}(Q,\mu')\|_{\infty}
        &\le
        L_f 
         \|\mu- \mu'\|_{\infty}.
    \end{align*}
    
    We then show that the operator $\underline{\mathcal{P}}$ is Lipschitz continuous. We have
    \begin{align*}
        &\|\underline{\mathcal{P}}(Q, \mu) - \underline{\mathcal{P}}(Q, \mu')\|_{\infty}
        \\
        &\le \|\mu \underline{P}^{Q,\mu} - \mu' \underline{P}^{Q,\mu'}\|_{\infty} + \|\mu - \mu'\|_{\infty}
        \\
        &\le \left\|\sum_x \Big(\underline{p}(\cdot | x, \softmin Q(x), \mu)\mu(x) - \underline{p}(\cdot | x, \softmin Q(x), \mu')\mu'(x)\Big) \right\|_{\infty} 
        \\
        &\qquad + \|\mu - \mu'\|_{\infty}.
    \end{align*}
    For the first term, we note that, for every $x \in \mathcal{X}$,
    \begin{align*}
        &\left\| \Big(\underline{p}(\cdot | x, \softmin Q(x), \mu)\mu(x) - \underline{p}(\cdot | x, \softmin Q(x), \mu')\mu'(x)\Big)\right\|_{\infty}
        \\
        &\le 
        \left\| \Big(\underline{p}(\cdot | x, \softmin Q(x), \mu) - \underline{p}(\cdot | x, \softmin Q(x), \mu') \Big)\mu(x)\right\|_{\infty}
        \\
        &\qquad\qquad + \left\| \underline{p}(\cdot | x, \softmin Q(x), \mu')\Big(\mu(x) - \mu'(x)\Big)\right\|_{\infty}
        \\
        &\le 
        (L_p + 1)  \left\|\mu - \mu'\right\|_{\infty},
    \end{align*}
    where we used the fact that discrete probability measures are non-negative and bounded by $1$. 
    
    Moreover, we have
    \begin{align*}
        \|\underline{\mathcal{P}}(Q, \mu) - \underline{\mathcal{P}}(Q', \mu)\|_{\infty}
        &\le 
        \|\mu (\underline{P}^{Q,\mu} -  \underline{P}^{Q',\mu'})\|_{\infty}
        \\
        &\le 
        \sum_x \|\underline{p}(\cdot | x, \softmin Q(x), \mu) - \underline{p}(\cdot | x, \softmin Q'(x), \mu)\|_{\infty}
        \\
        &\le \sum_x L_p \|\softmin Q(x) - \softmin Q'(x)\|_{\infty}
        \\
        &\le |\mathcal{X}| \, L_p \, L_s \, c_{2,\infty} \, \| Q - Q'\|_{\infty},
    \end{align*}
    which concludes the proof. 
\end{proof}

\subsection{Application to a discrete model for the LQ problem}\label{discretization}

 Recall that the continuous linear-quadratic model we consider is defined by~\eqref{eq: benchmark}. 
Here, we propose a finite space MDP which approximates the dynamics of a typical agent in this continuous LQ model. We consider that the action space is given by $\mc{A} = \{ a_0=-1, a_1 = -1+\Delta_{.}, \dots, a_{N_{\mc{A}}}=1-\Delta_{.}, a_{N_{\mc{A}}}=1 \}$ and the state space by  $\mc{X} = \{ x_0=x_c-2, x_1=x_c-2-\Delta_{.}, \dots, x_{N_{\mc{X}}-1}=x_c+2-\Delta_{.}, x_{N_{\mc{X}}}=x_c+2 \}$, where $x_c$ is the center of the state space. The step size for the discretization of the spaces  $\mc{X}$ and $\mc{A}$ is given by $\Delta_{.} = \sqrt{\Delta t} = 10^{-1} $.

Consider the transition probability:
$$
    p(x,x',a,\mu) = \mathbb{P}(Z^{x+a, \Delta t} \in [x'-\Delta_{.}/2, x'+\Delta_{.}/2])
    = \Phi_{x+a, \sigma^2 \Delta t}(x'+\Delta_{.}/2) - \Phi_{x+a, \sigma^2 \Delta t}(x'-\Delta_{.}/2),
$$
where $Z \sim \mathcal{N}(x+a,\sigma^2 \Delta t)$ and $\Phi_{x+a,\sigma^2 \Delta t}$ is the cumulative distribution function of the $\mathcal{N}(x+a,\sigma^2 \Delta t)$ distribution. 
Moreover, consider that the one-step cost function is given by $f(x,a,\mu) \Delta t$ with 
\begin{equation*}
  f(x,a,\mu) = \frac{1}{2}a^2 + c_1 \left( x- c_2 \sum_{\xi \in S} \mu(\xi)\right)^2 + c_3 \left( x- c_4 \right)^2 + c_5 \left(\sum_{\xi \in S} \mu(\xi)\right)^2,
    \qquad
    b(x,a,\mu) = a,   
\end{equation*}

For simplicity, we write $\bar\mu = \sum_{\xi \in S} \mu(\xi)$.

\begin{lemma}
    In this model, $f$ is Lipschitz continuous with respect to $\mu$ and $\underline{p}$ is Lipschitz continuous with respect to $\nu$ and $\mu$
\end{lemma}

\begin{proof}

We start with $f$. %
For the $\mu$ component, we have:
    \begin{align*}
    |f(x,a,\mu) - f(x,a,\mu')| 
    &\le c \left|\left( x- c_2 \bar\mu\right)^2 - \left( x- c_2 \bar\mu'\right)^2\right| + c \left|\left(\bar\mu\right)^2 - \left(\bar\mu'\right)^2\right|
    \\
    &\le c \left(\bar\mu' - \bar\mu \right) \cdot \left(2x + (\bar\mu' - \bar\mu)\right) + c (\bar\mu - \bar\mu')(\bar\mu + \bar\mu')
    \\
    &\le c \max_{x \in S}\|x\|_\infty \, \left(\bar\mu' - \bar\mu \right)
    \\
    &\le c \max_{x \in S}\|x\|_\infty \, \sum_{x \in S}\left(\mu'(x) - \mu(x) \right)
    \\
    &\le c \max_{x \in S}\|x\|_\infty \, |S| \, \|\mu' - \mu\|_\infty,
    \end{align*}
    where $c>0$ is a constant depending only on the parameters of the model and whose value may change from line to line.

Then we consider $\underline{p}$. It is independent of $\mu$ in this model. For the action component, we have:
    \begin{align*}
        &|\underline{p}(x,x',\nu,\mu) - \underline{p}(x,x',\nu',\mu)|
        \\
        &= \Big|\sum_{a}\nu(a)\Big(\Phi_{x+a, \sigma^2 \Delta t}(x'+\Delta_{.}/2) - \Phi_{x+a, \sigma^2 \Delta t}(x'-\Delta_{.}/2) \Big) 
        \\
        &\qquad\qquad - \sum_{a'}\nu'(a')\Big(\Phi_{x+a', \sigma^2 \Delta t}(x'+\Delta_{.}/2) - \Phi_{x+a', \sigma^2 \Delta t}(x'-\Delta_{.}/2)\Big)\Big|
        \\
        &= \left|\sum_{a} \left(\nu(a)\Phi_{x+a, \sigma^2 \Delta t}(x'+\Delta_{.}/2) - \nu'(a)\Phi_{x+a, \sigma^2 \Delta t}(x'+\Delta_{.}/2) \right)\right|
        \\
        &\qquad +
        \left|\sum_{a} \left(\nu(a)\Phi_{x+a, \sigma^2 \Delta t}(x'-\Delta_{.}/2) \Big) -  \nu'(a)\Phi_{x+a, \sigma^2 \Delta t}(x'-\Delta_{.}/2) \right)\right|
        \\
        &= \int_{-\infty}^{x'+\Delta_{.}/2} \frac{1}{\sigma \sqrt{2\pi \Delta t }} \left|\sum_{a}(\nu(a) - \nu'(a))e^{-\frac{(y-(x+a))^2}{2\sigma^2 \Delta t}} \right| dy
        \\
        &\qquad\qquad + \int_{-\infty}^{x'-\Delta_{.}/2} \frac{1}{\sigma \sqrt{2\pi \Delta t }} \left|\sum_{a}(\nu(a) - \nu'(a)) e^{-\frac{(y-(x+a))^2}{2\sigma^2 \Delta t}} \right| dy
        \\
        &\le c \|\nu - \nu'\|_\infty,
    \end{align*}
    where $c$ is a constant depending only on the model (and in particular on the state space, the action space and $\Delta t$).

\end{proof}

\section{The Bellman equation for the optimal Q function in the Asymptotic MFC framework}\label{appendix : MFC-RL}
In this appendix, we provide the derivation of the Bellman equation \eqref{eq:def-Q-functionMFC} for the modified $Q-$function presented in section \ref{subsec: Q-A-MFC}. \newline

\noindent Let $\mc{X}$ and $\mc{A}$ be discrete and finite state and action spaces. Let $V^\alpha: \mc{X}  \mapsto \mc{R} $ and $Q^\alpha: \mc{X} \times \mc{A} \mapsto \mc{R} $  be value function relative to the policy $\alpha$ and the corresponding   modified $Q-$function defined as follows
\begin{align}%
    V^{\alpha}(x)  &\vcentcolon= \mathbb{E} \left[ \sum_{n=0}^{\infty} \gamma^n f(X_{n}, \alpha(X_{n}),\mu^{\alpha})\,\Big\vert\, X_{0} = x \right], \label{def_v_pi}\\
    Q^\alpha(x,a)&\vcentcolon=f(x,a, \mu^{\tilde\alpha})+\EE\left[\sum_{n=1}^\infty \gamma^n f(X_{n},\alpha(X_{n}), \mu^\alpha) \,\Big\vert\, X_{0}=x,A_{0}=a\right],\label{def_q_pi}
\end{align}
where 
\begin{equation*}
    \mu^{\alpha}= \lim_{n\mapsto \infty} \mc{L}(X^{\alpha}_{n}) \quad \text{and} \quad \tilde{\alpha}(s) = \begin{cases}
        \alpha(s), &\quad \forall s\neq x,\\
        a,  &\quad \text{if } s= x.\\
    \end{cases}
\end{equation*}
\begin{theorem}\label{thm: bellman A-MFC}

The optimal $Q^*(x,a)=\min_\alpha Q^\alpha(x,a)$ satisfies the Bellman equation
\begin{equation}
\label{eq:bellman A-MFC}
    Q^*(x,a) = f(x, a, \tilde{\mu}^*) + \gamma\sum_{x' \in \mathcal{X}} p(x' | x, a, \tilde{\mu}^*) \min_{a'} Q^*(x',a'), \qquad (x,a) \in \mathcal{X} \times \mathcal{A}, \end{equation}
where the optimal control $\alpha^*$ is given by   $\alpha^*(x)=\argmin_a Q^*(x,a)$, the modification $\tilde\alpha^*(x)$ is based on the pair $(x,a)$ and 
$\tilde{\mu}^*:=\mu^{\tilde{\alpha}^*}$. 
\end{theorem}
\begin{remark}\label{rmk:mu_tilde_star}
The population distribution $\tilde{\mu}^*$ based on the modification of $\alpha^*$ given the pair $(x,\alpha^*(x))$ is equal to ${\mu}^*$ . Indeed, $\tilde{\alpha}^*$ is equal to $\alpha^*$ itself, i.e. 
\begin{equation*}
         \tilde{\alpha}^*(s) = \begin{cases}
        \alpha^*(s), &\quad \forall s\neq x,\\
        \alpha^*(s),  &\quad \text{if } s= x.\\
    \end{cases}
\end{equation*}
\end{remark}
\begin{remark}\label{rmk:min_Q_star}
The term $\min_{a'} Q^*(x',a')$ does not depend on $\tilde{\mu}^*$ , i.e. 
\begin{equation*}
\begin{aligned}
     \min_{a'} Q^*(x',a') &= Q^*(x',\alpha^*(x')) = \\
      &=f(x', \alpha^*(x'), \tilde{\mu}^*) + \gamma\sum_{x'' \in \mathcal{X}} p(x'' | x', \alpha^*(x'), \tilde{\mu}^*) \min_{a'} Q^*(x'',a')=\\
      &\myeq{$\square$}{=}f(x', \alpha^*(x'), {\mu}^*) + \gamma\sum_{x' \in \mathcal{X}} p(x'' | x', \alpha^*(x'), {\mu}^*) \min_{a'} Q^*(x'',a')
\end{aligned}
\end{equation*}
where step $\square$ is due to Remark \ref{rmk:mu_tilde_star}. It follows that \eqref{eq:bellman A-MFC} depends on $\tilde{\mu}^*$ only through the cost due to the first step. 
\end{remark}
In order to prove Theorem \ref{thm: bellman A-MFC}, the following results are required.
\begin{theorem}\label{thm:bellman_q_pi}
The Bellman equation for $Q^{\alpha}$ is given by
\begin{equation} \label{eq:bellman_q_pi}
  Q^{\alpha}(x,a)  =  f(x,a,\mu^{\tilde{\alpha}}) +\gamma \mathbb{E} \left[ Q^{\alpha}(X_{1},\alpha(X_{1}))\,\Big\vert\, X_{0} = x, A_{0} = a \right],    
\end{equation}
\end{theorem}
\begin{lemma} \label{lemma:property_1} The value function relative to the policy $\alpha$ is equivalent to the corresponding $Q-$function evaluated on the pair $(x,\alpha(x))$, i.e.
\begin{equation}\label{eq:property_1}
V^{\alpha}(x)  = Q^{\alpha}(x,\alpha(x)).
\end{equation}
\end{lemma}
\begin{theorem}[Policy improvement]\label{thm:policy_improvement_amfc}
Let $\tilde{\alpha}$ be a policy derived by $\alpha$
\begin{equation*}
\begin{aligned}
\tilde{\alpha}(s) &= \begin{cases} \alpha(s), \quad &\text{for } s\neq x,\\ a, \quad &\text{for } s= x.\end{cases} 
\end{aligned}
\end{equation*}
such that
\begin{equation}\label{eq:hp_policy_improvement_amfc}
    Q^{\alpha}(x,\tilde{\alpha}(x)) > V^{\alpha}(x). 
\end{equation}
Then,
\begin{equation}\label{eq:policy_improvement_amfc}
     V^{\tilde{\alpha}}(x') > V^{\alpha}(x') \quad \forall x' \in \mathcal{X}. 
\end{equation}
\end{theorem}

\begin{theorem}\label{thm:property_3} Let $V^*:\mc{X} \mapsto \mc{R}$ be defined as $V^*(x)=\max_{\alpha} V^{\alpha}(x)$. Then,
\begin{equation}\label{eq:property_3}
V^*(x)= \max_a \max_{\alpha} Q^{\alpha}(x,a), 
\end{equation}
\end{theorem}

\begin{proof}[Theorem \ref{thm:bellman_q_pi}]
\begin{equation*}
\begin{aligned}
Q^{\alpha}(x,a)  &= f(x,a,\mu^{\tilde{\alpha}}) + \\
&+\gamma \mathbb{E} \left[ \mathbb{E} \left[ \sum_{n=1}^{\infty} \gamma^{n-1} f(X_{n},\alpha(X_{n}),\mu^{\alpha})\,\Big\vert\, X_{0} = x, A_{0} = \alpha(x), X_{1} \right]\,\Big\vert\, X_{0} = x, A_{0} = a \right]= \\ 
&= f(x,a,\mu^{\tilde{\alpha}}) + \gamma \mathbb{E} \left[ \mathbb{E} \left[ \sum_{n=1}^{\infty} \gamma^{n-1} f(X_{n},\alpha(X_{n}),\mu^{\alpha})\,\Big\vert\, X_{1} \right]\,\Big\vert\, X_{0} = x, A_{0} = a\right]= \\ 
&= f(x,a,\mu^{\tilde{\alpha}}) +\\
&+ \gamma \mathbb{E} \left[  f(X_{1},\alpha(X_{1}),\mu^{\alpha}) + \gamma \mathbb{E} \left[ \sum_{n=2}^{\infty} \gamma^{n-2} f(X_{n},\alpha(X_{n}),\mu^{\alpha})\,\Big\vert\, X_{1} \right]\,\Big\vert\, X_{0} = x, A_{0} = a \right]= \\ 
&=  f(x,a,\mu^{\tilde{\alpha}}) +\gamma \mathbb{E} \left[ Q^{\alpha}(X_{1},\alpha(X_{1}))\,\Big\vert\, X_{0} = x, A_{0} = a \right], 
\end{aligned}
\end{equation*}
\end{proof}

\begin{proof}[Lemma \ref{lemma:property_1}]
\begin{equation*}
\begin{aligned}
V^{\alpha}(x) 
&= f(x,\alpha(x),\mu^{\alpha}) + \mathbb{E} \left[ \sum_{n=1}^{\infty} \gamma^n f(X_{n}, \alpha(X_{n}),\mu^{\alpha})\,\Big\vert\, X_{0} = x, A_{0} = \alpha(x) \right]= \\ 
&= f(x,\alpha(x),\mu^{\tilde\alpha}) + \mathbb{E} \left[ \sum_{n=1}^{\infty} \gamma^n f(X_{n}, \alpha(X_{n}),\mu^{\alpha})\,\Big\vert\, X_{0} = x, A_{0} = \alpha(x) \right]= \\ 
&\myeq{(\ref{def_q_pi})}{=} Q^{\alpha}(x,\alpha(x)) 
\end{aligned}
\end{equation*}
where we used that the modification of $\alpha$ given the pair $(x,\alpha(x))$ is equal to $\alpha$ itself and consequently $\mu^{\alpha}=\mu^{\tilde\alpha}$.
\end{proof}

\begin{proof}[Theorem \ref{thm:policy_improvement_amfc}]
\textbf{Step 1} Show that $V^{\alpha}(x) < V^{\tilde{\alpha}}(x) $. \vspace{10pt}\newline We observe that
\begin{equation*}
\begin{aligned}
V^{\alpha}(x) &< Q^{\alpha}(x,\tilde{\alpha}(x))=\\
&\myeq{(\ref{eq:bellman_q_pi})}{=} f(x,\tilde{\alpha}(x),\mu^{\tilde{\alpha}}) + \gamma \mathbb{E} \left[ Q^{\alpha}(X_{1},\alpha(X_{1})) \,\Big\vert\, X_{0}=x, A_{0} = \tilde{\alpha}(x) \right] = \\
&\myeq{(\ref{eq:property_1})}{=}  f(x,\tilde{\alpha}(x),\mu^{\tilde{\alpha}}) + \gamma \mathbb{E} \left[ V^{\alpha}(X_{1}) \,\Big\vert\, X_{0}=x, A_{0} = \tilde{\alpha}(x) \right] \leq \\
&\myeq{(\ref{eq:hp_policy_improvement_amfc})}{\leq} f(x,\tilde{\alpha}(x),\mu^{\tilde{\alpha}}) + \gamma \mathbb{E} \left[ Q^{\alpha}(X_{1},  \tilde{\alpha}(X_{1})) \,\Big\vert\, X_{0}=x, A_{0} = \tilde{\alpha}(x) \right] = \\
&\myeq{(\ref{eq:bellman_q_pi})}{=} f(x,\tilde{\alpha}(x),\mu^{\tilde{\alpha}}) + \gamma \mathbb{E} \left[  f(X_{1},\tilde{\alpha}(X_{1}),\mu^{\tilde{\alpha}}) + \gamma Q^{\alpha}(X_{t_2},  \alpha(X_{t_2})) \,\Big\vert\, X_{0}=x, A_{0} = \tilde{\alpha}(x) \right] \leq \\
\vdots \\
&\leq \mathbb{E} \left[ \sum_{n=0}^{k} \gamma^n f(X_{n}, \tilde{\alpha}(X_{n}),\mu^{\tilde{\alpha}}) + \gamma^{k+1}V^{\alpha}(X_{k+1}) \,\Big\vert\, X_{0}=x \right]
\end{aligned}
\end{equation*}
Considering the limit as $k\rightarrow \infty$, it follows that 
\begin{equation*}
V^{\alpha}(x) <\mathbb{E} \left[ \sum_{n=0}^{\infty} \gamma^n f(X_{n}, \tilde{\alpha}(X_{n}),\mu^{\tilde{\alpha}}) \,\Big\vert\, X_{0}=x \right] = V^{\tilde{\alpha}}(x)
\end{equation*}

\noindent\textbf{Step 2} Show that $V^{\alpha}(x') < V^{\tilde{\alpha}}(x') \quad \forall x' \in \mathcal{X}\setminus \{x\}$. 
 \vspace{10pt}\newline Let define $\tau_x = \min \{ n : X_{n} = x \}$. Then
\begin{equation*}
\begin{aligned}
V^{\alpha}(x') &= \mathbb{E} \left[ \sum_{n=0}^{\infty} \gamma^n f(X_{n},\alpha(X_{n}),\mu^{\alpha} )\,\Big\vert\, X_{0} = x'\right] = \\
&= \mathbb{E} \left[ \sum_{n=0}^{\tau_x-1} \gamma^n f(X_{n},\alpha(X_{n}),\mu^{\alpha} )+ \sum_{n=\tau_x}^{\infty} \gamma^n f(X_{n},\alpha(X_{n}),\mu^{\alpha})\,\Big\vert\, X_{0} = x'\right] = \\
&= \mathbb{E} \left[ \sum_{n=0}^{\tau_x-1} \gamma^n f(X_{n},\alpha(X_{n}),\mu^{\alpha} )\,\Big\vert\, X_{0} = x' \right]+\mathbb{E} \left[ \sum_{n=\tau_x}^{\infty} \gamma^n f(X_{n},\alpha(X_{n}),\mu^{\alpha})\,\Big\vert\, X_{0} = x'\right] = \\
&\vcentcolon= T_1 + T_2
\end{aligned}
\end{equation*}
We start analyzing the first term observing that $X_{n} \neq x$ and $\alpha(X_{n}) = \tilde{\alpha}(X_{n})$ for all $n<=\tau_x - 1$. Then,
\begin{equation*}
T_1 = \mathbb{E} \left[ \sum_{n=0}^{\tau_x-1} \gamma^n f(X_{n},\tilde{\alpha}(X_{n}),\mu^{\tilde{\alpha}} )\,\Big\vert\, X_{0} = x' \right]
\end{equation*}
The analyses of the term $T_2$ is based on the tower property (TP), the Markov property (MP) and Step 1 (S1). It follows that 
\begin{equation*}
\begin{aligned}
T_2 &\myeq{\tiny{(TP)}}{=} \mathbb{E}\left[ \mathbb{E}\left[  \sum_{n=\tau_x}^{\infty} \gamma^n f(X_{n},\alpha(X_{n}),\mu^{\alpha}) \,\Big\vert\, X_{0}=x', X_{1}, \dots, X_{\tau_x} \right] \,\Big\vert\,X_{0} = x' \right] = \\
&\myeq{\tiny{(MP)}}{=} \mathbb{E}\left[ \gamma^{\tau_x}\mathbb{E}\left[  \sum_{n=\tau_x}^{\infty} \gamma^{n-\tau_x} f(X_{n},\alpha(X_{n}),\mu^{\alpha}) \,\Big\vert\,  X_{\tau_x} \right] \,\Big\vert\,X_{0} = x' \right] = \\
&= \mathbb{E}\left[ \gamma^{\tau_x} V^{\alpha}(X_{\tau_x}) \,\Big\vert\, X_{0} = x' \right] <\\
&\myeq{\tiny{(S1)}}{<}  \mathbb{E}\left[ \gamma^{\tau_x} V^{\tilde{\alpha}}(X_{\tau_x}) \,\Big\vert\, X_{0} = x' \right] 
\end{aligned}
\end{equation*}
Combining the analyses of $T_1$ and $T_2$, it follows that
\begin{equation*}
\begin{aligned}
V^{\alpha}(x') &= T_1 + T_2 <\\
&< \mathbb{E} \left[ \sum_{n=0}^{\tau_x-1} \gamma^n f(X_{n},\tilde{\alpha}(X_{n}),\mu^{\tilde{\alpha}} )\,\Big\vert\, X_{0} = x' \right]+\mathbb{E}\left[ \gamma^{\tau_x} V^{\tilde{\alpha}}(X_{\tau_x}) \,\Big\vert\, X_{0} = x' \right]\\
&=  \mathbb{E} \left[ \sum_{n=0}^{\tau_x-1} \gamma^n f(X_{n},\tilde{\alpha}(X_{n}),\mu^{\tilde{\alpha}} )+\gamma^{\tau_x}\sum_{n=\tau_x}^{\infty} \gamma^{n-\tau_x} f(X_{n},\tilde{\alpha}(X_{n}),\mu^{\tilde{\alpha}})\,\Big\vert\, X_{0} = x' \right]=\\
& =\mathbb{E} \left[ \sum_{n=0}^{\infty} \gamma^n f(X_{n},\tilde{\alpha}(X_{n}),\mu^{\tilde{\alpha}} )\,\Big\vert\, X_{0} = x' \right]=\\
&=V^{\tilde{\alpha}}(x')
\end{aligned}
\end{equation*}
\end{proof}

\begin{proof}[Theorem \ref{thm:property_3}]
Let $\mathcal{X}=\{ x_1, \dots, x_n\}$ and $\mathcal{A}=\{ a_0, \dots, a_m\}$ be the state and action spaces. \newline \textbf{Step 1} Let $\alpha^0$ be an initial policy and define $\alpha^1$ as follows
\begin{equation*}
\alpha^1(x) = 
\begin{cases}
\arg\max_a Q^{\alpha^0}(x,a), \quad &\text{ if } x = x_1,\\
\alpha_0(x), \quad &\text{ o.w. }
\end{cases} 
\end{equation*}
Then, 
\begin{equation*}Q^{\alpha^0}(x_1,\alpha^1(x_1)) \geq V^{\alpha^0}(x_1) \myeq{(\ref{eq:policy_improvement_amfc})}{\implies} V^{\alpha^1}(x) \geq V^{\alpha^0}(x),  \quad \forall x
\end{equation*}
\textbf{Step 2} 
Consider $\alpha^2$ defined as follows
\begin{equation*}
\begin{aligned}
\alpha^2(x) &= 
\begin{cases}
\arg\max_a Q^{\alpha^1}(x,a), \quad &\text{ if } x = x_2,\\
\alpha_1(x), \quad &\text{ o.w. }  
\end{cases} \\
&=\begin{cases}
\arg\max_a Q^{\alpha^1}(x,a), \quad &\text{ if } x = x_2,\\
\arg\max_a Q^{\alpha^0}(x,a), \quad &\text{ if } x = x_1,\\
\alpha_0(x), \quad &\text{ o.w. } 
\end{cases}
\end{aligned}
\end{equation*}
Then, 
\begin{equation*}Q^{\alpha^1}(x_2,\alpha^2(x_2)) \geq V^{\alpha^1}(x_1) \myeq{(\ref{eq:policy_improvement_amfc})}{\implies} V^{\alpha^2}(x) \geq V^{\alpha^1}(x) \geq V^{\alpha^0}(x),  \quad \forall x
\end{equation*}
\textbf{Step $\boldsymbol{n}$} 
Consider $\alpha^{n}$ defined as follows
\begin{equation*}
\begin{aligned}
\alpha^n(x) &= 
\begin{cases}
\arg\max_a Q^{\alpha^{n-1}}(x,a), \quad &\text{ if } x = x_n,\\
\alpha_{n-1}(x), \quad &\text{ o.w. }  
\end{cases} \\
&=\arg\max_a Q^{\alpha^{k-1}}(x,a), \quad\quad\quad \text{ if } x = x_k, \text{ for } k =1,\dots, n ,
\end{aligned}
\end{equation*}
Then, 
\begin{equation*}Q^{\alpha^{n-1}}(x_n,\alpha^n(x_n)) \geq V^{\alpha^{n-1}}(x_n) \myeq{(\ref{eq:policy_improvement_amfc})}{\implies} V^{\alpha^n}(x) \geq V^{\alpha^{n-1}}(x) \geq V^{\alpha^0}(x),  \quad \forall x
\end{equation*}
\textbf{Step $\boldsymbol{N}$} Since the state and action spaces are finite, the policy can be improved only a finite number of times. In other words, $\exists N>0$ such that
\begin{equation*}
\alpha^{N}(x) = \arg\max_a Q^{\alpha^N}(x,a), \quad \forall x \in \mathcal{X}
\end{equation*}
and 
\begin{equation*}
V^{\alpha^N}(x) = Q^{\alpha^N}(x,\alpha^N(x)) = \max_a Q^{\alpha^N}(x,a), \quad \forall x \in \mathcal{X}.
\end{equation*}
Can $\alpha^N$ be still suboptimal? No,  by extending Bellman and Dreyfus's Optimality Theorem (1962), \cite{bellman2015applied}.
\end{proof}

\begin{proof}[Theorem (\ref{thm: bellman A-MFC})]
\begin{equation*}
\begin{aligned}
\text{RHS} &= f(x,a, \mu^{\tilde{\alpha}}) + \gamma \mathbb{E} \left[\max_{a^{\prime}} Q^*(X_{1},a^{\prime})\,\Big\vert\, X_{0} = x, A_{0} = a\right] = \\
&\myeq{(\ref{eq:property_3})}{=}f(x,a, \mu^{\tilde{\alpha}}) + \gamma \mathbb{E} \left[V^*(X_{1})\,\Big\vert\, X_{0} = x, A_{0} = a\right]\\
&\myeq{(\ref{eq:property_1})}{=}f(x,a, \mu^{\tilde{\alpha}}) + \gamma \mathbb{E} \left[Q^{\alpha^*}(X_{1},\alpha^*(X_{1}))\,\Big\vert\, X_{0} = x, A_{0} = a\right]=\\
&\myeq{(\ref{eq:bellman_q_pi})}{=} Q^{\alpha^*}(x,a) =  Q^*(x,a),
\end{aligned}
\end{equation*}
where the last step is due to what shown in the proof of equation (\ref{eq:property_3}), i.e. the same policy $\alpha^*$ optimizes  $V^{\alpha}$ and $Q^{\alpha}$.
\end{proof}

\end{document}